\documentclass[12pt]{article}
\usepackage[all]{xy}

\usepackage{amssymb,latexsym,amsmath,amsthm,amsfonts,bm}
\usepackage{mathrsfs} 
\usepackage{graphics}

\newtheorem{theorem}{Theorem}[subsection]
\newtheorem{proposition}[theorem]{Proposition}
\newtheorem{corollary}[theorem]{Corollary}

\newtheorem{lemma}[theorem]{Lemma}

\theoremstyle{remark}
\newtheorem{remark}[theorem]{Remark}

\theoremstyle{definition}
\newtheorem{definition}[theorem]{Definition}
\newtheorem{fact}[theorem]{Fact}
\newtheorem{assumption}[theorem]{Assumption}
\newtheorem{notation}[theorem]{Notation}

\numberwithin{equation}{subsection}
\DeclareMathOperator{\Aut}{Aut}
\DeclareMathOperator{\I}{\mathbb I}

\DeclareMathOperator{\id}{id}

\newcommand{\B}{{\rm B}}
\newcommand{\C}{\mathbb C}

\newcommand{\Ha}{\mathbb H}

\renewcommand{\star}{*}

\newcommand{\M}{\mathcal M}

\newcommand{\pglM}{\hat{\M}}
\newcommand{\slM}{\check{\M}}

\newcommand{\Z}{\mathbb Z}

\newcommand{\D}{\mathbb D}

\newcommand{\mmu}{{\mu}}

\DeclareMathOperator{\tr}{tr}

.tfm
\font\Bbb = msbm7.tfm
\def\beq{\begin{eqnarray}}
\def\eeq{\end{eqnarray}}
\def\bes{\begin{eqnarray*}}
\def\ees{\end{eqnarray*}}

\def\C{\mathbb{C}}

\def\P{\mathbb{P}}

\def\Su{\mathbb{S}}

\def\JJ{\mathbb{J}}
\def\Z{\mathbb{Z}}

\newcommand{\bino}[2]{\mbox{ $#1 \choose #2$}}
\newcommand{\nc}{\newcommand}
\nc{\op}[1]{\mathop{\mathchoice{\mbox{\rm #1}}{\mbox{\rm #1}}
{\mbox{\rm \scriptsize #1}}{\mbox{\rm \tiny #1}}}\nolimits}
\nc{\ep}{\varepsilon}
\nc{\Ga}{\Gamma}
\nc{\La}{\Lambda}
\nc{\si}{\sigma}
\nc{\Sig}{{\Gamma}}
\nc{\Om}{\Omega}
\nc{\om}{\omega}
\nc{\SL}{{\rm SL}}
\nc{\GL}{{\rm GL}}
\nc{\PGL}{{\rm PGL}}
\nc{\G}{{\rm G}}

\nc{\Pic}{{\op{Pic}}}
\nc{\Jac}{{\op{Pic}}}
\nc{\Prym}{{\op{Prym}}}
\nc{\cpt}{{\op{cpt}}}
\nc{\Dol}{{\op{Dol}}}
\nc{\DR}{{\op{DR}}}
\nc{\Triv}{\op{Triv}}
\nc{\Hod}{{\op{Hod}}}
\nc{\Est}{E_{\op{st}}}
\nc{\Hst}{H_{\op{st}}}
\nc{\Left}[1]{\hbox{$\left#1\vbox to
   10.5pt{}\right.\nulldelimiterspace=0pt \mathsurround=0pt$}}
\nc{\Right}[1]{\hbox{$\left.\vbox to
   10.5pt{}\right#1\nulldelimiterspace=0pt \mathsurround=0pt$}}
\nc{\LEFT}[1]{\hbox{$\left#1\vbox to
   15.5pt{}\right.\nulldelimiterspace=0pt \mathsurround=0pt$}}
\nc{\RIGHT}[1]{\hbox{$\left.\vbox to
   15.5pt{}\right#1\nulldelimiterspace=0pt \mathsurround=0pt$}}
\nc{\bee}{{\bf E}}
\nc{\bphi}{{\bf \Phi}}

\newcommand{\hi}{h}
\newcommand{\ff}{f}

\newcommand{\mi}{M}
\newcommand{\ai}{A}
\newcommand{\dai}{a}

\newcommand{\mdpgld}{\hat{\mathcal M}_{\Dol}}
\newcommand{\mdsld}{\check{\mathcal M}_{\Dol}}

\newcommand{\higgsbu}{ {\mathcal M}  }
\newcommand{\higgsbud}{ {\mathcal M}_{\Dol}  }  
\newcommand{\higgsbusm}{ {\mathcal M}_{{\rm reg}}  }  
\newcommand{\higgsbugood} { {\mathcal M}_{{\rm ell}}  }  
\newcommand{\hitmap}{\chi}

\newcommand{\hitmaph}{\hat\chi}
\newcommand{\hitmapc}{\check\chi}

\newcommand{\mdpgl}{\hat{\mathcal M}}
\newcommand{\mdsl}{\check{\mathcal M}}
\newcommand{\mdslsm}{\check{\mathcal M}_{{\rm reg}}}
\newcommand{\mdpglsm}{\hat{\mathcal M}_{{\rm reg}}}
\newcommand{\mdgl}{{\mathcal M}}
\newcommand{\extal}[1]{\bigwedge^{\bullet}{#1}}

\newcommand{\slt}{\mathfrak{sl}_2(\rat)}

\def\lorw{\longrightarrow}
\newcommand\n{\noindent}

\newcommand\ci{\cite}

\newcommand\rat{{\mathbb Q}}
\newcommand\comp{{\mathbb C}}
\newcommand\real{{\mathbb R}}
\newcommand\zed{{\mathbb Z}}
\newcommand\nat{{\mathbb N}}
\newcommand\pn[1]{{\Bbb P}^{#1}}

\renewcommand\blacksquare{{\hspace*{\fill} $\fbox{}$}}

\newcommand{\phix}[2]{ \,^{\mathfrak p}\!{\mathcal H}^{#1}(#2)}

\newcommand{\im}{ \hbox{\rm Im} }

\newcommand{\ke}{ \hbox{\rm Ker} }

\newcommand{\ptd}[1]{ \,^{\mathfrak p}\!\tau_{ \leq {#1} } }

\newcommand{\ptu}[1]{ \,^{\mathfrak p}\!\tau_{ \geq {#1} } }

\def\lorw{\longrightarrow}
\renewcommand\n{\noindent}

\newcommand{\hitmapsmooth}{{\chi_{\rm reg}}}    %HITCHIN MAP RESTRICTED TO THE SMOOTH LOCUS

\newcommand{\curv}{C}                         % THE CURVE
\newcommand{\gencurv}{{\mathcal C}}               % A GENERIC CURVE
\newcommand{\totspa}{{\mathbb V}(D)}         % TOTAL SPACE OF D
\newcommand{\pro}{\pi}                    %THE PROJECTION FROM SPECTR CURVE
\newcommand{\jac}[1]{{\op {Pic}}_{#1}}                  % THE JACOBIAN
\newcommand{\prym}{{\op {Prym}}}           %The Prym variety
\newcommand{\compjac}[1]{\overline{{\op{Pic}}_{#1}}} % THE COMPACTIFIED JACOBIAN
\newcommand{\base}{{\mathcal A}}
\newcommand{\baseo}{{{\mathcal A}^0}}
\newcommand{\baseosm}{{{\mathcal A}_{\rm reg}^0}}
\newcommand{\basesm}{{\base}_{\rm reg}}
\newcommand{\basegood}{{\base}_{\rm ell}}
\newcommand{\baseogood}{{\base}^0_{\rm ell}}

\newcommand{\rjmap}{p}
\newcommand{\reljac}{{{\mathscr P }}_{\rm reg}}
\newcommand{\spectfam}{u}                 % SPECTRAL FAMILY 
\newcommand{\spectfamsmooth}{{u_{\rm reg}}} % SPECTRAL FAMILY ON THE SMOOTH LOCUS
\newcommand{\shf}{{\mathcal F}}     %A TORSION FREE SHEAF OF RANK 1 
\newcommand{\dist}{{\mathcal N}}    % A DISTINGUISHED NEIGHBORHOOD

\newcommand{\nop}{{\omega}}
\newcommand{\nep}{{\epsilon}}
\newcommand{\pos}{{\mathfrak S}}
\newcommand{\cl}[2]{\,^{#1 \!}\lambda_{#2  }} %LAMBDA CYCLES ON THE $l$-TH HANDLE

\begin{document} 
\title{Topology of Hitchin systems and \\ Hodge theory of character varieties:\\ the case $A_1$}
\author{ Mark Andrea A.  de Cataldo\thanks{ Partially supported by N.S.A. and N.S.F.} \\{\it  Stony Brook University}\\ {\tt mde@math.sunysb.edu} \and Tam\'as Hausel\thanks{Supported by a Royal Society University Research Fellowship}
\\ {\it University of Oxford}
\\{\tt hausel@maths.ox.ac.uk}\and Luca Migliorini\thanks{Partially supported by PRIN 2007 project ``Spazi di moduli e teoria di Lie''} \\{\it Universit\`a di Bologna} \\{\tt migliori@dm.unibo.it} \\
}

		\maketitle
\begin{abstract} For $\G=\GL_2,\PGL_2, \SL_2$ we prove that the perverse filtration associated with the Hitchin map on the rational  cohomology of the moduli space of twisted $\G$-Higgs  bundles
	on a compact Riemann surface $C$ agrees with the weight filtration on the rational cohomology of the twisted $\G$ character variety of $C$, when the cohomologies are identified via 
	non-Abelian Hodge theory. The proof is accomplished by means of a study of the topology of the Hitchin map over the locus of integral spectral curves. 
	\end{abstract} 
	\newpage
		\tableofcontents
	\newpage
	\section{Introduction}
\label{intro}
\stepcounter{subsection}
Starting with the paper of
Weil \ci{weil}, and its commentary by Grothendieck \ci{groth},
the moduli space of  holomorphic vector bundles on a projective curve has become the focus of much 
important work in mathematics, and there is now an extensive literature concerning its construction and properties.
As is well known, the construction of this moduli space  via geometric invariant theory is naturally  paired with
the notions of {\em stable} and {\em semistable} vector bundle.

A central result is the theorem of Narasimhan and 
Seshadri \ci{narasesha}, which asserts  that, 
roughly speaking, the semistable vector bundles of degree zero on a complex nonsingular projective curve $\curv$
(which we  assume to be of genus $g\geq 2$)
are precisely the ones  associated with unitary representations of the fundamental group
of $\curv$ or, if we consider bundles with non zero degree on
$\curv$, of the punctured curve $\curv \setminus p$, 
with a prescribed  scalar monodromy 
around the puncture. 
Let us  spell the theorem out  in the case of rank two bundles:
the fundamental group $\pi_1(\curv \setminus p)$ has free generators $\{a_1, \dots , a_g, b_1, \dots , b_g\}$ such that 
$a_1^{-1} b_1^{-1} a_1 b_1 \dots a_g^{-1} b_g^{-1} a_g b_g $  is the homotopy class of a loop around $p$. 
Unitary local systems of rank $2$  on $\curv \setminus p$ with local monodromy $-{\rm I}$ around $p$ are automatically 
irreducible, and the set of their isomorphism classes is
$$
{\mathcal N}_B:= \{ A_1,B_1,\dots, A_g,B_g \in {\rm U}(2)\ | \
		 A_1^{-1} B_1^{-1} A_1 B_1 \dots A_g^{-1} B_g^{-1} A_g B_g = - {\rm I} \}/{\rm U}(2),
$$
where the unitary group ${\rm U}(2)$ acts by conjugation on the matrices $A_i,B_i$; this action factors through a free action of 
${\rm PU}(2)$, hence the quotient ${\mathcal N}_B$
is a real analytic variety.
The theorem of Narasimhan and Seshadri
states that there is a canonical diffeomorphism
${\mathcal N}_B \simeq {\mathcal N}$, where ${\mathcal N}$ is the 
moduli space of stable rank two vector bundles  of degree one
on $\curv$.
 
\medskip   
A ``complexified''  version of this set-up, 
taking into consideration the analogue of the variety ${\mathcal N}_B$ 
obtained by  replacing the unitary group ${\rm U}(2)$ with its complexification 
$\GL_2(\comp)$ (or, more generally, with any complex reductive group $\G$, in 
which case the matrix  $-{\rm I}$ is replaced by a
suitable element in the center of $\G$), 
arose in the work of Hitchin \ci{hit, hitduke}.
Even though this paper considers  the variants of this construction for the 
complex algebraic groups groups $\SL_2(\comp)$ and $\PGL_2(\comp)$,
in this introduction we focus on the group $\GL_2(\comp)$;  
more details can be found in Section \ref{tamas}.

%%% The $\GL_2(\comp)$-valued representations of  $\pi_1(\curv \setminus p)$ 
The representations of  $\pi_1(\curv \setminus p)$ into $\GL_2(\comp)$
with monodromy $-{\rm I}$ around $p$ are automatically irreducible,
and their isomorphism classes are parametrized by the {\em twisted character variety}  
$$
\M_\B:= \left\{ A_1,B_1,\dots, A_g,B_g \in \GL_2(\comp)\ | \
		 A_1^{-1} B_1^{-1} A_1 B_1 \dots A_g^{-1} B_g^{-1} A_g B_g = - {\rm I} \right\}/\!/\GL_2(\comp),
$$ 
where the quotient is taken in the sense of geometric invariant theory. As in the unitary picture, the action of $\GL_2(\comp)$
factors through a free action of ${\rm PGL}_2(\comp)$, and $\M_\B$ is a nonsingular irreducible complex affine variety
of dimension $8g-6$.

The {\em non-Abelian Hodge theorem} states that,
just as in the Narasimhan-Seshadri correspondence, $\M_\B$ is naturally 
diffeomorphic to another quasi-projective variety, i.e. the {\em moduli space of semistable Higgs bundles} $\higgsbud$ 
 parametrizing stable pairs  $(E,\phi)$ consisting of a degree one 
 rank two vector bundle $E$ on $C$ together with a Higgs field $\phi \in H^0(C,{\rm End}\,( E )\otimes K_{\curv})$, subject to a natural  condition of stability. 
If $E$ itself is a stable vector bundle, then $\phi$ is in a natural way a cotangent vector at the point 
$[{\rm E}] \in {\mathcal N}$.  It follows that  $\higgsbud$ contains
the cotangent bundle of ${\mathcal N}$ as a Zariski open subset, which turns out to be dense.

The variety  $\higgsbud$ has a rich geometry: it has a natural hyperk\"ahler metric, 
an $S^1$-action by isometries and,  importantly,  it carries a projective map $\hitmap: \higgsbud \lorw \base$,
the {\em Hitchin fibration}, where  the target $\base$ is (non canonically) isomorphic to $\comp^{4g-3}$
%the affine space 
%$\base=H^0(\curv, K_{\curv})\oplus H^0(\curv, 2K_{\curv})$ 
and 
the fibre of $\hitmap $ over a  general point $s \in \base$ is isomorphic to
the Jacobian of  a branched double covering of $\curv$ associated with $s$, the {\em spectral curve} $\curv_s$. 
This description of $\higgsbud$ is usually referred to as {\em abelianization} since it 
reduces, to some extent,  the study of Higgs bundles on $\curv$ to that 
of line bundles on the spectral curves.

While the algebraic varieties  $\M_\B$ and $\higgsbud$ are diffeomorphic, they
 are not biholomorphic:
  the former is affine and  the latter is foliated by the fibers of the Hitchin map
 which are  compact $(4g-3)$-dimensional 
algebraic subvarieties, Lagrangian with respect to the natural holomorphic symplectic structure 
associated with the hyperk\"ahler  metric on $\higgsbud$.
 Furthermore, just as in the case of ${\mathcal N}_B$ and $\mathcal{N}$,  the variety $\M_\B$
does not depend on the complex structure of $\curv$, whereas $\higgsbud$  does.

It is  natural to investigate the relation between some of the invariants
of  $\M_\B$ and $\higgsbud$.
This  paper takes a  step in this direction.

The paper \ci{hausel-villegas} investigates in depth one of the important algebro-geometric invariants
of $\M_\B$, namely the mixed Hodge structure on its cohomology groups. %%%
 In view of \ci[Corollary 4.1.11]{hausel-villegas}, 
the mixed Hodge structure of $H^*(\M_\B)$ is of Hodge-Tate type: 
the quotient pure Hodge structures ${\rm Gr}^W_{i}$ satisfy
\begin{equation}
\label{uazzisHT}
{\rm Gr}^W_{2i+1}H^*(\M_\B)=0\, \, \mbox{for all }i, \mbox{ and }
{\rm Gr}^W_{2i}H^*(\M_\B) \mbox{ is of type }\,(i,i).
\end{equation}
 
The weight filtration $W_{\bullet}$ has a natural splitting, and it is nontrivial in certain cohomological degrees: 
for instance $H^4(\M_\B)$ contains classes of type $(2,2)$ and $(4,4)$. A remarkable property 
of  $W_{\bullet}$ is the ``curious hard Lefschetz theorem": 
there is a cohomology class $\tilde{\alpha} \in H^2(\M_\B)$, of type $(2,2)$,
such that the map given by  iterated cup products with  $\tilde{\alpha}$ defines isomorphisms:
\begin{equation} 
\label{curhardlef}
\cup \tilde{\alpha} ^{l}:{\rm Gr}^W_{8g-6-2l}H^*(\M_\B)\stackrel{\cong}{\lorw} {\rm Gr}^W_{8g-6+2l}H^{*+2l}(\M_\B).
\end{equation} 
Note that the class $\tilde{\alpha}$
raises the cohomological degree by two and the weight type by four, and that
 $ \M_\B$ is affine; both facts are in contrast with the hypotheses of the classical hard Lefschetz theorem,
 hence the "curiousity" of (\ref{curhardlef}).

On the other hand, the Hodge structure  on $H^*(\higgsbud)$ is pure, i.e.
its weight filtration $W_{\bullet}$ is trivial in every cohomological degree. The class $\tilde{\alpha} \in H^2(\higgsbud)$ has pure type $(1,1)$.
This raises the following question: what is the meaning of the weight filtration $W_{\bullet}$ of $H^*(\M_\B)$
when viewed in $H^*(\higgsbud)$ via the diffeomorphism $\M_\B \simeq \higgsbud$ coming from the 
non-Abelian Hodge theorem? The answer we give in this paper  brings into the picture the 
perverse Leray filtration $P$ of $H^*(\higgsbud)$   which is naturally associated with the  Hitchin map $\hitmap:
\higgsbud \to \base$.  

The perverse Leray  filtration 
has been implicitly introduced in \ci{bbd}, and it has been studied and employed  in \ci{htam,decmigseattle,pflht,dec1,dec2}.
This filtration is the abutment of  the {\em perverse Leray spectral sequence} which, in turn, 
is a variant    of the classical  Leray spectral sequence.
In the case of proper, but not necessarily smooth maps, e.g. our Hitchin map $\hitmap$, this variant   is  better behaved than the classical Leray one. In fact,    
it always degenerates at $E_2$, and the graded pieces of the abutted  perverse Leray filtration 
satisfy a %%% variant
version,  called the relative hard Lefschetz, of the hard  Lefschetz  theorem, 
involving   the operation of cupping with the first  Chern class of a line bundle 
which is relatively ample with respect to the proper map.
Both the Leray and the perverse Leray filtration originate from  filtrations of
the  derived direct image complex of sheaves $\hitmap_*\rat$ on  $\base$.

Since the target $\base$ of the map $\hitmap$ is  affine, the perverse filtration has the following  simple geometric 
characterization (see \ci{pflht}, where a different numbering convention is used).
Let $s\geq 0$ and let $\Lambda^s \subseteq \base$ be 
a general  $s$-dimensional linear section of $\base$ relative to a chosen identification of $\base$ with $\comp^{4g-3}$; then
\beq \label{restrict}
P_pH^d (\higgsbud)= \ke \,\{H^d (\higgsbud) \lorw H^d (\hitmap^{-1}(\Lambda^{d-p-1}))\}.
\eeq 

The main result of this paper is that, up to a trivial renumbering of the filtrations, 
{\em the weight filtration $W_{\bullet}$ on $H^*(\M_\B)$ coincides with the perverse Leray filtration
$P_{\bullet}$ on $H^*(\higgsbud)$}:

\begin{theorem}
\label{mava} {\rm (``P=W")}
In terms  of the isomorphism 
$H^*(\M_\B) \stackrel{\simeq}{\lorw} H^*(\higgsbud) $
induced by the diffeomorphism $\M_\B \stackrel{\simeq}{\lorw} \higgsbud$ stemming from the non-Abelian Hodge theorem, 
we have
$$
W_{2k}H^*(\M_\B)=W_{2k+1}H^*(\M_\B)=P_k H^*(\higgsbud).
$$
\end{theorem}
Since
 the class $\tilde{\alpha} \in H^*(\higgsbud)$ is relatively ample with respect to   the Hitchin map, 
the curious hard Lefschetz theorem for $\tilde{\alpha}$ on $(H^*(\M_\B), W)$ can
thus  be explained in terms
of the relative hard Lefschetz theorem for $\tilde{\alpha}$ on  $(H^*(\higgsbud),P)$.

One may say informally that the weight filtration on  $H^*(\M_\B)$ keeps track of  certain topological properties of the Hitchin map on $\higgsbud$.
This is even more remarkable in view of the fact that the structure of algebraic variety
on $\M_\B$, and thus the shape of $W_{\bullet}$, depend only on the topology of the curve $\curv$, while
the complex/algebraic structure of the Higgs moduli space $\higgsbud$ and thus the Hitchin map
depend on the complex structure of $\curv$.

\medskip
In fact, as far as $P=W$ goes, 
we prove a more precise result. 
There are natural splittings (constructed by Deligne in \ci{deligneseattle})
of the perverse Leray filtration of $H^*(\higgsbud)$. The splittings 
induced on  $H^*(\higgsbud)$  are equal, and they  
coincide with the  splitting mentioned above of the 
weight filtration of $H^*(\M_\B)$. We also prove that these results hold  for the varieties associated with
$\SL_2(\comp)$ and $\PGL_2(\comp)$.
In \S \ref{dbiggerchi} we also prove   a version of the main theorem ``$P= W$"  for the  moduli spaces of Higgs bundles with poles
on $\curv$, namely when the canonical bundle is replaced by a different line bundle of high enough degree.
In this case, there is no character variety $\M_\B$ to be compared with $\higgsbud$. However, 
the cohomology ring $H^*(\higgsbud)$ admits yet another filtration
 which is quite visible in terms of generators and relations.
 We prove that this third filtration  coincides with the perverse Leray filtration associated with the Hitchin map (which is also defined in the context of poles). In the case where there are no poles,
 this third filtration coincides,  after a  simple renumbering, with the weight filtration.

Finally,  as we need it in the course of our proof of $P=W$ in the case when $\G=\SL_2$, in Remark~\ref{description} we give a description of the cohomology ring for $\G=\SL_2$ which does not %%%seem to 
appear in the literature. 
This ring had been earlier determined by M.Thaddeus in unpublished work.

\bigskip
Since the proof of our main result is lengthy, 
we  sketch below the main steps leading to it. Of course, 
for the sake of clarity, we  do so
by overlooking 
many technical issues.

\bigskip
The  ring  structure of $H^*(\higgsbud)$ is known in terms of generators and relations;
see \ci{hausel-thaddeus1,hausel-thaddeus2}.
By using a result of  M.Thaddeus', we prove   that 
the  place of the multiplicative generators  in the perverse
Leray filtration of $H^*(\higgsbud)$ is the same as in the weight filtration
of $H^*(\M_\B)$ (Theorem \ref{eccolo}). 
If the perverse Leray filtration were compatible
with cup products, then we 
could infer the same conclusion for the other cohomology classes.

However only the weaker  compatibility 
$$P_iH(\higgsbud) \cup P_jH(\higgsbud) \longrightarrow  P_{i+j+d}H(\higgsbud)$$ 
holds a priori for the perverse Leray filtration (see Proposition \ref{mult}, and \ci[Theorem 6.1]{decberkeley}), 
where $d$ is the relative dimension of the map $\hitmap$.
In contrast,  the compatibility in the strong form holds for the ordinary Leray filtration.

The Leray filtration is contained in the perverse Leray filtration. 
At the level of  the direct image complex, the two filtrations
coincide on the open subset  of regular values on the target of the map.

One key to our approach is that we prove that, for the Hitchin map, 
there is  a significantly larger open set of $\base$ where the Leray and the perverse Leray
 filtration coincide
on the direct image complex.
We define the ``elliptic'' locus $\basegood \subseteq \base$ to be the subset of 
points $s \in \base$
for which the corresponding  spectral curve $\curv_s$ is integral.  Let $\basesm \subseteq \basegood$ be the set of regular values for the Hitchin map $\hitmap$;
the corresponding spectral curves are irreducible nonsingular. 
The key result is then the following, which  we believe to be of independent interest:

\begin{theorem}
\label{rcmd}
Let $j: \basesm \lorw \basegood$ be  the inclusion and, for $l \geq 0$, 
let  $R^l$ denote the local system 
$s \mapsto H^l(\hitmap^{-1}(s))$ on $\basesm$.
Then, there is an isomorphism in the derived category of  sheaves on $\basegood$:
\[ (\hitmap_* \rat )_{| \hitmap^{-1}(\basegood) }\simeq \bigoplus_l R^0j_*\,R^l\,[-l].
\]
\end{theorem}

This theorem  contains two distinct statements:

\begin{enumerate}
\item
\label{quifunza}
the perverse sheaves  on $\basegood$ appearing in the decomposition theorem  (\ci{bbd})
for the Hitchin map restricted over the open set $\basegood$
are  supported on the whole $\basegood$; 
This is a special case of Ng\^o's support theorem 7.1.13 in \ci{ngo}, 
which holds for the Hitchin fibration of any group $G$;

\item
\label{pureli}
these perverse sheaves, which are the intersection cohomology complexes 
on $\basegood$ of the local systems $R^l$ on the smooth locus $\basesm$,
are  ordinary sheaves, as opposed to complexes; up to a shift,
they agree with the  higher direct images appearing in the Leray spectral sequence.

\end{enumerate}

The theorem implies   that the classical and the perverse Leray filtrations
   coincide on $\basegood$.
This puts us  in a position to compute the ``perversity'' of 
most monomials generators of $H^*(\higgsbud)$; 
see Lemma \ref{betapsieasy}.
As explained above in \eqref{restrict}, the perversity of a class is tested
by restricting it to the inverse image of  generic linear sections of $\base$.
The algebraic subset $\base \setminus \basegood$ is of  high codimension in $\base$. 
It follows that, in a certain range of dimensions,  
the general linear section can  be chosen  to lie entirely  in $\basegood$, where we know, by  Theorem \ref{rcmd},
that the perverse Leray  filtration 
is compatible with the cup product since it coincides with the Leray filtration.

At this point, we can conclude in the case of Higgs bundles for poles;
see section \ref{dbiggerchi}. In  the geometrically 
more significant case where there are  no poles,
some monomials are not covered by the above line of reasoning,
for the corresponding linear sections must meet $\base \setminus \basegood$ by simple reasons of dimension.
We treat these remaining classes using an ad hoc argument 
based on the properties of the Deligne splitting mentioned above; see Section \ref{dequalschi}.

In order to prove  Theorem \ref{rcmd}  above
we first determine an  upper bound  (see Theorem \ref{estbncj}) for the Betti numbers of the fibres of the Hitchin map  over
$\basegood$.  In the case when  $\G=\GL_2$,
these fibres are  the compactified Jacobians of the spectral curves, 
which, being double coverings of a nonsingular curve, have
singularities analytically isomorphic to $y^2-x^k=0$,
%{\bf The hypothesis on $\G$ {\bf which one?} is heavily used, for   the spectral curves have singularities of type $A_k$ {\bf paper's titled $A_1$...},
a fact we use in an essential way in our computations.
Next, we give a lower bound (see Theorem \ref{mainmonthm})
 for the dimension of the  stalks of the intersection cohomology complexes.   This bound is  
based  on the computation of the local monodromy
of the family of nonsingular spectral curves
 around a singular  integral spectral curve.
 It is achieved by a repeated use of the Picard-Lefschetz formula. 
Since the upper and lower bounds coincide, the 
 decomposition theorem (\ci{bbd}) gives the wanted result.

We see at least two
difficulties to extend the results in this paper for complex reductive groups of higher rank:
the monodromy computation of Theorem \ref{mainest} 
which leads to the proof of  Theorem  \ref{mainmonthm}  would be 
more complicated  and we do not know enough about   
compactified Jacobians of curves with singularities which 
are not double points.  Already for the group $\GL_3$,  \ref{pureli}. above fails, and the intersection cohomology complexes 
are not  shifed sheaves.

On the other hand, a curious hard Lefschetz theorem is conjectured in \cite[Conjecture 4.2.7]{hausel-villegas}
to hold for the character variety for $\PGL_n$ which would of course follow, if $P=W$, from 
the relative  hard Lefschetz theorem. Additionally, in a recent work of physicists 
Chuang-Diaconescu-Pan \ci{diaconescu-etal} a certain refined Gopakumar-Vafa conjecture for local curves in a Calabi-Yau $3$-fold  leads to a conjecture on the dimension of the graded pieces of the perverse filtration on the cohomology of the moduli space of twisted $\GL_n$-Higgs bundles on $C$. Their conjecture agrees with
the conjectured \cite[Conjecture 4.2.1]{hausel-villegas} dimension of the graded pieces of the weight filtration on the cohomology of the twisted $\GL_n$-character variety. The compatibility of these two conjectures maybe considered the strongest indication so far that
$P=W$ should hold for higher rank Higgs bundles as well.

In the paper \ci{dehami} we prove that a result
analogous to our main theorem $P=W$ holds in a situation which is expected to be closely related to the moduli space of certain parabolic Higgs bundles
of rank n on a genus one curve. 
Interestingly, in this case, the coincidence of the two filtrations
holds, whereas  the result \ref{quifunza}. above, concerning the supports of the perverse sheaves being
maximal on a large open set,  fails, due to the fact that every 
new stratum contributes a new direct summand sheaf. 

While property \ref{pureli}. above seems to hold only  for Hitchin fibrations associated with groups of type $A_1$, the case studied in the present paper, and  property \ref{quifunza}.  may not  hold  for parabolic Higgs bundles, we expect that the $P=W$ phenomenon should
be  a general feature of non-Abelian Hodge theory for curves.  More generally, in \ci{dehami}, \S 4.4, we also conjecture that this exchange of filtration phenomenon should hold 
 for holomorphic symplectic varieties with a $\comp^*$-action, that, roughly speaking, behave like an algebraically completely integrable system.

\bigskip
\bigskip
\n
{\bf Acknowledgements.}

\smallskip
\n
The authors would like to thank D.E. Diaconescu, N. Hitchin, M.N. Kumar, G. Laumon, E. Markman,  B.C. Ng\^o and A. Vistoli
for useful conversations and suggestions. We thank the anonymous referees for the careful reading and for s providing very useful suggestions
which substantially improved the paper.

The first-named author dedicates this paper to his family:
Mikki, Caterina, Amelie and Dylan. The second-named author dedicates this paper to his mother Hausel P\'aln\'e D\'egi Ir\'en.  The third-named author dedicates this paper to Fiamma, Enrico and Gaia.
		
	\begin{subsection}{Cohomology of moduli spaces}
		\label{tamas}
	\begin{subsubsection}{Character variety}
	\label{charactvrty}
	
In this section, we recall some definitions and results from \cite{hausel-villegas}. 
Throughout the paper,  the singular  homology and cohomology groups
are taken with  rational coefficients.

		Let $\Sigma$ be a closed Riemann surface of genus $g\geq 2$ and let $\G$ be a complex reductive group. In this paper, we consider only the cases 
$\G= \GL_2, \PGL_2$ and $\SL_2$. 
		We are interested
		in the variety parameterizing  certain twisted 
		representations of the fundamental group $\pi_1(\Sigma)$ into $\G$ modulo isomorphism. Specifically,
		we consider the {\em $\GL_2$-character variety}:
		\bes{\M}_\B :=\{ A_1,B_1,\dots, A_g,B_g \in \GL_2\ |\ 
		 A_1^{-1} B_1^{-1} A_1 B_1 \dots A_g^{-1} B_g^{-1} A_g B_g = - {\rm I} \}/\! /\PGL_2,
		\ees 
		i.e. the affine GIT quotient by the diagonal adjoint action of 
$\PGL_2$  on the matrices $A_i,B_i$. We also define the {\em $\SL_2$-character variety}: 
		\bes{{\slM}}_\B :=\{ A_1,B_1,\dots, A_g,B_g \in {\SL}_2\ | \
		 A_1^{-1} B_1^{-1} A_1 B_1 \dots A_g^{-1} B_g^{-1} A_g B_g = - {\rm I} \}/\!/{\rm PGL}_2.
		\ees The torus $\GL_1^{2g}$ acts on $\GL_2^{2g}$ by coordinate-wise multiplication
		and this yields  an action of $\GL_1^{2g}$ on $\M_\B$. Similarly, the finite subgroup of order $2$
		 elements $\mmu_2^{2g}\subset \GL_1^{2g}$, with $\mmu_2:=\{\pm 1\}\subset \GL_1$, acts on $\SL_2^{2g}$ by coordinate-wise multiplication and we  define the {\em $\PGL_2$-character variety} as:
		\begin{equation}\label{quotient}{\pglM}_\B :={\M}_{\B}/\!/ \GL_1^{2g}=\slM_\B/\mmu_2^{2g}.\end{equation}  The
		surjective group homomorphism $\SL_2\times \GL_1\to \GL_2$ 
		with finite kernel $\mmu_2$ induces a covering \begin{equation}\slM_\B\times \GL_1^{2g}\to \M_\B\label{covering}\end{equation} with covering group $\mmu_2^{2g}$.
		
		  The varieties  $\M_\B$, $\slM_\B$ are non-singular  and affine (cf.
		  \cite[\S 2.2]{hausel-villegas}), whereas 
		   $\pglM_\B$ is  affine with finite quotient singularities, and parameterizes the representations of $\pi_1(\Sigma)$ to $\PGL_2$ 
		   which do not admit a lift to representations of $\SL_2$.  We have  $\dim\, \M_\B=8g-6$ and $\dim\, \slM_\B =\dim\, \pglM_\B =6g-6$.  In view of \eqref{quotient}, we have that $$H^*(\pglM_\B)=H^*(\slM_\B)^{\mmu_2^{2g}},$$ 
		 the subring of $\mmu_2^{2g}$ invariants,  while \eqref{covering} implies that  
		   \begin{equation} 
		   \label{gl2-split}
		   H^*(\M_\B)=H^*(\GL_1^{2g})\otimes H^*(\slM_\B)^{\mmu_2^{2g}}=H^*(\GL_1^{2g})\otimes H^*(\pglM_\B).
		   \end{equation}
		
		 The cohomology ring $H^*(\M_\B)$ is generated by certain 
		universal classes $\epsilon_i\in H^1(\GL_1^{2g})\subset H^1(\M_\B)$ for $i=1,\dots,2g$, $\alpha\in H^2(\slM_\B)^{\mmu_2^{2g}}\subset H^2(\M_\B)$, $\psi_i \in H^3(\slM_\B)^{\mmu_2^{2g}}\subset H^3(\M_\B)$ for $i=1,\dots,2g$, and $\beta\in H^4(\slM_\B)^{\mmu_2^{2g}}\subset H^4(\M_\B)$.
		The proof can be found in \cite{hausel-thaddeus1} (generators) and in its sequel
		\cite{hausel-thaddeus2} (relations).  The construction of these
	 universal classes is explained in  \cite[\S 4.1]{hausel-villegas}.
		 The paper \cite{hausel-villegas} used this
		information to determine the mixed Hodge structure on $H^*(\M_\B)$. 
For use in this paper, we summarize these  results  as follows. 
Let $(H,W_{\bullet}, F^{\bullet})$ be a mixed Hodge structure (see the textbook \cite{petstee} for a comprehensive treatment of mixed Hodge theory).
%We recall that a mixed Hodge structure on a rational vector space 
%$H$ consist of a finite  increasing filtration $W_{\bullet}$ (the weight filtration) on $H$, and a finite decreasing filtration $F^{\bullet}$ (the Hodge filtration) on the complexification $H_%{\comp} $, which induces a $(p,q)$-decomposition 
%of weight $k$ on the complexified graded pieces 
%${\rm Gr}^W_kH_{\comp}=  
%\left(W_kH/W_{k-1}H\right)_{\comp}$, namely:
%$$
%{\rm Gr}^W_kH_{\comp}=\bigoplus_{p+q=k}
%\left({\rm Gr}^W_kH_{\comp}\right)^{p \, q},
%$$
%where
%$$
%\left({\rm Gr}^W_kH_{\comp}\right)^{p \, q}:=
%F^p{\rm Gr}^W_kH_{\comp}\cap \overline{F^q{\rm Gr}^W_kH_{\comp}}.
%$$

A class $\sigma \in H$ is said to be of {\em homogeneous weight  } $k$ 
(\cite[Definition 4.1.6]{hausel-villegas}) if 
its image in $H_{\comp}$, still denoted by $\sigma$, satisfies
\begin{equation}
\label{hwk}
\sigma \in W_{2k} H_{\comp}\cap F^k H_{\comp}.
\end{equation}
Note that if $\sigma$ has {\em homogeneous weight  } $k$, then its image in 
${\rm Gr}^W_{2k}H_{\comp}$ is of type $(k,k)$. 

\bigskip
The natural mixed Hodge structure on $H^{i}(\M_\B)$  satisfies 
$W_kH^{i}(\M_\B)=H^{i}(\M_\B)$ for $k\geq 2i$, and,
as $\M_\B$ is nonsingular,  $W_kH^{i}(\M_\B)=0$ for $k\leq i-1$. The 
following is proved in \cite[Theorem 4.1.8]{hausel-villegas}:
		
\begin{theorem} 
The cohomology classes $\epsilon_i \in H^1(\M_\B)$ have homogenous weight $1$, 
while the classes $\alpha\in  H^2(\M_\B)$, $\psi_i\in  H^3(\M_\B)$ for $i=1,\dots,2g$, and $\beta\in H^4(\M_\B)$  have homogenous weight $2$. 
\end{theorem} 
It follows that  homogenous elements generate  $H^*(\M_\B)$.  
The following is Corollary 4.1.11. in \ci{hausel-villegas}:
\begin{theorem}
\label{hdgttin}
The weight filtration $W_{\bullet}H^*(\M_\B)$ satisfies: 
\begin{enumerate}
\item
$W_{2k}H^*(\M_\B)=W_{2k+1}H^*(\M_\B)$ for all $k$.
\item
$\left({\rm Gr}^W_{2k}H^*(\M_\B)_{\comp}\right)^{pq}=0$ if  $(p,q) \neq (k,k)$.
\end{enumerate}
\end{theorem}
\n
Denoting  by $W^d_k(\M_\B)\subset H^d(\M_\B)$ the subspace of degree $d$ homogenous weight $k$ cohomology classes, 
we have the following splittings: 
\begin{equation} 
\label{weightsplitting} 
H^d(\M_\B)= \bigoplus_{k} W^d_k(\M_\B) \qquad W_{2k}H^d(\M_\B)= \bigoplus_{i\leq k} W^d_i(\M_\B).
\end{equation} 
Theorem 1.1.3 of \cite{hausel-villegas} gives a formula for
the mixed Hodge polynomials of $\pglM_\B$ and $\M_\B$ which implies the curious symmetries 
$$
\dim {\rm Gr}^W_{\dim \M_\B-2k}H^{*}(\M_\B)=\dim {\rm Gr}^W_{\dim \M_\B+2k}H^{*+2k}(\M_\B)
$$
 and 
$$
\dim {\rm Gr}^W_{\dim \pglM_\B-2k}H^{*}(\pglM_\B)=\dim {\rm Gr}^W_{\dim \pglM_\B+2k}H^{*+2k}(\pglM_\B).
$$

\bigskip
These equalities, called {\em curious Poincar\'e duality} in  \cite{hausel-villegas}, are made more precise and 
significant by the {\em curious hard Lefschetz } theorems. %%%: consider 
 Consider the class $\alpha \in H^2(\pglM_\B)$ introduced above, and the class 
$\tilde{\alpha} \in H^2(\M_\B)$ defined in terms of the isomorphism  (\ref{gl2-split})
by 
\begin{equation}
\label{defalpfatld}
\tilde{\alpha}:=1 \otimes \alpha + \left(\sum_{i=1}^g \epsilon_i\epsilon_{i+g}\right)\otimes 1.
\end{equation}
We then have ({\rm \cite[Theorem 1.1.5]{hausel-villegas}})
\begin{theorem}{\rm (Curious hard Lefschetz)}
\label{curioushrdlf}
The map given by  iterated cup product with  $\tilde{\alpha}$  induces isomorphisms:
\begin{equation} 
\label{cchhll}
\cup \tilde{\alpha} ^{k}:{\rm Gr}^W_{\dim \M_\B-2k}H^*(\M_\B)\stackrel{\cong}{\lorw} {\rm Gr}^W_{\dim \M_\B+2k}H^{*+2k}(\M_\B), \quad\forall k \geq 0.
\end{equation} 
Similarly, cupping with $\alpha$ defines isomorphisms
\begin{equation} 
\label{cchhll2}
\cup \alpha^{k}:{\rm Gr}^W_{\dim \pglM_\B-2k}H^*(\pglM_\B)\stackrel{\cong}{\lorw} {\rm Gr}^W_{\dim \pglM_\B+2k}H^{*+2k }(\pglM_\B), \quad\forall k \geq 0.
\end{equation} 
\end{theorem}
The present paper was  partly motivated by the desire of  giving a more conceptual explanation  for these curious hard Lefschetz theorems. 
	\end{subsubsection}
	\begin{subsubsection}{Moduli of Higgs bundles and their cohomology ring} 
\label{modhiggsbnd}
Let $C$ be a smooth complex projective curve  of genus $g \geq 2$. A Higgs bundle is a pair $(E,\phi)$ of a vector bundle $E$ on $C$ and a Higgs field 
$\phi\in H^0(C,{\rm End}\, E \otimes K_{\curv})$. Let $\M_\Dol$  denote the {\em $\GL_2$-Higgs  moduli space},  i.e. the moduli space of  stable Higgs bundles of rank $2$ and 
	degree $1$. It is a non-singular quasi-projective variety with $\dim \, \M_\Dol=8g-6$.
		
		Let us fix a degree 1 line bundle
		$\Lambda$ on $C$. Let $\slM_\Dol$ be the {\em $\SL_2$-Higgs moduli space} of stable Higgs bundles $(E,\phi)$ or rank $2$, with determinant $\det(E) \simeq \Lambda$ and trace-free ${\tr}(\phi)=0$ Higgs field. The moduli space $\slM_\Dol$ is a non-singular quasi-projective variety 
		with $\dim{\slM_\Dol}= 6g-6$. 
Defining the  map 
$$ \lambda_{Dol}:\M_{Dol}\mapsto \Jac^1_C\times H^0(C,K_{\curv}), \qquad  \lambda_{Dol}(E,\phi):=({\det}(E),{\rm tr}(\phi)),$$
 we have  $$\slM_{Dol}=\lambda_{Dol}^{-1}((\Lambda,0)).$$
Let $\M_\Dol ^0 \subseteq \M_\Dol$ be the subset of stable Higgs bundles with traceless
Higgs field:
$$
\M_\Dol ^0 = \{(E,\phi ) \hbox { with }{\rm tr}(\phi)=0 \}.
$$
The group $\Jac^0_C$ of degree $0$ holomorphic line bundles on $C$ acts on $\M_\Dol ^0$
 as follows:  $L\in \Jac^0_C$ sends $(E,\phi)$ to 
$(E\otimes L,\phi \otimes {\rm Id}_L)$. The group
$\Gamma:= \Jac^0_C[2]\cong \Z_2^{2g}$   of order $2$ line bundles  on $\curv$
 acts naturally on $\slM_\Dol$ in the same way.
The  two resulting  quotients
are easily seen to be isomorphic.  We call  the  resulting variety
the {\em $\PGL_2$-Higgs moduli space} and denote it by 
$$\pglM_\Dol=\M^0_\Dol/ \Jac^0_C =\slM_\Dol/\Gamma.$$ 
It is a quasi-projective $(6g-6)$-dimensional algebraic variety with finite quotient singularities.

The fundamental theorem of non-Abelian Hodge theory on the  curve $\curv$ for the
groups $\G=\GL_2, \SL_2$ and $\PGL_2$ under consideration can be stated as follows  
(\cite{hit,simpson,donaldson,corlette}):
\begin{theorem}[Non-Abelian Hodge theorem] 
There are canonical diffeomorphisms:
	$$\M_\B\cong \M_\Dol, \,\,\,\slM_\B\cong \slM_\Dol, \,\,\, \pglM_\B\cong \pglM_\Dol. $$
\end{theorem}
At the level of cohomology, the non-Abelian Hodge theorem yields canonical isomorphisms \begin{equation} \label{naht} H^*(\M_\B)\cong H^*(\M_\Dol), \,\,\,H^*(\slM_\B)\cong H^*(\slM_\Dol), \,\,\, H^*(\pglM_\B)\cong H^*(\pglM_\Dol). \end{equation}

\begin{remark}
{\rm The Hodge structure on the cohomology of these Higgs moduli spaces is  pure, and  its Hodge polynomial is known, see 
Conjecture 5.6 in  \ci{hausel}, which also proposes a conjectural formula for any rank.} 
\end{remark}

\medskip
Given a line bundle $D$ on $\curv$, we  can consider, more generally,  the moduli space of 
stable pairs $(E,\phi)$, where $E$ is a rank $2$ degree $1$ bundle on $\curv$ and $\phi \in H^0(\curv, {\rm End}\, E\otimes D).$ 
The corresponding moduli space is connected by Theorem 7.5 in  \ci{nitsure}, and,
if $\deg D>2g-2$, or if $D=K_{\curv}$, nonsingular (\ci{nitsure}  Proposition 7.4).  

\begin{notation}
\label{imptnt}
{\rm For the sake of notational simplicity, this moduli space
is denoted by $\higgsbu$ in the sequel of the paper, without mentioning its dependence on the line bundle $D$, always
meant to satisfy $\deg D>2g-2$, or  $D=K_{\curv}$. 
Whenever we talk specifically of the case $D=K_{\curv}$ we denote the corresponding moduli space 
by $\higgsbud$. 
}
\end{notation}

We still have the map 
$\lambda_{D}: \higgsbu \lorw \Jac_{\curv}^1 \times H^0(C,D)$ defined by  $\lambda_{D}(E,\phi) = ({\det}(E),{\rm tr}(\phi)).$ We set
$\mdsl:=\lambda_{D}^{-1}((\Lambda,0)), \, \mdpgl:=\higgsbu^0 /\Jac^0_C=
\mdsl/\Gamma$, where, as above,
$$
\higgsbu^0= \{(E,\phi )\in \higgsbu \hbox { with }{\rm tr}(\phi)=0 \},
$$
and $\Gamma= \Jac^0_C[2]\cong \Z_2^{2g}$ is the group of order $2$ line bundles on $\curv$, 
see \S \ref{comp3gra}. 

\bigskip

It is proved in 
\cite[(4.4)]{hausel-thaddeus1} that there is a  Higgs bundle $({\mathbb E}, {\bf \Phi})$ on 
$ \higgsbu \times C$ with the property that, for every family  of Higgs bundles $({\mathbb E}_S, {\bf \Phi}_S)$ parametrized by an algebraic variety $S$, there is a unique map 
$a : S \lorw \higgsbu $ and an isomorphism 
$$
({\mathbb E}_S, {\bf \Phi}_S) \simeq {\mathcal L} \otimes (a \times {\rm Id})^*({\mathbb E}, {\bf \Phi}) 
$$
for a uniquely determined line bundle $\mathcal L$ on $S$.

\begin{remark}
{\rm
The vector bundle ${\mathbb E}$ with the universal property stated above 
is determined up to twisting with a line bundle pulled back from $\higgsbu$; hence,  
given two different choices  ${\mathbb E}, {\mathbb E}^\prime$, we have a canonical isomorphism
of the associated endomorphisms bundles ${\rm End }\,\mathbb E \simeq {\rm End }\,\mathbb E^\prime$.
The vector bundle ${\rm End }\,\mathbb E$ on $\higgsbu \times C$ is thus unambiguously defined.}
\end{remark}
Let $e_1,\dots,e_{2g}$ be a symplectic basis of $H^1(C)$ and $\omega\in H^2(C)$ be the Poincar\'e dual of the class of a point.
The K\"unneth decomposition of the second Chern class of 
${\rm End }\,\mathbb E$
\begin{equation}
\label{defabps}
-c_2({\rm End }{\mathbb E})=\alpha \otimes \omega + \sum_{i=1}^{2g} \psi_i \otimes e_i + \beta \otimes 1 \in H^*(\higgsbu)\otimes H^*(C) 
\end{equation}
defines the classes $\alpha\in H^2(\higgsbu)$, $\psi_i\in H^3(\higgsbu)$ and $\beta\in H^4(\higgsbu)$. 
These classes define also classes in $H^*(\mdsl)$ by restriction, and in $H^*(\mdpgl)$ by restriction and $\Jac^0_C[2]$-invariance. They  will be denoted with the same letters.
In the case $D=K_{\curv}$, these classes coincide, via the isomorphisms \ref{naht}, with the classes in $H^*(\M_\B)$, denoted by the same symbols, 
defined in \S \ref{charactvrty}.

The generators of $H^*(\Jac^1_C)$ pull back to the classes
$\epsilon_1, \dots, \epsilon_{2g} \in H^1(\higgsbu)$ via the morphism 
$\higgsbu \to \Jac^1_C$ given by $(E, \phi) \mapsto \det(E)$.
 
The paper  \cite{hausel-thaddeus1} shows that the universal classes $\{
\epsilon_1, \ldots, \epsilon_{2g}, 
\alpha, \psi_1, \ldots, \psi_{2g},\beta \}$ are a set of  multiplicative generators for  $H^*(\higgsbu)$; 
the relations among these universal classes were determined in \cite{hausel-thaddeus2}.
Due to the role these relations play in this paper we summarize the  main result of 
\cite{hausel-thaddeus2}.

Because of the isomorphism $H^*( \higgsbu)\simeq H^*(\mdpgl)\otimes H^*(\Jac^0_C)$,  it is enough to describe
the ring $H^*(\mdpgl)$. 
We introduce the element
$$
\gamma:= -2 \sum_i \psi_i \psi_{i+g},
$$
we set $\Psi:={\rm Span}( \psi_i) \subseteq H^3(\mdpgl)$, and we define
$$
\Lambda_0^k:= {\rm Ker} \left\{ \gamma^{g+1-k}:\bigwedge^k \Psi \lorw \bigwedge^{2g+2-k}\Psi \right\}\,\, \hbox{ for } \, 0 \leq k \leq g \,\,\hbox{ and }\, \Lambda_0^k=0 \,\hbox{ for } k>g.
$$
By the standard representation theory of the symplectic group, 
there is a direct sum decomposition
$$
\bigwedge^k \Psi=\bigoplus_i \gamma^i \Lambda_0^{k-2i}. 
$$

\begin{definition}
Given two integers $a,b \geq 0$, we define $I^a_b$ to be the ideal of $\rat[\alpha, \beta, \gamma]$ generated by $\gamma^{a+1}$ and 
\begin{equation}
\label{relaht}
\rho_{r,s,t}^c:= \sum_{i=0}^{{\rm min}(c,r,s)}
\frac{\alpha^{r-i}}{(r-i)!}
\frac{\beta^{s-i}}{(s-i)!} \frac{(2\gamma)^{t+i}}{i!},
\end{equation}
where $c:= r+3s+2t-2a+2-b$, for all the $r,s,t\geq 0$ such that
\begin{equation}
\label{indexrange}
r+3s+3t>3a-3+b, \hbox{ and } \,r+2s+2t\geq 2a-2+b.
\end{equation}
\end{definition}

\begin{remark}
\label{verytrivial}
{\rm If $r=0$ and $b>0$, then the second inequality in (\ref{indexrange}) is strictly stronger 
than  the first.}
\end{remark}

\n
The main result of \ci{hausel-thaddeus1} is then 
\begin{theorem}
\label{hauthaddmntm}
The cohomology ring of $\mdpgl$ has the presentation
$$
H^*(\mdpgl)=\sum_{k=0}^g \Lambda_0^k(\psi)\otimes \left(\rat[\alpha, \beta,\gamma]/I^{g-k}_{\deg D+2-2g+k}\right).
$$
\end{theorem}

\medskip
\n
The form of the relations (\ref{relaht}) affords the following
\begin{definition}
\label{abstrwfiltr}
We define  the grading $w$ on $H^*(\mdpgl)$ by setting
$$
w(\alpha)=w(\beta)=w(\psi_i)=2, 
$$
and extending it by multiplicativity. This grading is well-defined since
the relations of Theorem \ref{hauthaddmntm} are homogenous 
with respect to this grading. We denote by $W'_{\bullet}$ the increasing filtration associated to this grading. 
\end{definition}
\medskip
If  $D=K_{\curv}$, 
thanks to the results of \ci{hausel-villegas} described in \S \ref{charactvrty},
$W'_{\bullet}$ on $H^*(\mdpgl)$ coincides, up to a simple renumbering
and via (\ref{naht}), with the weight filtration associated with 
the mixed Hodge structure on $H^*(\pglM_\B)$.
As mentioned in \cite[Remark 5.2.3]{hausel-villegas},  for a general $D$,  even though there is no associated Betti moduli space,
the filtration $W'_{\bullet}$ on $H^*(\mdpgl)$ turns out to have the same formal 
properties of the weight filtration   on $H^*(\pglM_\B)$
described in \S \ref{charactvrty}.  In particular,  \cite[Lemma 5.3.3]{hausel-villegas} implies that it satisfies 
the following curious hard Lefschetz property completely analogous to (\ref{cchhll2}) 
of Theorem \ref{curioushrdlf}: \begin{theorem}\label{chlwithpoles} For $\alpha\in H^2(\pglM)$ and we have the isomorphisms:\begin{equation} 
\cup \alpha^{k}:{\rm Gr}^W_{\dim \pglM-2k}H^*(\pglM)\stackrel{\cong}{\lorw} {\rm Gr}^W_{\dim \pglM+2k}H^{*+2k }(\pglM), \quad\forall k \geq 0.
\end{equation} 
\end{theorem}
Finally, setting $w(\epsilon_i)=1$, we get a grading and an associated filtration on $H^*(\higgsbu)$,
and all the discussion above goes through without any change.
\end{subsubsection}	 
\end{subsection}

\subsection{The Hitchin fibration and spectral curves}
\label{luca}

\subsubsection{The Hitchin fibration ($\G=\GL_2$)}
\label{hitchmap}

Given a Higgs field 
$\phi \in H^0(\curv, {\rm End}\,E \otimes D)$,
we have ${\rm tr}(\phi)\in H^0(\curv, D)$ and ${\rm det}(\phi) \in  H^0(\curv, 2D).$
The Hitchin map, 
 $\hitmap: \higgsbu \lorw \base$ assigns
\beq
\label{ppooii}
\higgsbu \ni  (E, \phi) \longmapsto
({\rm tr}(\phi), {\rm det}(\phi) )\in \base := H^0(\curv, D)\times H^0(\curv, 2D).
\eeq

Note that we don't indicate the dependence on the line bundle $D$ in the notation for the target 
$\base$ of the Hitchin map (cfr. the conventions introduced in Notation \ref{imptnt}). 
It follows from Theorem 6.1 in \ci{nitsure} that the  map $\hitmap$ is proper.

In the case of $\mdsl \subseteq \higgsbu$, the corresponding Hitchin fibration  $\hitmapc$
is just the restriction of $\hitmap$ to $\mdsl$. Since, by definition, if $(E, \phi) \in \mdsl$, 
then ${\rm tr}(\phi)=0$, we have  
\begin{equation}
\hitmapc: \mdsl \lorw \baseo:=H^0(\curv, 2D) \subseteq \base.
\end{equation} 
The map descends to the quotient $\mdpgl=\mdsl/\Gamma$, and we have
\begin{equation}
\hitmaph: \mdpgl \lorw \baseo.
\end{equation} 
In the rest of this section, we  concentrate on the map $\hitmap$.
The necessary changes  for dealing with the cases of $\hitmaph$ and $\hitmapc$ 
are discussed in Section \ref{comp3gra}. 

\subsubsection{The spectral curve construction}
\label{spctrcrv}
Let $\pi_D: \totspa \lorw \curv$  be  the total space of the 
line bundle $D$.  
For  $s:=(s_1,s_2) \in \base$ as in (\ref{ppooii}), 
the {\em spectral curve } $ \curv_s$     is 
the curve on $ \totspa$
defined by the  equation

\begin{equation}
\label{equaspectcur}
\left\{ y \in \totspa \,:\, y^2- \pi_D^*(s_1)y+ \pi_D^*(s_2)=0 \right\}.
\end{equation}
Spectral curves can be  singular, reducible, even non-reduced; they are locally planar, and, 
in force of our assumptions on the genus of $\curv$ and the degree of $D$, connected.
The restriction $\pro_s: \curv_s \lorw \curv$  
of the projection $\pi_D:   \totspa \lorw \curv$
exhibits $\curv_s$ as  a double cover of $\curv$. 
The equation (\ref{equaspectcur}) in $\totspa  \times
\base$ defines the  flat family $u$ of spectral curves 
\begin{equation}
\label{spettrale}
\xymatrix{
{{\mathscr C}}_{\mathcal A}\ar[rr]^{\pro} \ar[rd]^{\spectfam}&  & \curv \times \base \ar[ld]^{p_2 } \\ 
   &      \base &
}
\end{equation} 
where $\spectfam^{-1}(s)=\curv_s$, for all $ s \in \base$.
The family is equipped with the involution $\iota: {{\mathscr C}}_{\mathcal A} \lorw {{\mathscr C}}_{\mathcal A}$
over $\curv \times \base$  exchanging the two sheets of the covering. 

\medskip
The restriction of the relative Picard scheme  of the family $u$ to 
the smooth locus 
$$\basesm : = \{ s \in \base \hbox{ such that } \curv_s 
\,\hbox{is smooth} \},$$  is the disjoint union over $l\in \zed$ of the proper families   
$\rjmap^{\,l}: \reljac^{\,l} \lorw \basesm $
such that $(\rjmap^{\,l})^{-1}(s)=\jac{\curv_s}^{\, l}$
is the  component of the Picard variety of $\curv_s$
parametrizing
degree $l$ line bundles on $\curv_s$.
  
\begin{remark}
\label{oddeven}
{\rm Fix a  degree one line bundle ${\mathcal L}$  on $\curv$.
The operation of tensoring line bundles
of fixed degree with $\pro^*{\mathcal L}$ defines isomorphisms  
$\reljac^{\,l} \lorw \reljac^{\,l+2}$ of schemes over $\basesm$.
It follows that, up to isomorphisms,
there are only two such families, the abelian scheme $\reljac^{\,0}$ 
and the $\reljac^{\,0}$-torsor $\reljac^{\,1}$.
Sending a point $\widehat{c}\in \curv_s$ to the line bundle
$\mathcal O_{\curv_s}(\widehat{c})$ defines  an Abel-Jacobi-type $\basesm$-map ${\mathscr C}_{\basesm} \lorw \reljac^{\,1}$. }
\end{remark}

\medskip
The  Riemann-Hurwitz formula and  (\ref{equaspectcur}) imply at once the following
\begin{proposition} 
\label{singspcrv}
Let $s =(s_1, s_2) \in \base$. 
Assume $s_1^2-4s_2 \neq 0 \in  H^0(\curv, 2D)$.
\begin{enumerate}
\item
The spectral curve $\curv_s$ is reduced, and 
the covering $\pro_s: \curv_s \lorw \curv$ 
is branched at the zeros of $s_1^2-4s_2 $.
The point $s=(s_1,s_2) \in \basesm$ if and only if $s_1^2-4s_2$ 
has simple zeros, in which case $g(\curv_s) =2g-1+ \deg D$. 

\item
If  $s_1^2-4s_2$ vanishes to  finite order $k \geq 2$ at a point 
$ c\in \curv$, then the spectral curve $\curv_s$ 
has a planar singularity at the point   
$\pro_s^{-1}(c)$ which is  locally  analytically
isomorphic to $\{y^2-x^{k}=0\} \subseteq \comp^2$.
\end{enumerate}
\end{proposition}

\begin{remark}
\label{classmap}
{\rm  Associating with $s=(s_1,s_2) \in \base$ its discriminant divisor
$(s_1^2-4s_2) \in C^{(2r)},$ where $r:= \deg D$ and  $C^{(2r)}$ is    
the  $2r$-th symmetric product of $C$, gives a map $\Theta: \base \lorw C^{(2r)}$. }
\end{remark}

\bigskip
 We recall that if 
 $\mathcal F$ is a torsion-free sheaf on 
 an integral curve $\gencurv$, the rank of $\mathcal F$ is the dimension of its stalk at the  generic point of $\gencurv$, and the degree $\deg {\mathcal F}$ is defined as $\deg {\mathcal F}:=\chi(\gencurv, {\mathcal F})-{\rm rank}({\mathcal F})\chi(\gencurv, {\mathcal O}_{\gencurv})$.
For $l \in \zed$, then
the {\em compactified Jacobian $\compjac{\gencurv}^{\,l}$ of degree $l$}  
parameterizes torsion-free sheaves  of  rank $1$ and degree $l$ on $\gencurv$ (see \ci{ds, ak}).
Tensoring with a line bundle of degree $l$ gives an isomorphism 
$\compjac{\gencurv}^{\,0} \simeq \compjac{\gencurv}^{\,l}$.
If $\gencurv$ is smooth, then every rank $1$ torsion free sheaf is 
locally free and $\compjac{\gencurv}^{\,l}=\jac{\gencurv}^{\,l}$.

The following theorem (\ci{bnr}, Proposition 3.6) 
describes the fibres of the Hitchin map over a rather large open subset
of the base $\base$. Recall  Remark \ref{oddeven}
and that we are considering Higgs bundles of {\em odd} degree. 

\begin{theorem}
\label{fibres}
Let $s \in \base$ be  such that the spectral curve $\curv_s$ is integral. 
There is an isomorphism of varieties
$\hitmap^{-1}(s)\simeq \compjac{\curv_s}^a$, with $a=0$ if $\deg D$ is odd, and $a=1$ if $\deg D$ is even.
\end{theorem}

The isomorphisms assigns to a torsion free sheaf $\shf$ of degree $a$ 
on $\curv_s$ its direct image ${\pro_{s }}_* \shf$, which is a 
 torsion free ${\mathcal O}_C$-module  on $\curv$. Since $C$ is smooth,
 the sheaf  ${\pro_s}_* \shf$ is locally free of rank two, in view of the fact
 that  $\pro_s$ has degree $2$.
Since the map $\pro_s$ is finite, there are no higher cohomology sheaves, and 
$\chi(\curv_s,\shf)= \chi(\curv, \pro_{s*} \shf).$
The Riemann-Roch theorem
$$
\deg \pro_{s*} \shf + 2(1-g)= \chi(\curv, \pro_* \shf)=\chi(\curv_s,\shf)= \deg \shf +\chi(\curv_s,{\mathcal O}_{\curv_s})=    \deg \shf +2(1-g)-\deg D,
$$
implies that
$\pro_{s*} \shf$ has odd degree if $\deg \shf -\deg D$ is odd. 
The Higgs field arises  as multiplication by $y$ (see (\ref{equaspectcur}))
in view of   the natural structure of 
 $\pro_{s*} {\mathcal O}_{\curv_s}$-module on
$\pro_{s*} \shf$ (see \ci{bnr}, \S 3  for details).

In particular,  for every  $s \in \basesm$ the fiber $\hitmap^{-1} (s)$  can be identified,  noncanonically, 
with the Jacobian variety of the smooth spectral curve $\curv_s$. 
In fact, the  Abelian scheme $\reljac^{\,0}$ acts on 
$\higgsbusm:= \hitmap^{-1} (\basesm)$ making it into a torsor (see \ci{ngo}, Section 4.3).

The following is  well-known:

\begin{lemma}
\label{locsystors}
Let $\alpha: A \lorw S$ be an Abelian scheme,  let $\tau :P \lorw S$ be an $A$-torsor
and let $j \geq 0$.
Then there are natural isomorphisms of local systems
$$ 
R^j \tau_* \rat_P \simeq  R^j \alpha_* \rat_A \simeq   
\bigwedge^j R^1 \alpha_* \rat_A .
$$
\end{lemma}
\n
{\em Proof.}
Since the fibers of $A$ are connected, the action by translation on the cohomology of the fibers of $P$ is trivial.
Hence, the isomorphism of local systems
$ (R^j \tau_* \rat_P)_{|U} \simeq ( R^j \alpha_* \rat_A)_{|U} $ associated with a local trivialization of $P$
does not depend on the chosen trivialization.  
Consequently, the isomorphisms associated to a trivializing cover $\{U_i\}$ of $S$ 
glue to a global isomorphism of local systems.
The second isomorphism follows from the  K\"unneth isomorphism 
$H^l((S^1)^a)\simeq \bigwedge^l H^1((S^1)^a)$. 
\blacksquare

\begin{corollary}
\label{locsys}
There are canonical isomorphisms of local systems on $\basesm$:
\[ R^{\, l} \hitmapsmooth_* \rat_{\higgsbusm} \simeq  R^{\,l}\rjmap_* \rat_{\reljac^{\,1}} \simeq \bigwedge^l 
R^1 \rjmap_* \rat_{\reljac^{\,1}} \simeq \bigwedge^l R^1 \spectfamsmooth_* 
\rat_{{\mathscr C}_{\basesm}}.\]
\end{corollary}
\n{\em Proof.}
The first and second isomorphism follow by applying Lemma \ref{locsystors} to 
the $\reljac^{\,0}$-torsor $\higgsbusm$.
The  Abel-Jacobi $\basesm$-map ${\mathscr C}_{\basesm} \lorw \reljac^{\,1}$ 
of  Remark \ref{oddeven} induces via pullback a map of local systems
$R^1\rjmap_*^1 \rat \lorw R^1\spectfamsmooth_* \rat_{}$ which
is an isomorphism on each stalk, and this proves the third isomorphism. 
\blacksquare

\subsection{Perverse filtration}
\label{mark}
Let 
\[
\hi : \mi^{\dai +\ff} \lorw \ai^{\dai} 
\]
be a proper map
 of relative dimension $\ff$ between  irreducible   varieties of the indicated dimensions.
We assume that $\mi$ is nonsingular, or with at worst finite quotient singularities, and that the fibres have constant dimension $\ff$.
Let $\eta \in H^2(\mi)$ be the first Chern class of
a relatively ample (or $h$-ample) line bundle on $\mi$, i.e. a line bundle which is ample 
when restricted to every fiber of $\hi$. 

The goal of this section is to define  the perverse Leray filtration
$P$ on the cohomology groups 
$H^*(\mi)$  and to list and discuss some of  its relevant  properties.

\subsubsection{Definition of the perverse filtration $P$ on $H^*(\mi)$}
\label{sspf}
We employ freely the language of derived categories and perverse sheaves
(see the seminal paper \ci{bbd}, the survey \cite{bams}, 
or for example the paper \ci{htam}). Standard textbooks on the subject are \ci{dimca, iv,k-s}.

Let $D_{\ai}$ be the full subcategory of the bounded derived category 
of the category of sheaves of rational vector spaces on $\ai$ with objects
the bounded complexes with constructible cohomology sheaves.
We denote the derived direct image $R\hi_*$
simply by $\hi_*$ and, for $i \in \zed$,   the $i$-th hypercohomology group
of $\ai$   with coefficients in $K\in D_{\ai}$
by $H^i(\ai, K)$. If the index $i$ is unimportant (but fixed), we simply write
$H^*(\ai,K)$. We set $H^{\bullet}(\ai, K):= \oplus_i H^i(\ai, K)$.
 We work
with the middle perversity $t$-structure.  The corresponding
category of perverse sheaves is denoted by  $P_{\ai}$. Given $K \in D_{\ai}$, we 
have  the sequence of maps
of  ``truncated" complexes 
\[
\ldots \lorw \ptd{p-1} K \lorw   \ptd{p}K  \lorw \ptd{p+1} K \lorw \ldots
\lorw K \qquad 
p \in \zed,
\]
where $\ptd{p}K = 0$ for every $p \ll 0$ and $\ptd{p}K = K$ for every $p \gg 0$.
 The  (increasing) perverse filtration $P$ on the cohomology groups
 $H^*(\ai, K)$ is defined by taking the images of the truncation maps
 in cohomology:
 \begin{equation}
 \label{pf}
 P_p H^*(\ai,K): = \im \,\{ H^*(\ai, \ptd{p} K) \lorw H^*(\ai, K)\}.
 \end{equation}
Clearly, the perverse filtration on $H^*(\mi) = H^*(\mi, \rat_\mi)$  becomes trivial after a dimensional shift.
On the other hand, we also have the perverse filtration on $H^*(\ai, \hi_*\rat) = H^*(\mi)$
which, as it is the case for its variant given by 
the Leray filtration, is highly nontrivial. This is what may be 
called the perverse Leray filtration on $H^*(M)$ associated with $\hi$.

For the needs of this paper, we  want the perverse Leray filtration $P$ on $H^*(\mi)$
to be  of type $[0,2\ff]$, i.e. $P_{-1} =\{0\}$ and $P_{2\ff} = H^*(\mi)$,
and to satisfy  $1 \in P_0 H^0(\mi)$. In order to achieve this, we 
 define (with slight abuse of notation) the perverse Leray filtration on $H^*(M)$
 (with respect to $\hi$) by setting
 \[
 P_p H^*(\mi) :=  P_p H^{*- \dai}(\ai, \hi_* \rat_{\mi} [\dai]).\]
 Note that  in \cite{htam}, the perverse Leray  filtration is defined
 so that it is   of type $[-f,f]$. 

In order to simplify the notation, we set
\begin{equation}
\label{not}
H^*_{\leq p} (\mi) : = P_pH^*(\mi), \qquad
H^*_p(\mi) := Gr^P_p H^*(\mi):= P_p/P_{p-1}.
\end{equation}
In this paper, we  also use  the  graded spaces for the weight
filtration $Gr^W_wH^*(\mi)$ and we  employ the same
notation $H^*_w(\mi)$. In those cases,  we make it clear  which
meaning should be given to the symbols.

\subsubsection{Decomposition and relative hard Lefschetz theorems, primitive decomposition}
\label{dtrhl}
%Given a complex $K \in D_{\ai}$, 
%the  
%perverse cohomology  complexes $\phix{p}{K} \in P_{\ai}$, for  $p \in \zed$,
% are defined by forming the 
%distinguished triangles in $D_{\ai}$:
%\[  \ptd{p-1} K \lorw   \ptd{p} K \lorw   \phix{p}{K}[-p] \stackrel{[1]}\lorw.\]
Define
\begin{equation}
\label{ppii}
{\mathcal P}^p: = \phix{p}{ \hi_* \rat_{\mi} [\dai]} \, \in \, P_{\ai}, \quad p \in \zed,
\end{equation}
where $\phix{p}{-}$ denotes the $p$-th perverse cohomology functor. 
We have that ${\mathcal P}^p=0$ for $p \notin [0, 2\ff]$.
The decomposition theorem for the proper map
$\hi: \mi \to \ai$ then  gives the existence of   isomorphisms in $D_{\ai}$
\begin{equation}
\label{dt}
\varphi: 
\bigoplus_{p=0}^{2\ff} {\mathcal P}^p[-p] \stackrel{\simeq}\lorw
 \hi_* \rat_{\mi} [\dai].
\end{equation}
We have identifications
\begin{equation}
\label{5tgb}
H^*_{\leq p} (\mi) =  
\bigoplus_{p'= 0}^{ p}  \varphi  \left(H^*_{p'}(\mi) \right), \qquad
H^*_p(\mi) =  H^{*- \dai-p} (\ai, {\mathcal P}^{p} ). \end{equation}
\begin{remark} 
\label{attocc}
{\rm
The images $\varphi  \left(H^*_{p}(\mi) \right) \subseteq
H^*(\mi) $ depend on $\varphi$. If $H^*_{\leq p-1} (\mi)= \{0\}$, then the image
$\varphi  \left(H^*_{p}(\mi) \right)=
H^*_{\leq p} (\mi)$ is independent of $\varphi$.
In particular, the image 
$\varphi  \left(H^*_{0}(\mi) \right)=H^*_{\leq 0}(\mi)$ 
is independent of $\varphi$.
}
\end{remark}

One of the deep assertions of the decomposition theorem is that
each perverse sheaf ${\mathcal P}^p$ is semisimple and
splits canonically into a direct sum 
\begin{equation}
\label{phi}
 {\mathcal P}^p = \bigoplus_{Z}   IC_Z(L_{Z,p})
 \end{equation}
 of intersection complexes
over a finite collection of  distinct  irreducible closed  subvarieties $Z$ in $\ai$ with coefficients given
by semisimple   local systems $L_{Z,p}$  defined on a dense open subset  $Z^o\subseteq
Z_{reg} \subseteq Z$  of 
the regular part of $Z$.

There are the following three basic symmetries.

\begin{enumerate}
\item
(PVD) Poincar\'e-Verdier duality :   if we denote the Verdier dual of $K$ by $K^{\vee}$,
then we have:
\begin{equation}
\label{pvd}
{\mathcal P}^{\ff-i}\simeq  ({{\mathcal P}^{\ff +i}})^{\vee},
 \qquad \forall i \in \zed.
\end{equation}

\item
(RHL)  Relative hard Lefschetz:  for every $i \geq 0$, the $i$-th iteration of the operation
 of cupping with the $h$-ample line bundle  $\eta$ yields isomorphisms
 of perverse sheaves
 \begin{equation}
 \label{rhl}
 \eta^i : {\mathcal P}^{\ff-i} \stackrel{\simeq}\lorw
 {\mathcal P}^{\ff+i};\end{equation}
in particular, we have the hard Lefschetz isomorphisms
 at the level of graded groups (still called RHL):
 \begin{equation}
 \label{rhlg} \eta^i:
 H^*_{\ff -i}(\mi)  \stackrel{\simeq}\lorw H^{*+2i}_{\ff+i} (\mi), \qquad 
 \forall \;  i \geq 0 .\end{equation}
 
 \item
 (Self-duality) The isomorphisms PVD and RHL are  compatible
 with the direct sum decomposition (\ref{phi}) and, in particular,
 the local systems $L_{Z,p}$ are self-dual.
 \end{enumerate}
 
 Recall that $P_{\ai}$ is an Abelian category. By a standard abuse of notation,
 which greatly simplifies the notation, 
 we view kernels and images as subobjects.
 
 For 
 $0 \leq i \leq \ff$ and $0 \leq j \leq  \ff -i$ define
 \begin{equation}
 \label{qs}
 {\mathcal Q}^{i,0} := \ke \, \left\{ \eta^{\ff -i +1}: {\mathcal P}^{i}
 \lorw  {\mathcal P}^{2\ff- i+2} \right\}, \qquad
 {\mathcal Q}^{i,j} := \im \, \left\{ \eta^{j}: {\mathcal Q}^{i,0} \lorw 
   {\mathcal P}^{i+2j} \right\}
 \end{equation}
 and  set ${\mathcal Q}^{i,j}=0$ for all the other values of $(i,j)$.
  The RHL (\ref{rhl})  then yields the  
 natural primitive decompositions in $P_{\ai}$:
% \begin{equation}
% \label{pld}
% \phix{\ff-i}{  \hi_* \rat_{\mi}[\dai]} =\bigoplus_{j\geq 0} {\mathcal Q}^{i+2j,j},
% \qquad
 % \phix{\ff+i}{  \hi_* \rat_{\mi}[\dai]} =\bigoplus_{j\geq 0} {\mathcal Q}^{i+2j,j+i}.
% \end{equation}
 %This can be written also
 \begin{equation}
 \label{inuno}
  {\mathcal P}^{k} = \bigoplus_{j \geq 0} {\mathcal Q}^{k-2j,j},
  \qquad
  \forall k \in \zed^{\geq 0}.
 \end{equation}
 
\subsubsection{Deligne's $Q$-splitting associated with the relatively ample $\eta$}
\label{delspl}  
The paper \cite{deligneseattle} defines three preferred decomposition isomorphisms (\ref{dt}) associated with the 
$h$-ample line bundle  $\eta$. 
We consider the first of them, (see also  \cite{decmigseattle}),  
 which we denote by $\phi_{\eta}$
 and name the Deligne isomorphism; notice, however, that the indexing scheme employed here differs from that of  \cite{deligneseattle,decmigseattle}.
 The cohomological  properties of the  Deligne isomorphism
 needed in this paper are the following:
 
 \begin{fact}
 \label{chardeliso}
 {\rm
 The map %%%(Please check this correction, Tamas) $\phi_{\eta}$
 $$\phi_{\eta}: 
\bigoplus_{p=0}^{2\ff} {\mathcal P}^p[-p] \stackrel{\simeq}\lorw
 \hi_* \rat_{\mi} [\dai].$$ is characterized by the following properties. Let $0 \leq i \leq \ff$. Then:
 
 \n
 (i) Applying the functor  $\phix{i}{-} $ to the map  ${\phi_\eta}_{|}: {\mathcal Q}^{i,0}[-i] \to h_* \rat [\dai]$ gives 
 the canonical inclusion ${\mathcal Q}^{i,0} \subseteq {\mathcal P}^{i}$ ;
 
 \n
 (ii)  for every $s>  \ff -i$, the composition below is zero
\[{\mathcal Q}^{i,0}[-i] \stackrel{{\phi_\eta}_{|}}\lorw h_* \rat [\dai]  \stackrel{\eta^s}\lorw
h_* \rat [\dai][2s] \lorw \left(\ptu{\ff +s} \rat [\dai] \right) [2s],
 \]  
or,  equivalently, the composition $\eta^s \circ
{{\phi_\eta}_{|}}$  factors through $^{\mathfrak p}\!D^{\leq f-s-1}_{\ai}$.
 }
 \end{fact}

 For $d \geq 0$, $0 \leq i \leq \ff$ and $0 \leq j \leq \ff -i$, define
 \begin{equation}
 \label{pofqs}
 Q^{i,j;\, d} := \phi_\eta  \left(   H^{d - \dai  -i -2j} \left(\ai, {\mathcal Q}^{i,j} \right)   \right)
  \subseteq
 H^d_{\leq i +2j}(\mi),
 \end{equation}
 \begin{equation}
 \label{pofqss}
  Q^{i,j}:= \bigoplus_{d\geq 0} Q^{i, j; \, d}\subseteq
 \bigoplus_{d \geq 0} H^{d}_{\leq i+2j}(\mi),
 \end{equation}
  and define $Q^{i,j}= Q^{i,j;\, d}= \{0\}$ for all the other values of $(i,j;d)$.
 We  then have the following decompositions, which depend on $\phi_\eta$:
 \begin{equation}
 \label{qdec}
  H^{\bullet} (\mi)= 
 \bigoplus_{i,j}Q^{i,j}, \qquad     H^{d} (\mi)=  \bigoplus_{ i,j} Q^{i,j;d},
 \end{equation}
 \begin{equation}\label{qdecc}
 H_{\leq p}^{\bullet}(\mi) =
 \bigoplus_{d} H^d_{\leq p} (\mi) =\bigoplus_{ i,j,d, \,
 i+2j \leq p} Q^{i,j;d}=
 \bigoplus_{i,j, \,  i+2j \leq p} Q^{i,j}.
\end{equation}
Every  $u\in H^{\bullet} (\mi)$ admits the  {\em $Q$-decomposition}
associated with the splitting $\phi_\eta$:
 \begin{equation}
 \label{qdeci}
 u = \sum u^{i,j}, \qquad u^{i,j} \in Q^{i,j}.
 \end{equation}

By construction, we have 
\[
H_p^d (\mi) = \bigoplus_{i+2j=p} Q^{i,j;\, d},\qquad
H_p (\mi) = \bigoplus_{i+2j=p} Q^{i,j}.\]

 The properties  of the Deligne splitting that we need, and that follow
 from Fact \ref{chardeliso}, are 
 \begin{equation}
\label{novel}
\eta\, Q^{i,j} = Q^{i,j+1}, \; \forall\; 0 \leq j <\ff -i, \qquad
\eta\, Q^{i, f-i} 
\,\subseteq \,\bigoplus_{0 \leq l \leq \min{(\ff-i, \ff -k)}} Q^{k,l}.
\end{equation}
In particular,  we have the  simple relation
\begin{equation}
\label{12w}
Q^{i,j}= \eta^j Q^{i,0}, \qquad \forall \, 0 \leq j \leq \ff -i.
\end{equation}

Here is some ad-hoc notation and terminology.
Let $p \in \zed$ and  $u \in H^{\bullet}_{\leq p}(\mi)$. Denote by $[u]_p\in H^{\bullet}_p(\mi)$
 the natural projection to the graded group.
In what follows, we add over  $ 0 \leq i \leq \ff $ and $0 \leq j \leq \ff -i$. We have 
\[u = \sum_{i+2j \leq p} u^{i,j},
\qquad
[u]_p =  \sum_{ i+2j = p} [u^{i,j}]_p.\]
We say that:

\begin{itemize}
\item
$u$ {\em has   perversity $\leq p$}; 

\item
$u$ {\em has  perversity $p$} if   $[u]_p \neq 0$;

\item
the class  $0\in H^\bullet(\mi)$ has
perversity $p$, for every $p \in \zed$;

\item $u$ is {\em sharp} if $u^{i,j}=0$
whenever $i +2j <p$; note that the zero class  is automatically sharp
and that a class may have a given perversity without being sharp;

\item
 $u= u^{p,0} \in Q^{p,0} \subseteq H^{\bullet}_{\leq p}(\mi)$
is  {\em RHL-primitive}; note that such a class is automatically sharp.

\end{itemize}

One should not confuse RHL-primitivity with primitivity:
if $u \in Q^{p,0}$, then
\begin{equation}
\label{predth}
\eta^{\ff -p+1} u^{p,0} \in H^{\bullet}_{\leq 2\ff-p+1}(\mi),
\qquad {\rm i.e.} \qquad
 [\eta^{\ff -p+1} u^{p,0}]_{2\ff -p+2}=0,\end{equation}
whereas, one could have  $\eta^{\ff -p+1} u^{p,0}\neq 0$.
 
  \medskip
  Recall that we have chosen the Deligne splitting $\phi_\eta$ associated with $\eta$.
  The following lemma does not hold
  for an arbitrary splitting  $\varphi$  in (\ref{dt}). 
  
  \begin{lemma}
  \label{nomix}
 {\rm  ({\bf Non-mixing lemma})}
 Let $u  \in H^{\bullet}_{\leq p}(\mi)$.  If
  $\eta^{\ff - p+1}u =0$,    then $u$ is RHL-primitive, i.e.  $u= u^{p,0}\in Q^{p,0}$.
    \end{lemma}
\n
{\em Proof.}
In view of the $Q$-decomposition, we  can write
\[
u= u^{p,0} + \sum_{j\geq 1} u^{p- 2j,j} + \sum_{s+2t <p} u^{s,t},
\] 
where the first two summands are sharp and have perversity $p$
and the third has perversity $\leq p-1$.
By  (\ref{novel}) and  (\ref{12w}) 
we deduce that
\[
\eta^{\ff -p+1} u^{p,0}   \in   \bigoplus_{0 \leq l\leq \ff - p, \ff -k}Q^{k,l},
\quad
\eta^{\ff -p+1} u^{p-2j,j}   \in   Q^{p-2j,\ff -p +1 +j}, \quad
\eta^{\ff-p+1} u^{s,t}   \in   Q^{s,t+\ff -p+1}.
\]
The  three collections of $Q$-spaces above have no term in common.
It follows  that all three terms in $\eta^{\ff -p+1} u=0$ are zero.
By RHL (\ref{rhl}), cupping with $\eta^{\ff -p+1}$ is injective on the spaces
$Q^{p-2j,j}$, $j\geq 1$, and $Q^{s,t}$ above. We deduce
that $u^{p+2j,j}= u^{s,t} =0$.
\blacksquare

\subsubsection{The perverse filtration and cup-product}
\label{badb}
The following is a crude, completely general, estimate:
\begin{lemma}
\label{fax}
Let $u \in H^d(\mi)$. Then the cup product map with $u$ satisfies:
\[
\cup u:  H^*_{\leq p}(\mi) \lorw H^{*+d}_{\leq p + d} (\mi).
\]
\end{lemma}
\n
{\em Proof.}
We have  $H^d(\mi)= {\rm Hom}_{D_{\mi}}(\rat_{\mi}, \rat_{\mi} [d])$
so that we may  view the cohomology class $u$  as a map $u: \rat_{\mi}[a] \to \rat_{\mi} [\dai+d].$
The cup product map $\cup u$ coincides with the map induced in cohomology by the pushed-forward map
$\hi_*u:  \hi_*\rat_{\mi}[\dai] \to \hi_* \rat_{\mi} [\dai+d]$.
We apply truncation and obtain the map
$\ptd{p}  \hi_*\rat_{\mi}[\dai]  \to \ptd{p} \hi_*\rat_{\mi} [\dai+d] = 
(\ptd{p+d}  \hi_*\rat_{\mi} [\dai]) [d]$. The assertion follows  after taking cohomology.
\blacksquare 

A much better estimate, leading to the key Proposition \ref{testpr}, holds under the following:

\begin{assumption}
\label{ipot}
 The intersection complexes $IC_Z (L_{Z,p})$  {\rm (\ref{phi})}  appearing in the decomposition theorem  
for $\hi_* \rat_{\mi}[\dai]$ have   strict support  $\ai$ (i.e. each $Z =A$) 
and  $IC_A(L_{\ai, p}) = R^0 j^o_* L_{\ai,p} [\dai]$,  where $j^o:
\ai^o \to \ai$ is the  immersion of an open dense subset.
\end{assumption}

\begin{fact}
\label{fatto}
{\rm 
 Take 
$\ai^o$ to be the open set over which $\hi$ is smooth,
and  set $R^p:= (R^p \hi_* \rat_{\mi})_{| \ai^o}$. 
We may re-phrase  Assumption \ref{ipot} as follows:
\[\hi_* \rat_{\mi}[\dai] \simeq 
\bigoplus_{p=0}^{2\ff} IC_{\ai}(R^p)[-p] = 
\bigoplus_{p=0}^{2\ff} R^0 j^o_*R^p[\dai][-p] =
\bigoplus_{p=0}^{2\ff} R^p\hi_* \rat_{\mi} [\dai] [-p].\]
As a consequence, if Assumption \ref{ipot} holds, then 
the perverse Leray filtration on $H^*(\mi)=
H^{*- \dai}(\ai, \hi_* \rat_{\mi}[\dai])$  coincides with the
standard Leray filtration on $H^*(\mi)= H^*(\ai, \hi_*\rat_{\mi})$.
}
\end{fact}

\begin{proposition}
\label{mult}
If Assumption {\rm \ref{ipot}} holds, 
 then we have
\begin{equation}
\label{cupcomp}
H^*_{\leq p} (\mi) \otimes H^\star_{\leq q}(\mi) \lorw H^{*+\star}_{\leq p +q}(\mi).
\end{equation}
\end{proposition}
\n
{\em Proof.} It is a known fact that the multiplicativity property (\ref{cupcomp}) holds with respect to the standard Leray filtration (see \cite{decberkeley}, Theorem 6.1 for a proof).
The statement then follows, since, as noticed in Fact \ref{fatto}, the Assumption \ref{ipot} implies  that the perverse Leray filtration and the  standard Leray filtration coincide. 
\blacksquare

\medskip
Let us assume that the target $\ai$ of the map $\hi: \mi \to \ai$ is affine
of dimension $\dai$, and let $\ai \subseteq {\mathbb C}^N$ be an arbitrary closed embedding.
Let $s \geq 0$, $\Lambda^s \subseteq \ai$ be 
a general  $s$-dimensional linear section and let
$\mi_{\Lambda^s} := \hi^{-1}( \Lambda^s)$. For $s <0$, we define
$\mi_{\Lambda^s}:= \emptyset$.

The following is the main result
of \ci{pflht} (Theorem 4.1.1).

\begin{theorem}
\label{gdpf}
A class $u \in H^d_{\leq p} (\mi)$ iff $u_{| \mi_{\Lambda^{d-p-1} }}=0$.
\end{theorem}
 
 \begin{remark}
 \label{cbop}
 {\rm Theorem \ref{gdpf} implies in particular that $H^d_{\leq p} (M) =0$ if $p<d-a$,  and that $H^d_{\leq p} (M) =H^d(M)$ if $p\geq d$.}
 \end{remark}
 \begin{remark}
 \label{transv}
 {\rm 
 For  $\Lambda^{s}$   general,    transversality implies 
 the following (see \ci{htam}, Lemma 4.3.8):
 if $u\in H^d_{\leq p}(\mi)$, then,  $u_{|\mi_{\Lambda^s}  } \in
  H^d_{\leq p}(\mi_{\Lambda^s})$; in other words, the change in perversity is compensated
  by the change in codimension.}
 \end{remark}
 
% \subsubsection{A test for perversity}
% \label{test}
 Let $U \subseteq \ai$ be a  Zariski dense open  subset
 satisfying Assumption \ref{ipot} (with $U$ replacing $\ai$),
 and hence the conclusions of Fact \ref{fatto}.
  Let $Y:= \ai \setminus U$ be the closed complement. Note that such an open
   set $U$
 always exists, e.g. $U= \ai^o$, the set over which the map is smooth.
 However, $Y$ could be rather large, i.e. have small codimension.
 The following proposition is key to our analysis of the
 perverse filtration in the cohomology ring of the moduli of Higgs bundles,
 where, as it turns out, the set $Y$ is small  just enough to let us go  by.

 \begin{proposition}
 \label{testpr}
Let  $u_i \in H^{d_i}_{\leq p_i} (\mi)$ for $i =1,\cdots, l$. 
Let $d:= \sum d_i$ and $p:=\sum p_i$.  
 If $ d  - p  -1 < {\rm codim}\, Y$, then $v:= u_1 \cup u_2\cup \cdots, \cup u_l  \in H^{d}_{\leq p} (\mi)$.
 \end{proposition}
 \n{\em Proof.}
 By the  assumption on  ${\rm codim}\, Y$, a general $\Lambda^{d-p-1}$ 
 misses $Y$.
 By applying Remark \ref{transv} to the classes  $u_i$,  and then Proposition
 \ref{mult} to their  restriction to $\mi_{\Lambda^{d-p-1}}$,
 we conclude that the resulting  $v_{| \mi_{\Lambda^{d-p-1}}} \in 
 H^{d}_{\leq p} (\mi_{\Lambda^{d- p-1}})$.
 As noticed in Remark \ref{cbop}, we have $v_{| \mi_{\Lambda^{d-p-1}}}=0$.
 We conclude by Theorem
 \ref{gdpf}.
 \blacksquare

 \subsubsection{Extra vanishing when $H^j(\mi)=0$, $\forall j >2\ff$.}
 \label{louso}

 \begin{proposition}
 \label{ndomett}
 Assume that  $\ai$ is affine and that $H^j(\mi) =0$, $\forall j > 2\ff$. 
 Then the  perversity of 
$u \in H^d(\mi)$  is in the interval
$
\left[  \lceil \frac{d}{2} \rceil, 
d \right]$.
\end{proposition}
\n{\em Proof.} We may assume that $u \neq 0$. 
 Let $p$ be the perversity of $u$. Assume that $p < \lceil \frac{d}{2}\rceil $.
 In particular, $p <\ff$ and $ 2p <d$. By RHL (\ref{rhlg}) and by the
  assumption
 on vanishing, we reach the contradiction 
 $0 \neq \eta^{\ff -p} u \in H^{2\ff-2p +d}(\mi)= \{0\}$. 
 The upper bound follows from
Theorem \ref{gdpf},   as noticed in Remark \ref{cbop}. \blacksquare

\begin{corollary}
\label{ndomett2}
Under the hypothesis of Proposition {\rm \ref{ndomett}}, we have that
\[
H^d_{\leq \lceil \frac{d}{2} \rceil} (\mi) = Q^{\lceil \frac{d}{2} \rceil, 0;\, d},
\qquad
H^d_{\leq \lceil \frac{d}{2}  \rceil +1} (\mi) = Q^{\lceil \frac{d}{2} \rceil, 0;\,d}
\bigoplus Q^{\lceil \frac{d}{2} \rceil -1, 1; \, d}.
\]
 \end{corollary}
 \n{\em Proof.}  
 By Proposition \ref{ndomett}, we have $H^d_{\leq \lceil \frac{d}{2} \rceil -1} (\mi) =\{0\}$.
 Then (\ref{5tgb}) implies that 
$H^d_{\leq \lceil \frac{d}{2} \rceil } (\mi) =
\phi_\eta(H^{d-a -  \lceil \frac{d}{2}\rceil} (\ai, {\cal P}^{\lceil \frac{d}{2} \rceil}))$. The equation
(\ref{inuno}) implies that
 \[
 H^d_{\leq \lceil \frac{d}{2} \rceil} (\mi) = 
 \bigoplus_{j \geq 0}
 Q^{\ \lceil \frac{d}{2} \rceil -2j, j; \, d}=
 \bigoplus_{j \geq 0}
 \eta^j\, Q^{\lceil \frac{d}{2} \rceil -2j, 0; \, d-2j}.
 \]
By  Proposition \ref{ndomett}, since 
$\lceil \frac{d}{2} \rceil -2j <  \lceil (d-2j)/2 \rceil$, 
we have  that 
$Q^{\lceil \frac{d}{2} \rceil -2j, 0; \, d-2j}=\{0\}$
for $j >0$. 

\n
The  assertion in perversity $\lceil \frac{d}{2} \rceil +1$  is proved in the same way.
  \blacksquare
 
% \begin{remark}
 %\label{ancheal}
 %{\rm
 %As the reader may verify, Corollary \ref{ndomett2} is readily generalized
 %to all perversities, provided one adds suitable summands.
 %}
 %\end{remark}

\section{Cohomology over the elliptic locus}
\label{prfmaintm}
\subsection{Statement of Theorem \ref{final}} 
\label{statmnthm}

We go back to the set-up of Section \ref{luca}.
 
\begin{definition}
\label{goodlocus}
The {\em elliptic locus} $\basegood \subseteq \base$
is the set of points $s=(s_1,s_2) \in \base$ for which  the associated spectral curve $\curv_s$
is integral. We set $\higgsbugood := \hitmap^{-1}( \basegood)$.
\end{definition}

\begin{remark}
\label{goodisirre}
{\rm 
Let $s=(s_1,s_2) \in \base$. Since the covers $\curv_s \to \curv$
have degree $2$, if the section  
$s_1^2-4s_2 \in H^0(\curv, 2D)$    vanishes    with odd multiplicity  
at least at one point of $\curv$, then
$s\in \basegood$.
Since $2D$ has even degree,
there is an even number of  points on $C$ where  $s_1^2-4s_2$ has  odd vanishing order.}
\end{remark}

\begin{lemma}
\label{smingo}
The set $\basegood$ is Zariski open and dense in $\base$ and contains $\basesm$.
The complement $\base \setminus \basegood$
is a closed  algebraic subset of codimension $\deg D$ if $\deg D>2g-2$,
and of codimension $2g-3$ if $D=K_{\curv}$.
\end{lemma} 
\n{\em Proof.}
Since  $\deg D>0$,
given $s=(s_1, s_2)\in \base$,  the zero locus of its discriminant divisor $s_1^2-4s_2$ is not empty, so that 
the spectral covering $\pi_s:C_s \longrightarrow C$ is never \'etale. 
A nonsingular spectral curve must therefore be irreducible, namely $\basesm \subseteq \basegood$. 
The spectral curve associated with $s=(s_1, s_2)$ is a divisor
on the  nonsingular surface $\totspa$, and it is not  integral
precisely when $s$ is in the image of the finite map 
$H^0(\curv,D)\times H^0(\curv,D) \longrightarrow    \base$ sending $(t_1,t_2)$ to $(t_1+t_2,t_1t_2)$; therefore, the  image $\base \setminus \basegood$ is a closed subset.
By the Riemann-Roch theorem on $\curv$, we have that if $\deg D>2g-2$, then 
$\dim (\base \setminus \basegood)= 2(\deg D+1-g)$ and  $\dim \base = 3\deg D+2(1-g)$, while, for  $D=K_{\curv}$, we have that
$\dim (\base \setminus \basegood)= 2g$ and $\dim \base = 4g-3$.
\blacksquare

\medskip
We denote by $j:\basesm \lorw \basegood$ the open imbedding, and by $b_l$ the
$l$-th Betti number.

\medskip
Recall from Theorem~\ref{fibres} the noncanonical isomorphism
$\compjac{\curv_s}^0 \simeq \hitmap^{-1}(s)$.
Section \ref{prfmaintm} is devoted to the  proof of the following
\begin{theorem}
\label{final}
For $s\in \basegood$ and for  $l\geq 0$, we have
\begin{equation}
\label{unglblch}
\dim\left( (R^0 j_*R^l\hitmapsmooth_* \rat)_s\right)=b_l(\compjac{\curv_s})=
b_l(\hitmap^{-1}(s)).
\end{equation}
\end{theorem}

\medskip
\n
Theorem  \ref{final} 
readily implies the following:
\begin{corollary}
\label{icareshfongl}
The  perverse sheaves ${\mathcal P}_{\basegood}^l$ 
appearing in  the statement of the decomposition theorem
{\rm (\ref{phi})} of \S {\rm \ref{dtrhl}} for the Hitchin map over the open set
 $\basegood$ satisfy 
$$
{\mathcal P}_{\basegood}^l = IC_{\basegood}(R^l\hitmapsmooth_* \rat) 
= R^0 j_*R^l\hitmapsmooth_* \rat  \, [ \dim{\base}], \quad
\forall \,l,
$$
i.e. there is only one intersection complex, supported on the whole $\basegood$,
 given by a sheaf in  the  single cohomological  
 degree $-\dim{\base}$. In particular,  Assumption {\rm \ref{ipot}} of \S {\rm \ref{badb}}
is fulfilled. 
\end{corollary}
\n{\em Proof that Theorem \ref{final} implies Corollary \ref{icareshfongl}.}
Set  $R^l:=R^l\hitmapsmooth_* \rat$  and $a:= \dim{ \base}$.
On the smooth locus $\basesm$, the decomposition theorem takes the form
$\hitmapsmooth_* \rat \simeq  \bigoplus R^l [-l].$
It follows that
$(\hitmap_* \rat[a])_{|\basegood} \simeq   \left( \bigoplus IC(R^l )[-l] \right) \bigoplus K,$
where $K$ is a direct sum of shifted semisimple perverse sheaves supported on 
proper subsets of $\basegood$. 
Taking the stalk at $s \in \basegood$ of the cohomology sheaves 
$$
H^{k+a}(\hitmap ^{-1}(s))\simeq {\mathcal H}^k(\hitmap_* \rat[a ])_s \simeq  \left( \bigoplus_l
{\mathcal H}^{k-l}( IC( R^l ))_s \right) \bigoplus {\mathcal H}^k(K)_s,
$$
and  
$$
b_{k+a}(\hitmap^{-1}(s))= \dim {\mathcal H}^k(\hi_* \rat[\dai ])_s=
\sum_l \dim {\mathcal H}^{k-l}( IC(R^l))_s + \dim {\mathcal H}^k(K)_s.
$$
By the very definition of intersection cohomology complex   
${\mathcal H}^{-\dai}( IC(R^l))=R^0 j^0_*R^l$, hence the equality (\ref{unglblch}) forces 
${\mathcal H}^{r}( IC(R^l))=0$  for $ r\neq -\dai$  and 
${\mathcal H}^{r}(K)=0$ for all  $r$.
\blacksquare

\bigskip
Let us briefly outline the structure of proof of Theorem \ref{final}, which
occupies the remainder of \S \ref{prfmaintm}. 
In \S \ref{cpctjcb}, we prove an upper bound, Theorem
\ref{estbncj}, on the Betti numbers of the compactified Jacobian of an integral 
curve with $A_k$-singularities. The partial normalizations of such a curve define a natural stratification 
of the compactified Jacobian; the cohomology groups of the strata are easy to determine, 
and the spectral sequence arising from the stratification gives the
desired  upper bound.
In \S \ref{lbebncj}, we complement  this upper bound estimate  with   a 
lower bound estimate, Theorem \ref{mainmonthm},  for the dimension of the stalks $(R^0 j_*R^l\hitmapsmooth_* \rat)_s$. 
The proof of Theorem \ref{mainmonthm} consists of a monodromy computation
which  
is completed in  \S \ref{bsc}. 
In  \S \ref{lbebncj},  we also prove that 
the decomposition theorem  forces the  equality of the two bounds. 
This completes the proof of Theorem \ref{final}.

\begin{remark}
{\rm   The arguments used in the proof of  Theorem \ref{final}
do not depend on the specific features of the Hitchin map, and
hold more generally in the following setting: 
suppose $S$ is a nonsingular complex variety, and 
${\mathscr C} \lorw S$ is a proper family of integral curves, 
smooth over the open set
$S_{\rm reg}  \stackrel{j}{\hookrightarrow} S$,
which are branched double coverings of a fixed curve $C$.
Let ${\mathscr I} \stackrel{f}{\lorw} S$ be the associated  family
of compactified Jacobians. If $\mathscr I$  
is  nonsingular, or has at worst finite quotient singularities, then Theorem \ref{final} and its Corollary \ref{icareshfongl} hold: in particular, for $s \in S$, we have
$b_l(f^{-1}(s))=  \dim\left( (R^0 j_*R^l f_* \rat)_s\right)$.
}
\end{remark}

\subsection{The upper bound estimate}
\label{cpctjcb}
In this section, $\gencurv$  denotes an integral  projective  curve whose singularities are 
 double points of type $A_k$, i.e.  analytically isomorphic to $(y^2-x^{k+1} =0) \subseteq \comp^2$
for some $k\geq 1$.  In view of Proposition \ref{singspcrv},
and  of the fact that we work exclusively with integral spectral curves,
this is the generality we need.
If $k\geq 3$,  then  blowing up  a point of type $A_k$ produces a point of type $A_{k-2}$,
and
if $k=1,2$, respectively corresponding to an ordinary node and a cusp, then blowing up a point resolves the singularity.
The invariant $\delta_c:= \dim_{\comp} \widetilde{\mathcal O}_{\gencurv,c}/ {\mathcal O}_{\gencurv,c}$ measures the drop of the arithmetic  genus under   normalization.
If  $c \in \gencurv$ is a singular point of type $A_k$, then
$\delta_c = \left\lceil \frac{k}{2} \right\rceil$.
The point $c$ must be blown-up $\delta_c$ times in order to be resolved, and,  with the exception of the blowing up of
$A_1$, each blowing up  map is bijective.

The following theorem lists a few well-known facts concerning the 
Picard variety of a curve with singularities of type $A_k$:

\begin{theorem}
\label{genfctspic}
Let  $\gencurv$ be a reduced and connected projective curve
with at  worst $A_k$-singularities. Let  $\gencurv_{\rm sing} \subseteq \gencurv$ be its singular locus, and let $\nu: \widetilde{\gencurv} \lorw \gencurv$ be  the normalization map.
Let $\nu^*: \jac{\gencurv}^0 \lorw \jac {\widetilde{\gencurv}}^0 $ be the map induced
by pull-back, where $\jac{\gencurv}^0$ (resp.    $\jac {\widetilde{\gencurv}}^0$) is the connected component of the identity of the Picard scheme of
$\gencurv$ (resp. ${\widetilde{\gencurv}}$).  
If  $c \in \gencurv_{\rm sing}$ is a singular point of type $A_{k}$, set
$$
 {\mathcal P}_{c}:=\left\{\begin{array}{ll}
\comp^{\delta_c}  & \hbox{ if } k \hbox{ is even},\\
\comp^{\times}\times \comp^{\delta_c-1}   &  \hbox{ if } k \hbox{ is odd},
\end{array}\right. 
$$
and define the commutative algebraic group ${\mathcal P}:= \prod_{c \in \gencurv_{\rm sing}} {\mathcal P}_{c} $.
Then:
\begin{enumerate}
\item
If  $\gencurv$ is irreducible, there is an exact sequence of  
commutative algebraic groups  
\begin{equation}
\label{picirre}
1 \lorw   {\mathcal P} \lorw
\jac{\gencurv}^0 {\lorw} \jac {\widetilde{\gencurv}}^0 \lorw 1,
\end{equation}
\item
If $\gencurv$ is reducible, and $\sharp$ is the number of irreducible components,  we have an exact sequence
\begin{equation}
\label{picre}
1 \lorw  {\mathcal P}/(\comp^{\times})^{\#-1} \lorw
\jac{\gencurv}^0 {\lorw} \jac {\widetilde{\gencurv}}^0 \lorw 1.
\end{equation}
\end{enumerate}
\end{theorem}
\n{\em Proof.}
These facts  follow directly from the exact sequence
of sheaves of groups on $\gencurv$
$$
1 \lorw {\mathcal O}_{\gencurv }^{\times} \lorw \nu_*{\mathcal O}_{\widetilde{\gencurv} }^{\times} \lorw
\nu_*{\mathcal O}_{\widetilde{\gencurv} }^{\times}/ {\mathcal O}_{\gencurv }^{\times} \lorw 1,
$$
and a local computation (see \cite{liu}, \S 7.5, especially Thm. 5.19).
\blacksquare

\bigskip
The connected group  $\jac{\gencurv}^0$  acts, via tensor product, on the compactification 
$\compjac{\gencurv}^0$, which is obtained by adding  degree zero
rank $1$ torsion free sheaves
on $\gencurv$   which are not locally free.

Let  $\nu: \gencurv' \lorw \gencurv$  
be a   finite birational map. There is the  direct image map
\[\nu_*: \compjac{\gencurv'}^0 \lorw \compjac{\gencurv}^0, \qquad
\shf' 
\longmapsto \nu_*\shf' 
.\]
The following theorem summarizes most of the properties
 of the compactified Jacobians of blow-ups that we need in the sequel
of the paper.
\begin{theorem}
\label{genfctscj}
Let  $\gencurv$ be an integral, projective curve
with at  worst $A_k$-singularities, 
 let $\nu: \gencurv' \lorw \gencurv$ be  a   finite  birational map
 and let $\shf \in  \compjac{\gencurv}^0$.    Then we have
\begin{enumerate}

\item 
The compactified Jacobian $ \compjac{\gencurv}^0$ is irreducible. The action of $\jac{\gencurv}^0$
 has finitely many orbits.
The orbit corresponding to locally free sheaves is dense.

\item
The direct image map
$\nu_*: \compjac{\gencurv'}^0 \lorw \compjac{\gencurv}^0$
is  a closed imbedding with image a closed 
$\jac{\gencurv}^0$-invariant subset of $\compjac{\gencurv }^0$. 
The image  of $\jac{\gencurv '}^0$
is a locally closed 
$\jac{\gencurv}^0$-invariant subset of $\compjac{\gencurv }^0$. 

\item
There are  a unique finite
birational map $\mu: \gencurv_{\shf} \lorw \gencurv$,
obtained as a composition of simple  blow-ups, 
and a line bundle $\mathcal L_{\shf}$ on $\gencurv_{\shf}$, such that
$\shf=\mu_* \mathcal L_{\shf}$. 

\item
Let $\pos$ be the  poset of  blow-ups  $\gencurv' \lorw \gencurv$.
There is a decomposition into locally closed subsets
\begin{equation}
\label{decpctjac}
\compjac{\gencurv }^0=\coprod_{\{\gencurv ' \to \gencurv \} \in {\pos} }\jac{\gencurv '}^0.
\end{equation}
\end{enumerate}
\end{theorem}
\n{\em Proof.} The proof of 1.   can be found in  \ci{rego, aik}. The  proof of
2. can be found in  \ci{beauvi}. 
The proof of 3. and 4. can be found in  \ci{gagne}, Proposition 3.4. 
\blacksquare

\bigskip
The main goal of this section is to prove Theorem \ref{estbncj},
which  gives an upper bound for the Betti numbers of  the compactified Jacobian
$\compjac{\gencurv}^0$.  
In order to achieve this upper bound,  we study the decomposition (\ref{decpctjac}) by
describing the poset ${\pos}$ of  all the  blowing-ups of $\gencurv$.

\bigskip
\begin{definition}
\label{stfcrvc}
An integral projective curve $\gencurv$  with    $A_k$-singularities is 
said to be of {\em singular type} 
\[\underline{k}:=(k_1, \ldots, k_o;k_{o+1}, \ldots, k_{o+e})\]
 if its singular locus consists of   $o+e$  distinct points
$\{c_1,\ldots, c_o,c_{o+1}, \ldots, c_{o+e}\}$, where $c_a$ is singular of type $A_{k_a}$,
with 
 $k_a$  odd for $1\leq a\leq o$, and  $k_a$  even for $o+1\leq a \leq o+e$. 
We say that  each 
  singular point is  of one of  two  possible types: odd, or  even, and
   we set $O:=\{c_1,\ldots, c_o\} \subseteq \gencurv_{\rm sing}$, the set of odd singular points, and $E:=\{c_{o+1},\ldots, c_{o+e}\} \subseteq \gencurv_{\rm sing}$ the set of even singular points.
   \end{definition}
  Recall that for each entry $k_a$ above, we have defined an integer $\delta_{c_a}:= \lceil k_a/2 \rceil$.

\begin{lemma}
\label{betnum}
Let $\gencurv$ be of singular type $\underline{k}$, let  $\widetilde\gencurv$
be its normalization and let  
 $\widetilde{g}: = g( \widetilde{\gencurv})$.
Then
$$
\sum_l b_l(\jac{\gencurv}^0)=2^{2 \widetilde{g}+ o}. 
$$
\end{lemma}
\n{\em Proof. }
A connected commutative Lie group is isomorphic to $(S^1)^r \times \real^s$, 
for some $r$ and $s$. The Betti numbers satisfy
$\sum_l b_l((S^1)^r \times \real^s )=2^r$. In our  case,  Theorem
\ref{genfctspic} implies  that $r=2\widetilde{g}+ o$. 
\blacksquare

\bigskip

The poset of blowing ups
 of $\gencurv$ can be described as the set 
$$
{\pos}= \{ I=(i_1,\ldots,i_o ; i_{o+1}, \ldots ,  i_{o+e})\in \nat^{o+e}  \, | \;
 0\leq i_a \leq \delta_{c_a},\;  \forall a =1, \cdots , o+e \},
$$
where we say that $I\geq I'$ if $i_a\geq i_a'$ for all $a=1, \ldots , o+e$.
Let $\gencurv_I$ be the curve obtained from $\gencurv$ by
blowing up, in any order, $i_1$ times the point $c_1,$  
$i_2$ times the point $c_2$,  etc. 
Let  $\nu_I:\gencurv_I \lorw \gencurv$ be  the corresponding  finite birational map.
The singular points of  $\gencurv_I$ are still of type $A_k$.

Theorems \ref{genfctscj}  and \ref{genfctspic} can be   applied
 to $\jac{\gencurv_I}^0$ and
to  $\compjac{\gencurv_I}^0$.
Note that $\gencurv_I \lorw \gencurv$ factors through $\gencurv_{I'} \lorw \gencurv$
if and only if $I\geq I'$, and that
if $I=(\delta_{c_1}, \ldots , \delta_{c_{o+e}})$, then $\widetilde{\gencurv}= \gencurv_I \lorw \gencurv$ is the normalization.
Define $|I|:=\sum i_a$. 
Theorem  \ref{genfctspic} implies that 
\begin{equation}
\label{dimiok}
\dim \jac{\gencurv_I}^0=\dim \jac{\gencurv}^0 -|I|.\end{equation}
For $I \in {\pos}$, the direct image  $\nu_{I,*}$ defines a locally closed imbedding 
$\jac{\gencurv_I}^0 \lorw \compjac{\gencurv}^0$. By applying (\ref{decpctjac}) to the
natural  maps 
$\gencurv_{I'} \lorw \gencurv_I$ for $I'\geq I$, we see that
$$
\compjac{\gencurv_I}^0=\coprod_{I'\geq I} \jac{\gencurv_{I'}}^0.
$$

\begin{proposition}
\label{spsqnc}
We have the following inequality concerning Betti numbers 
$$\sum_{l\geq 0} b_l(\compjac{\gencurv}^0)\leq 
\sum_{I\in {\pos}}\sum_{l\geq 0} b_l(\jac{\gencurv_I}^0).
$$
\end{proposition}
\n{\em Proof.} 
Let  $r$ be a nonnegative integer. Define the subset of $\compjac{\gencurv}^0$
$$
Z_r = \coprod_{|I|\geq r} \jac{\gencurv_I}^0.
$$
In view of the discussion above, we have

\begin{enumerate}
\item
$Z_r$ is a closed subset of $\compjac{\gencurv}^0$. In particular, it is compact;

\item
there are closed inclusions
$\emptyset \subseteq Z_{\delta} \subseteq \ldots 
\subseteq Z_1 \subseteq Z_0=\compjac{\gencurv}^0,$ where $\delta:=\sum_{c \in \gencurv_{\rm sing}} \delta_{c}$

\item
$Z_r \setminus Z_{r+1}= \coprod_{|I|=r} \jac{\gencurv_I}^0$, 
where the union is over the
connected components, all of which have the same dimension by (\ref{dimiok}).
\end{enumerate}
The nested inclusions 2., yield the classical spectral sequence 

\begin{equation}
\label{clospseq}
E_1^{p,q} = H^{p+q}(Z_{-p}, Z_{-p+1}) \Longrightarrow H^{p+q}( \compjac{\gencurv}^0).
\end{equation}
In view of the compactness 1., the $E_1$-term reads
$$
E_1^{p,q}=H^{p+q}_c(Z_{-p}\setminus Z_{-p+1})=\bigoplus_{|I|=-p}H^{p+q}_c(\jac{\gencurv_I}^0).
$$
By Poincar\'e duality, we have
that $\sum_{l\geq 0} b_l (\jac{\gencurv_I}^0) = \sum_{l\geq 0} \dim{H^l_c (\jac{\gencurv_I}^0)}$.
It follows that 
$$
\sum_{p,q} \dim{ E_1^{p,q}}= 
\sum_{I\in {\pos}}\sum_{l\geq 0} b_l(\jac{\gencurv_I}^0).
$$ 
Clearly, $\sum_{p,q}\dim{E^{p,q}_r} \leq \sum_{p,q} \dim{E_1^{p,q}}$,
for  every $r\geq 1$, and the statement follows.
\blacksquare

\bigskip
In what follows, we adopt the convention that a product over the empty set equals $1$.
For every subset $J \subseteq O$,  let $\delta_J:= 
\prod_{c \in J}\delta_c$. We have the following  
\begin{lemma}
\label{tot}
We have 
\[
\sum_{ I \in {\pos}} \sum_{l\geq 0}  b_l(\jac{\gencurv_I}^0)=2^{2\widetilde{g}} \left( \prod_{c \in O }(2\delta_c+1)\right)\left(\prod_{c \in E}(\delta_c+1)\right).
\]
\end{lemma}
\n{\em Proof.}
Let $o_I$ be the number of odd points on $\gencurv_I$.
For  every $0 \leq r \leq o$, let $\#_r$  be the number of  curves $\gencurv_I$ with a given $o_I =r$.
Since Lemma \ref{betnum} holds for every $\gencurv_I$, we have that
\begin{equation}
\label{99ii}
\sum_{ I \in {\pos}} \sum_{l\geq 0}  b_l(\jac{\gencurv_I}^0) =\sum_{ I \in {\pos} }2^{2\widetilde{g}+o_I}  =\sum_{r=0}^o \#_r \, 2^{2\widetilde{g} + r}.\end{equation}
We have 
$$
\#_o= \left(\prod_{c \in O} \delta_c \right)\left(\prod_{c \in E} (\delta_c+1)\right);$$ 
in fact,   the following two operations leave the number of
odd points unchanged: blowing up
    $t$ times, $0 \leq t \leq \delta_c$,  an even point $c \in E$, and blowing up 
   $t$ times,
   $0 \leq t < \delta_c$   an  odd point $c \in O$.

\n
In order to have precisely $o-1$ odd points, we need
to first blow up   
$\delta_c$ times  an odd point $c$. Once this is done, we repeat the count above
and deduce that 
$$
\#_{o-1}= \left(\sum _{j=1}^o\prod_{\stackrel{c \in O}{ c\neq c_j}} \delta_c \right)
\prod_{c \in E} \left(\delta_c+1\right)=
\left(\sum_{\stackrel{J \subseteq O} { \sharp J=o-1}} \delta_J
\right)
\prod_{c \in E}(\delta_c+1).
$$
It is clear that we can repeat this  argument and  re-write
the last term in 
(\ref{99ii}) as
\[
2^{2\widetilde{g}}
\left(\sum_{r=0}^o 2^r \sum_{\stackrel{J\subseteq O } { \sharp J=r}}\delta_J\right)
\prod_{c \in E} (\delta_c+1)= \left(\prod_{c \in O}(2\delta_c+1)\right)\left(\prod_{c \in E} (\delta_c+1)\right),
\]
by the elementary equality $\sum_{r=0}^o 2^r \sum_{\stackrel{J\subseteq O} { \sharp J=r}}\delta_J =\prod_{c \in O}(2\delta_c+1).$
\blacksquare

\medskip
\n
Finally, we combine  Proposition \ref{spsqnc} and  Lemma \ref{tot}
and obtain the desired upper bound:

\begin{theorem}
\label{estbncj}
Let $\gencurv$ be an integral  curve all of whose singularities are of type $A_k$.
Denote by 
$O:=\{c_1,\ldots, c_o\}$, the set of  its singular points of type $A_k$ with $k$ odd, and by 
$E:=\{c_{o+1},\ldots, c_{o+e}\}$ the set of its  singular points of type $A_k$ with $k$ even.
Denote by $\widetilde{g}$ the 
genus of the normalization $\widetilde{\gencurv}$ of $\gencurv$. Then:
$$\sum_l b_l \left(\compjac{\gencurv}^0 \right)\leq    
2^{2\widetilde{g}}  \left( \prod_{c \in O}(2\delta_c+1)\right)\left(\prod_{c \in E}(\delta_c+1)\right).
$$
\end{theorem}

In fact,  Theorem \ref{final} below implies  that the inequality above 
is in fact an equality. In particular, see Corollary~\ref{btnmbrjccom}, the spectral sequence (\ref{clospseq})
degenerates at $E_1$.
\subsection{The lower bound estimate}
\label{lbebncj}
The aim of this section is to prove  Theorem \ref{mainmonthm}, which, as we show below,
readily implies Theorem \ref{final}.

\begin{theorem}
\label{mainmonthm}
Let $s \in \basegood$, let  $\curv_s$ be the corresponding spectral  curve
with its singular locus
 $\{c_1, \ldots, c_o, c_{o+1}, \ldots, c_{o+e}\}$.  Let
$O:=\{ c_1, \ldots, c_o \}$ be the set of points of type $A_k$ with $k$ odd, and let
$E:=\{c_o, \ldots, c_{o+e} \}$ be the set of points of type  $A_k$ with $k$ even. 
Denote by $j:\basesm \lorw \basegood$ the  open imbedding.
Then
\[   
2^{2\widetilde{g}}  \left( \prod_{c \in O}(2\delta_c+1)\right)\left(\prod_{c \in E}(\delta_c+1)\right) \leq
\sum_{l} \dim{ \left(R^0 j_*R^l\hitmapsmooth_* \rat \right)_s},
\]
where $\widetilde{g}$ denotes the 
genus of the normalization $\widetilde{\curv _s}$ of $\curv_s$. 
\end{theorem}

\n{\em Proof that Theorem {\rm \ref{mainmonthm}} implies Theorem {\rm \ref{final}}}.
We have the following inequalities
\[
 \sum_{l} \dim{ \left(R^0 j_*R^l\hitmapsmooth_* \rat \right)_s} \, \leq  \,
\sum_l b_l \left(\compjac{\curv_s}^0 \right)
\,  \leq \,  \sum_{l} \dim{ \left(R^0 j_*R^l\hitmapsmooth_* \rat \right)_s} ,
\]
where the first one follows from the general equality
\[
{\mathcal H}^{-\dim\, \base}( IC_{\base}(R^l\hitmapsmooth_* \rat)  )_s=
\left(R^0 j_*R^l\hitmapsmooth_* \rat \right)_s
\] combined with the decomposition theorem 
(\ref{dt}) and (\ref{phi}) in \S \ref{dtrhl} 
for the Hitchin map over $\basegood$ where we add up only 
the summands supported on $\basegood$, and the second inequality
follows immediately by combining Theorem \ref{estbncj}
and Theorem \ref{mainmonthm}.
\blacksquare

\bigskip

\n{\bf Outline of the strategy for the proof of Theorem \ref{mainmonthm}.}

\n
Let 
$s\in \basegood$.
In view of Corollary \ref{locsys}, we have the natural isomorphism:
$$
\left(R^0 j_*R^l\hitmapsmooth_* \rat \right)_s \simeq
\left(R^0 j_*\bigwedge^l R^1\spectfamsmooth_* \rat \right)_s=
\lim_{\lorw} \Gamma \left(N \cap \basesm, \bigwedge^l R^1\spectfamsmooth_* \rat
\right), 
$$
where the direct limit is taken over the set of connected
neighborhoods $N$  of $s$ in $\base$.

Fix a base-point $n_0 \in N\cap \basesm$. We have the monodromy representation 
$$
\pi_1(N \cap \basesm,n_0) \lorw {\rm Aut} (H_1(C_{n_0})),
$$
and  its exterior powers     
$$
\pi_1(N \cap \basesm,n_0) \lorw {\rm Aut} (\bigwedge^l H_1(C_{n_0})).
$$

The  evaluation map 
$\Gamma (N \cap \basesm, \bigwedge^l R^1\spectfamsmooth_* \rat) \lorw \bigwedge^l H_1(C_{n_0})    $
at the point $n_0$ identifies the vector space of sections $\Gamma (N \cap \basesm, \bigwedge^l R^1\spectfamsmooth_* \rat)$ 
with the subspace of monodromy invariants of $\bigwedge^l H_1(C_{n_0})$. Thus, in order to prove 
Theorem \ref{mainmonthm} we need to investigate the monodromy  of the  restriction  of the spectral curve family
$\spectfamsmooth :{\curv}_{\basesm} \to  \basesm$
to $N \cap \basesm$, where $N$ is a small enough
connected neighborhood of $s$ in $\base$.

We consider the local family $   {\mathscr C}_{\mathcal U} \lorw  {\mathcal U}$ of double coverings of $\curv$ whose 
branch locus is ``close'' to that of $\curv_s$,
i.e. it is contained in a neighborhood ${\mathcal U}$ of the divisor $(s_1^2-4s_2)$ in the symmetric product of $\curv$.
The family has the property that every other family of double coverings whose branch locus is contained in ${\mathcal U}$
is the pull-back of $ {\mathscr C}_{\mathcal U} \lorw  {\mathcal U}$    via a uniquely determined map, see Proposition \ref{univfam} for a precise statement.
We investigate the monodromy of the smooth part of this  family, and  we 
determine the dimension of the subspace of  monodromy invariants in the exterior powers of the associated local system. 
Since the spectral curve family, restricted to a 
small enough neighborhood of $s$ in $\base$, is 
isomorphic to the pullback of  this local family via the map $\Theta$ of  Remark \ref{classmap}, 
the dimension of the subspace of  monodromy invariants of the local family gives a lower bound for the 
dimension of  the monodromy invariants of the spectral curve family, thus proving Theorem \ref{mainmonthm}. 

\begin{remark}
{\rm While our analysis of the monodromy is purely local, 
a detailed study of the global monodromy of the family $\higgsbusm \lorw \basesm$ 
has been carried out, for $\curv$ a hyperelliptic curve, by Copeland in \cite{copeland}.}
\end{remark}

\begin{notation}
\label{pernonimpazzire}
{\rm %%In the sequel 
In the remaining of \S \ref{lbebncj}, 
for notational simplicity, we denote with the same symbol a cycle (resp. relative cycle) 
and the homology (resp. relative homology) class 
it defines. In particular an equality of cycles (resp. relative cycles) will always mean equality of their homology 
(resp. relative homology) classes. }
\end{notation}

\subsubsection{The double covering of a disc }
\label{dcd}
We review some basic facts (see \cite{ar}, Part 1) concerning the topology of a holomorphic branched double
covering $\rho: \Su \lorw \overline{\D}$ of the closed  unit disc
$\overline{\D} \subseteq \comp$, 
with boundary $\partial\, \D$ and interior $\D$, under the following:
\begin{assumption}
\label{enopblass} 
The map $\rho$ is the restriction of a holomorphic mapping
from  thickenings of domain and codomain, there are no  branch points on $\partial \D$, and 
the degree $2r$ of the branch locus  divisor $Z$  is even.
\end{assumption}

Let $p_Z(z)$ be the monic degree $2r$ polynomial vanishing on $Z$: then
\begin{equation}
\label{model}
\Su=\left\{(z,w) \in \overline{\D} \times \comp \hbox{ such that } w^2=p_Z(z) \right\}, \quad \rho(z,w)=z.
\end{equation}

\medskip
\begin{remark}
\label{kindergarten}
{\rm
Since   $p_{Z}(z)$ has even degree,  
the boundary  $\partial \, \Su=\rho^{-1}(\partial \, \D)$ of $\Su$ consists of 
two connected components $\partial\,'$ and $ \partial\,''$, which we endow with the 
orientation induced from $\Su$. We denote the  resulting 
cycles in homology with the same symbols (cf. Notation \ref{pernonimpazzire}).
}
\end{remark}

Assume $Z$ consists of $2r$ distinct points.
By the Riemann-Hurwitz formula,  $\Su$ is
biholomorphic to a compact Riemann surface   of genus $r-1$ 
with two open disks removed.

Denote by $\I:=[0,2r+1] \subseteq \real $, and
let $\beta: \I \lorw \overline{\D}$ be a differentiable imbedding such that
$\partial \, \D \cap \beta(\I)= \{\beta(0), \beta(2r+1) \} $ and
$Z= \{\beta(1), \cdots , \beta(2r)\}$.
The subsets
$$
\lambda_j := \rho^{-1}\left(\beta([j,j+1] )\right), \,\, j=1, \cdots , 2r-1,
$$
are closed curves, which we orient subject to the requirements
\begin{equation}
\label{intnum1}
(\lambda_j,\lambda_{j+1})=1,
\end{equation}
where $(\, , \,)$ denotes the intersection product with respect to the the natural orientation
of $\Su$, and the equality in homology
\begin{equation}
\label{relbdry}
\sum_{j=1}^{r} \lambda_{2j-1} =\partial '.
\end{equation}
The $1$-cycles $\{\lambda_j\}_{j=1}^{2r-1}$
form  a basis for the first homology group
$H_1(\Su)$.

\begin{remark}
\label{relhom}
{\rm   In view of the
long exact sequence in relative homology  of the pair $(\Su, \partial \, \Su)$,
the kernel of the natural map 
$H_1(\Su) \lorw H_1(\Su, \partial \,\Su)$ is one dimensional,
generated by  the cycle $\sum_{j=1}^{r} \lambda_{2j-1} =\partial '$.
In order to complete the set $\{\lambda_j\}_{j=1}^{2r-1}$  to a system of generators of
$H_1(\Su, \partial \,\Su)$, we need to  add a relative $1$-cycle sent via the boundary map to a generator of 
$\ke \, \{H_0(\partial \,\Su)\lorw H_0(\Su)\}$,
namely a cycle
joining the two components
$\partial\,', \partial\,''$ of the boundary. 
We take the relative homology class of the curve 
\begin{equation}
\label{defimu}
\mu:= \rho^{-1}\left(\beta([0,1] )\right) \subseteq \Su
\end{equation}
oriented so that  we have the first equality below. We have 
\begin{equation}
\label{intnum2}
(\mu, \lambda_1)=1, \,\hbox{ and } \,(\mu, \lambda_j) = 0, \;\; \forall \; 
1 < j \leq 2r-1  .
\end{equation}

\n
The relative homology classes
of the cycles  $\{ \lambda_1, \ldots, \lambda_{2r-1}, \mu \}$  form  a set of generators for
$H_1(\Su, \partial \Su)$ subject to the only  relation $\sum_{j=1}^{r} \lambda_{2j-1} =0$. }
\end{remark}

\subsubsection{The family of coverings of the disc and its monodromy}
\label{ufcd}
We identify the symmetric product
$\D^{(2r)}$, parametrizing the  effective divisors of degree $2r$ 
on the unit disk  $\D$,
with the space of monic 
polynomials  of degree $2r$ whose roots have absolute values
less than $1$, by sending $v=(v_1, \ldots, v_{2r}) \in \D^{(2r)}$ to $p_v(X)=\prod_{1}^{2r}(X-v_l)$.
The elementary 
symmetric functions   of $(v_1, \ldots, v_{2r})$  give a system of coordinates for  
$\D^{(2r)}$, thus realizing it   as a bounded open subset of $\comp^{2r}$.

\begin{notation}
\label{notasim}
{\rm  We denote a point in $\D^{(2r)} \subseteq \comp^{2r}$
by the divisor $v$ on $\D$ or  by its associated monic  polynomial $p_v$.
}
\end{notation}

On $\overline{\D}\times \D^{(2r)}$ there is 
the divisor 
$$
{\mathscr Z}_{2r}:=\left\{(z, p)\in \overline{\D}\times \D^{(2r)} \hbox{ such that }p(z)=0      \right\}
$$
and the  double covering 
$$
{\mathscr S}_{2r}=\left\lbrace (z,p,w) \in \overline{\D} \times \D^{(2r)} \times \comp \hbox{ such that } w^2=p(z)   \right\rbrace,
$$
defining the family $\Phi_{2r}:{\mathscr S}_{2r} \lorw \D^{(2r)}$  
of (possibly singular) Riemann surfaces with boundary
(for every fiber $\Su_v$, the singularities are disjoint from the boundary)

\beq
\label{famsimdisc}
\xymatrix{  {\mathscr S}_{2r}    \ar[rd]_{\rho_{2r}} \ar[rr]^{\Phi_{2r}}    &  &  \D^{(2r)}    \\
   & \overline{\D} \times    \D^{(2r)}.    \ar[ru]^{p_2}  &  
}
\eeq

The map $\rho_{2r}$ is a double covering branched over ${\mathscr Z}_{2r}$, and,
for  $v \in \D^{(2r)}$, the fibre $\Su_v := \Phi_{2r}^{-1}(v)$ 
is  the double covering $\Su_v  \lorw \overline{\D}$ 
of equation $w^2=p_v(z)$ branched precisely  over the effective divisor $v$ in $\D$.

\begin{remark}
\label{bdrytriv}
{\rm 
By Remark \ref{kindergarten},
the boundary of every fibre of the map 
$\Phi_{2r}$ consists of two connected components.
Since $\D^{(2r)}$ is contractible,  
we have a smooth  trivialization 
$\partial {\mathscr S}_{2r}\simeq \left( S^1 \coprod S^1     \right)\times  \D^{(2r)}$,
well-defined up to isotopy.  
}
\end{remark}

\medskip
The locus $E$ of polynomials with  vanishing discriminant is a divisor  in $\D^{(2r)}$, 
and $\D^{(2r)}_{{\rm reg}}:= \D^{(2r)}\setminus E$
is the open subset  corresponding
to multiplicity free divisors, namely 
$2r$-tuples 
of distinct points in $\D$. 
The double covering $\Su_v= \Phi_{2r}^{-1}(v)$  of $\overline{\D}$  introduced in \S \ref{ufcd} 
is nonsingular if and only if $v \in \D^{(2r)}_{{\rm reg}}$.

\medskip

We choose a base-point   $\underline{v} \in \D^{(2r)}_{{\rm reg}}$. The fundamental group  
$\pi_1(\D^{(2r)}_{{\rm reg}},\underline{v})$  is the  classical 
braid group $\mathscr B^{2r}$ on $2r$ strands (see \ci{ar}, \S 3.3). 
As in \S \ref{dcd}, a differentiable imbedding $\beta: \I \lorw \D$ such that $\underline{v}=\{\beta(1), \cdots , \beta(2r)\}$, 
 defines a basis $\{\lambda_j\}_{j=1}^{2r-1}$ of  $H_1(\Su_{\underline{v}})$, the relative class   $\mu \in H_1(\Su_{\underline{v}}, \partial \, \Su_{\underline{v}})$, and
the usual set $T_1,\ldots, T_{2r-1} $ of 
generators of $\mathscr B^{2r}$: if $v_i:=\beta(i)$,  the braid
$T_i$  exchanges $v_i$ with $v_{i+1}$ by a  half-turn.  More precisely,  
let $\D^+, \D^-$ be the two open half-discs determined by $\beta$ and its orientation; then $T_i$  can be represented by two curves $\tau^+, \tau^{-}: [0,1] \lorw \D$ such that 
\begin{equation}
\label{genbraids}
\tau^+(0)=\tau^-(1)=v_i,   \, \, \tau^+(1)=\tau^-(0)=v_{i+1}, \, \, \tau^+((0,1)) \subseteq D^+,\, \,  \tau^{-}((0,1))\subseteq D^- .
\end{equation}

We apply  the Ehresmann fibration lemma 
 to the restriction of the  family
$\Phi_{2r}$ to $\D^{(2r)}_{{\rm reg}}$. 
 We have monodromy 
homeomorphisms, 
$M(T_i): (\Su_{\underline{v}}, \partial \, \Su_{\underline{v}}) \lorw (\Su_{\underline{v}}, \partial \, \Su_{\underline{v}})$, 
for $i=1, \ldots, 2r-1$, which restrict to the identity on the boundary $\partial \, \Su_{\underline{v}}$. 
They are unique up to an isotopy which fixes the boundary pointwise.

Let $\gamma \in H_1(\Su_{\underline{v}}, \partial \, \Su_{\underline{v}})$ be a relative $1$-cycle. 
Since  the monodromy homeomorphisms fix the boundary,  the difference 
$M(T_i)(\gamma) -\gamma$ is homologous to
a cycle, denoted ${\rm Var}_i(\gamma)$, disjoint from the boundary.
This defines  the classical 
{\em variation maps} (see  \ci{ar}, \S 2.1): 
\[ 
{\rm Var}_i: H_1(\Su_{\underline{v}}, \partial \, \Su_{\underline{v}})\lorw 
H_1(\Su_{\underline{v}}), \qquad  i=1, \ldots, 2r-1.
\] 

\begin{proposition}
\label{mndrmydisc}
The following holds:  
$$
{\rm Var}_i(\lambda_j)=\left\{\begin{array}{rl}
\lambda_i    & \hbox{ if } j= i-1\\
-\lambda_i   & \hbox{ if } j=i+1\\ 
0    &  \hbox{ if } j\neq i-1,i+1,
\end{array}\right.
\qquad
{\rm Var}_i(\mu)=\left\{\begin{array}{ll}
0   & \hbox{ if } i\neq 1\\
\lambda_1   & \hbox{ if } i=1.
\end{array}\right.
$$
\end{proposition}
\n{\em Proof.} The monodromy $M(T_i)$
is associated  with  the degeneration  of $\Su_{\underline{v}}$
in which the $i$-th and $(i+1)$-th ramification points come together
and the covering acquires a node. It follows that
$M(T_i)$  is a Dehn twist around $\lambda_i$.
The Picard-Lefschetz formula (\ci{ar}, \S 1.3)   gives: 
\[\hbox{ if } c \in  H_1(\Su_{\underline{v}}, \partial \, \Su_{\underline{v}}), \hbox{ then } {\rm Var}_i(c)=(c, \lambda_i)\lambda_i. \]
We conclude by combining  the above with 
(\ref{intnum1}) and (\ref{intnum2}).
\blacksquare

\subsubsection{The local family}
\label{locmondc}
Let  $d=2r$ be  an even positive integer, and let $\mathfrak a$ be a partition of $d$, which we write
\begin{equation}
\label{mutyp}
{\mathfrak a }=(a_1, \ldots, a_{2\nop+\nep}),
\end{equation}
where $a_1, \ldots, a_{2\nop}$ are odd positive integers and 
$a_{2\nop+1}, \ldots, a_{2\nop+\nep}$ are even positive integers;
we set
\begin{equation}
\label{mutypdai}
 d_0:=0,  \, d_{i}:= \sum_{j=1}^{i}a_j  \,\, \hbox{ for } i=1, \cdots , 2\nop+\nep.
\end{equation}
Clearly $d= d_{2\nop+\nep}$.

Let $\sigma$ be an effective divisor of degree $d$ on a projective nonsingular curve
$\curv$ with {\em multiplicity type $\mathfrak a$}, namely 
\begin{equation}
\label{ddiv}
\sigma=\sum_{i=1}^{2\nop+\nep} a_i q_i \, ,
\end{equation}
where the points $q_1, \ldots, q_{2\nop+\nep}$ of $\curv$ are distinct;

\medskip
Let ${\mathcal O}_{\curv}(\sigma)$  be the corresponding line bundle on $\curv$
and let 
$s \in \Gamma(\curv,{\mathcal O}_{\curv}(\sigma) )$ be  the section
vanishing at  $\sigma$, well-defined up to a non-zero scalar. 
We choose a square root of ${\mathcal O}_{\curv}(\sigma)$, that is, 
a line bundle $L$  on $\curv$ such that $L^{\otimes 2}\simeq {\mathcal O}_{\curv}(\sigma)$.
The double cover $\curv_{\sigma}$ of $\curv$  branched over  $\sigma$ is the curve on  the total space 
${\mathbb V}(L) \stackrel{\pi}{\lorw}\curv$ of $L$ defined by
$$
\{y \in {\mathbb V}(L) \,:\,
y^2 = \pi^* s \}.
$$  
Note that the topology, e.g. its being connected or not,  depends on the choice of 
the square root $L$, and not only on $\sigma$.

From this point on, we work under the following:
\begin{assumption}
\label{basicass}
The double covering 
$\curv_{\sigma }$ is integral. 
\end{assumption}

\bigskip
The effective divisors of degree $2r$ on the curve $\curv$ are parametrized by 
the symmetric product $\curv^{(2r)}$, which is a nonsingular algebraic variety, stratified by the loci 
corresponding to divisors with a fixed multiplicity type.
We denote by $\curv^{(2r)}_{\rm reg}$ the open subset consisting of multiplicity-free divisors
and, for every  subset $Y \subseteq \curv^{(2r)}$, we set $Y_{\rm reg}:=Y \cap \curv^{(2r)}_{\rm reg}$. 
We have the divisor 
$
Z:=\left\{(c,u) \in \curv \times \curv^{(2r)} \,:\, c \in u\right\} \subseteq  \curv \times \curv^{(2r)},
$
the associated line bundle ${\mathcal O}(Z)$ on $\curv \times \curv^{(2r)}$
and its tautological  section $S \in \Gamma (\curv \times \curv^{(2r)}, {\mathcal O}(Z))$ vanishing
 at $Z$.
Given   an open subset  $V \subseteq \curv \times \curv^{(2r)}$, we set $Z_V:=Z \cap V$  and denote by ${\mathcal O}_V(Z)$
and  $S_V $
the restrictions of the corresponding objects to $V$.

The following proposition follows readily from the fact that the squaring map 
$\jac{\curv}^{\, r} \lorw \jac{\curv}^{\, 2r} $ is \'etale:

\begin{proposition}
\label{univfam}
Let $ \mathcal U$ be a connected and simply connected open neighborhood of $\sigma$ in $ \curv^{(2r)}$,
and let $Z_{\mathcal U}:=Z\cap  (\curv \times \mathcal U )$. Then
for every line bundle $L$ on $\curv$ such that $L^{\otimes 2}\simeq {\mathcal O}_{\curv}(\sigma)$,
there is a projective family $\Phi_{\mathcal U}$:
$$
\xymatrix{ &   &   {\mathscr C}_{\mathcal U} \ar[ld]_{\rho_{\mathcal U}} \ar[dd]^{\Phi_{\mathcal U}}        \\
Z_{\mathcal U}  \ar@{^{(}->}[r] &  \curv \times \mathcal U \ar[ld]_{p_1} \ar[rd]^{p_2}  &  \\
\curv &  & \mathcal U 
}$$
with the following properties

\begin{enumerate}
\item
for $u \in {\mathcal U}$, the curve $\curv_u:=\Phi^{-1}_{\mathcal U}(u) \stackrel{\rho_u}{\lorw} \curv$
is a double covering of $\curv$ ramified at the effective divisor $u=Z_{\mathcal U} \cap p_2^{-1}(u)$, and 
$\curv_{\sigma}:=\Phi^{-1}_{\mathcal U}(\sigma) \stackrel{\rho_{\sigma}}{\lorw} \curv$ 
is the double covering of $\curv$ ramified at $\sigma$ corresponding to the choice of the square root $L$.

\item
The map $\rho_{\mathcal U}$ is a double covering  branched over 
$Z_{\mathcal U}$.

\item
The restriction   $\Phi_{\mathcal U_{\rm reg}}: {\mathscr C}_{\mathcal U_{\rm reg}}:=\Phi_{\mathcal U}^{-1}(\mathcal U_{\rm reg} )      \lorw \mathcal U_{\rm reg}$
is a smooth family.

\item
If
$$
\xymatrix{ &   &   {\mathscr C }' \ar[ld]_{\rho '} \ar[dd]^{\Phi ' }        \\
Z'  \ar@{^{(}->}[r] &  \curv \times T  \ar[rd]^{p_2}  &  \\
 &  & T 
}$$
is a family of double coverings of $C$ with $\rho '$ ramified over the divisor $Z'$, and, for $t_0 \in T$, there is an isomorphism
$$
\xymatrix{  
\Phi '^{-1}(t_0) \ar[rd] \ar[rr]^{\simeq}&   &   {\curv }_{\sigma}  \ar[ld]^{\rho _{\sigma}}         \\
                               &  \curv    &  
                               }$$
then, for a suitable neighborhood $V \subseteq T$ of $t_0$,  the map $\theta: V \lorw {\mathcal U}$  associating to  $t \in T$ the branch locus of $\Phi '^{-1}(t) \lorw \curv$,
defines an isomorphism   $\Phi '^{-1}(V) \simeq {\mathscr C}_{\mathcal U} \times_{\theta} V$   over $V$.
 \end{enumerate}
\end{proposition}
\bigskip
We now  define the 
 {\em distinguished neighborhoods} of $\sigma$ in $\curv^{(2r)}$.
 
Choose a closed disc $\overline{\Delta} \subseteq \curv$ whose interior $\Delta$ contains the support of $\sigma$.   
Choose open discs $\Delta_1, \ldots, \Delta_{2\nop+\nep}\subseteq \Delta \subseteq \curv$ so that:

\begin{enumerate}
\item
$q_i \in \Delta_i\, \hbox{ for all }i.$
\item
$\overline{\Delta_i}\subseteq \Delta  \hbox{ for all }i \hbox{ and }   \overline{\Delta_i}\cap \overline{\Delta_j}=\emptyset \, \hbox{ for all }i\neq j$.
\end{enumerate}
As in \S \ref{ufcd}, we have the $a_i$-th symmetric product $\Delta_i^{(a_i)}$  and 
its open subset $\Delta_{i,{\rm reg}}^{(a_i)}$  corresponding to $a_i$-tuples of distinct points.
The  set  of  effective divisors of degree $2r$ consisting of  
$a_i$ points contained in $\Delta_i$, where $i = 1, \ldots, 2\nop+\nep$, 
defines  a {\em distinguished} neighborhood of $\sigma \in \curv^{(2r)}$:
$$
\dist :=\prod_i \Delta_i^{(a_i)} \subseteq \Delta^{(2r)}  \subseteq  \curv^{(2r)}.
$$
Distinguished neighborhoods are contractible and   give rise to  a fundamental 
system of neighborhoods of $\sigma \in \curv^{(2r)}.$
We also have the open subset 
\beq
\label{shreg}
\dist_{{\rm reg}}:= \dist \cap \curv^{(2r)}_{\rm reg}=   \prod_i \Delta_{i,{\rm reg}}^{(a_i)} \subseteq \Delta^{(2r)}_{\rm reg},
\eeq
consisting of the simple, i.e. multiplicity-free, divisors in $\dist$.

\medskip
By Proposition \ref{univfam}, the choice of a square root  $L$ of the line bundle ${\mathcal O}_C (\sigma)$ yields
the  family $\Phi_{\Delta^{(2r)}}:  {\mathscr C}_{\Delta^{(2r)}} \lorw \Delta^{(2r)} $, the smooth family 
$\Phi_{\Delta^{(2r)}}:  {\mathscr C}_{\Delta^{(2r)}_{\rm reg}} \lorw \Delta^{(2r)}_{\rm reg} $,
and their restrictions $\Phi_{\dist}:{\mathscr C}_{\dist} \lorw \dist$  and 
$\Phi_{\dist_{{\rm reg}}}:{\mathscr C}_{\dist_{{\rm reg}}} \lorw \dist_{{\rm reg}}$.

Our aim is  the proof of Theorem \ref{mainest} below.
This result 
is the main step in the proof of Theorem \ref{mainmonthm} which, as we have
seen at the beginning of \S\ref{lbebncj},
completes the proof of the main  Theorem \ref{final} of this section.

\begin{theorem}
\label{mainest}
Let $\curv$ be a nonsingular projective curve of genus $g$, 
let $\sigma \in \curv^{(2r)}$ be an effective divisor of multiplicity type 
${\mathfrak a}=(a_1, \ldots, a_{2\nop+\nep})$, and let 
$L$ be a square root of ${\mathcal O}_{\curv}(\sigma)$
such that the associated  double covering $\rho_{\sigma}:C_{\sigma} \lorw C$ is integral. 
Let $\dist$ be a distinguished neighborhood of $\sigma$, let
$\dist_{{\rm reg}}$ be the open subset of simple divisors in $\dist$ and $j:\dist_{{\rm reg}} \lorw \dist$ the corresponding imbedding.
Then:
\begin{equation}
\label{qequ}
\sum_{l=0}^{4g+2r-2} \dim  \left( R^0j_* \bigwedge^l {R^1 \Phi_{\dist_{\rm reg}}}_* \rat \right)_{\sigma} =\sum_{l=0}^{4g+2r-2} \dim \Gamma\left(\dist_{\rm reg}, \bigwedge^l {R^1 \Phi_{\dist_{\rm reg}}}_* \rat\right) =
  2^{4g-2}  \left( \prod_{i=1}^{2\nop+\nep}(a_i+1)\right).
\end{equation}

\end{theorem}

\begin{remark}
\label{connect}
{\rm Recall the Definition \ref{stfcrvc} of singular type of an integral curve with $A_k$-singularities and of the two sets $O$ and $E$.
With the notation of Theorem \ref{mainest}, we have $E= \{ \rho_{\sigma}^{-1}(q_i)\}$ for $i \in \{1, \cdots, 2\nop\} $ such that $a_i \neq 1$, and  $O= \{ \rho_{\sigma}^{-1}(q_i)\}$  for $2\nop + 1 \leq i \leq 2\nop+ \nep$.
With the convention that 
if $a_i=1$, then the term $a_i-1$ should be deleted from the singular type vector $\underline{k}$, 
we have that $\curv_{\sigma}$ has singular type 
$
\underline{k}=(a_{2\nop+1}-1, \ldots,a_{2\nop+\nep}-1; a_1-1,\ldots , a_{2\nop}-1).
$

\medskip
\n
Then, the right-hand-side of equation \ref{qequ} equals 
the quantity 
$2^{2\widetilde{g}}  \left( \prod_{c \in O}(2\delta_c+1)\right)\left(\prod_{ c \in E}(\delta_c+1)\right)$
associated with the singular curve $C_{\sigma}$ (see Theorem \ref{estbncj}).

For $i=1, \cdots, 2\nop+\nep$, we set $\delta_i:=\delta_{\rho_{\sigma}^{-1}(q_i)}$, with the convention that
$\delta_i=0$ if $a_i=1$, i.e. if $\rho_{\sigma}^{-1}(q_i)$ is a nonsingular point of $C_{\sigma}$. Clearly 
\begin{equation}
\label{changeindexset}
  \left( \prod_{c \in O}(2\delta_c+1)\right)\left(\prod_{ c \in E}(\delta_c+1)\right)=
   \left( \prod_{i=1}^{2\nop}(\delta_i+1)\right)\left(\prod_{i=2\nop+1}^{2\nop
+\nep}(2\delta_i+1)\right).
\end{equation}
By the Riemann-Hurwitz formula
and the definition of  the $\delta$-invariant (\S \ref{cpctjcb}),
we have that 
\begin{equation}
\label{generenormalizz}
\widetilde{g}=2g+ r -1-\sum \delta_i=
2g + \frac{1}{2}\sum_{i=1}^{2\nop+\nep}a_i -1 - 
\frac{1}{2} \sum_{i=1}^{2\nop}(a_i-1) - 
\frac{1}{2}\sum_{i=\nop+1}^{2\nop+\nep} a_i = 2g+\nop -1;
\end{equation}
now  use  the fact that
\begin{equation}
a_i=\left\{\begin{array}{rl}
2\delta_i  & \hbox{ if } a_i \hbox{ is even }\\ 
2\delta_i+1 &  \hbox{ if } a_i \hbox{ is odd},
\end{array}\right.
\end{equation}
and deduce the equality
$$
2^{4g-2}  \left( \prod_{i=1}^{2\nop+\nep}(a_i+1)\right)= 
2^{2\tilde{g}-2\nop}\prod_{i=1}^{2\nop}(2\delta_i+2)\prod_{i=2\nop +1}^{2\nop+\nep}(2\delta_i+1) = 2^{2\tilde{g}}\prod_{i=1}^{2\nop}(\delta_i+1)\prod_{i=2\nop +1}^{2\nop+\nep}(2\delta_i+1),
$$
whose right-hand side coincides, by (\ref{changeindexset}), with 
$2^{2\widetilde{g}}  \left( \prod_{c \in O}(2\delta_c+1)\right)\left(\prod_{ c \in E}(\delta_c+1)\right).$
}
\end{remark}

\bigskip

\n
For $i=1, \cdots , 2\nop+\nep$, let $\underline{u}_i \in \Delta_{i,{\rm reg}}^{(a_i)}$, and
let $\underline{u}=(\underline{u}_1, \cdots ,\underline{u}_{2\nop +\nep})\in \dist_{\rm reg}$.

We have the monodromy representation 
$$
\pi_1(\dist_{\rm reg},\underline{u}) \lorw {\rm Aut} (H_1(\curv_{\underline{u}})),
$$
and  its exterior powers     
$$
\pi_1(\dist_{\rm reg},\underline{u}) \lorw {\rm Aut} \left(\bigwedge^l H_1(\curv_{\underline{u}})\right).
$$
The evaluation map at the base-point  $\underline{u}$  gives an isomorphism:
$$
\Gamma\left(\dist_{\rm reg}, \bigwedge^l R^1 {\Phi_{\dist_{\rm reg}}}_* \rat\right) \stackrel{\simeq}{\lorw}
\left(\bigwedge^l H_1(\curv_{\underline{u}})\right)^{\pi_1}
$$
where $(- )^{\pi_1}$ denotes the subspace of  invariants.

\begin{remark}
\label{moncomprmk}
{\rm Since the family  $\Phi_{  \dist_{{\rm reg}} }:{\mathscr C}_{\dist_{{\rm reg}}} \lorw \dist_{{\rm reg}}$ is the restriction to $\dist_{\rm reg} $ of the family
$\Phi_{ \Delta^{(2r)} _{{\rm reg}} }:{\mathscr C}_{\Delta^{(2r)} _{{\rm reg}}}   \lorw \Delta^{(2r)} _{{\rm reg}}$, its monodromy representation
is the composition
\beq
\label{moncomp}
\pi_1(\dist_{\rm reg},\underline{u}) \lorw \pi_1({\Delta}^{(2r)}_{\rm reg},\underline{u}) \lorw  {\rm Aut} (H_1(\curv_{\underline{u}})).
\eeq
}
\end{remark}

\subsubsection{Proof of Theorem  \ref{mainest}, step 1:  
splitting off the constant part}
\label{splittcp}
We have the diagram
$$
\xymatrix{    {\mathscr C}_{ \Delta^{(2r)}} \ar[rd]_{\rho_{\Delta^{(2r)}}} \ar[rr]^{\Phi_{\Delta^{(2r)} }}    &  &  \Delta^{(2r)}      \\
   & \curv \times   \Delta^{(2r)}    \ar[ru]^{p_2}  &  
}$$
of Proposition \ref{univfam} and the
nonsingular branched double covering
\begin{equation}
\label{basedblcvr}
\rho_{\underline{u}}: \Phi_{ \Delta^{(2r)}    }^{-1}(\underline{u})=:
\curv_{\underline{u}} \lorw \curv. 
\end{equation}
By the Riemann-Hurwitz formula we have $g(\curv_{\underline{u}})=2g +r-1,$ where, we remind the reader, $r=\frac{\deg\, \sigma}{2}=\frac{\sum_i a_i}{2}$.
We set 
$$
\widehat{\curv} :=\rho_{\underline{u}}^{-1}(\curv \setminus \Delta), \qquad 
\Xi_{\underline{u}}:=\rho_{\underline{u}}^{-1}(\overline{\Delta}).
$$
\medskip
\begin{remark}
\label{gluefamily}
{\rm In view of Remark \ref{kindergarten}, the inverse image 
$\rho_{\underline{u}}^{-1}(\partial \overline{\Delta})=\partial \, \Xi_{\underline{u}}$ consists of
two connected  components.
There are two distinct possibilities for the
restriction of the covering $\rho_{\underline{u}}$ to $\widehat{\curv}$.
The former is that this restricted covering is
disconnected and thus  biholomorphic to two copies of $\curv \setminus \Delta $. This is the case 
if the square root $L$ of ${\mathcal O}_C (\sigma)$ is a trivial line bundle on $\curv \setminus \Delta$.
The latter, corresponding to the case in which $L$ is a non-trivial line bundle on $\curv \setminus \Delta$, is that 
$\widehat{\curv}$ is connected,
in which case    $\widehat{\curv}= C' \setminus (U_1 \coprod U_2)$ is obtained by removing two discs $U_1, U_2$ 
from a connected compact Riemann surface $C'$ of genus $2g-1$. 
}
\end{remark}

\medskip
Since the line bundle associated with a divisor on $\curv$ supported on $\Delta$ is trivial on $\curv \setminus \Delta$, we have a biholomorphism (of surfaces with boundaries)
$ \rho_{\Delta^{(2r)}}^{-1}((\curv \setminus {\Delta}) \times \Delta^{(2r)}) \simeq \widehat{\curv} \times \Delta^{(2r)}$, 
 and
the family $\Phi_{\Delta^{(2r)}}:  {\mathscr C}_{\Delta^{(2r)}} \lorw \Delta^{(2r)} $ 
is obtained by glueing   the family $ \rho_{\Delta^{(2r)}}^{-1}(\overline{\Delta} \times \Delta^{(2r)})    \lorw \Delta^{(2r)}$ to the constant family   $\widehat{\curv}  \times  \Delta^{(2r)} \lorw \Delta^{(2r)}$ along the boundary $ (S^1 \coprod S^1)\times  \Delta^{(2r)} $;
the same clearly applies to its restrictions $\Phi_{\dist} , \Phi_{\dist_{{\rm reg}}}$.

\bigskip
The long exact sequence of relative cohomology of the pair $\Xi_{\underline{u}} \subseteq \curv_{\underline{u}}$, 
 the vanishing $H_2(\Xi_{\underline{u}})=0$ and the fact that 
$H_0(\Xi_{\underline{u}}) \lorw H_0(\curv_{\underline{u}})$ is an isomorphism, give
the exact sequence
\begin{equation}
\label{lexrelcoh}
\xymatrix{
0 \ar[r] & H_2(\curv_{\underline{u}}) \ar[r] & H_2(\curv_{\underline{u}}, \Xi_{\underline{u}}) \ar[r] & H_1(\Xi_{\underline{u}}) \ar[r] & H_1(\curv_{\underline{u}}) \ar[r] & 
H_1(\curv_{\underline{u}}, \Xi_{\underline{u}}) \ar[r] & 0 \\
         &                                   &    H_2(\widehat{\curv}, \partial \, \widehat{\curv}) \ar[u]_{\simeq}    &    &      & H_1(\widehat{\curv}, \partial \, \widehat{\curv}) \ar[u]_{\simeq}        &
}
\end{equation}
where the vertical arrows indicate the  excision isomorphisms.

\bigskip
\n
{\em Case 1: $\widehat{\curv}$  is disconnected.}

\n
In this case $\dim H_2(\curv_{\underline{u}}, \Xi_{\underline{u}}) =2$, and 
$\dim H_1(\curv_{\underline{u}}, \Xi_{\underline{u}}) \simeq H_1(\curv )^{\oplus 2}=4g$;
define
\begin{equation}
\label{Hdisco}
H_{\mathfrak a,{\rm disc}}:=\im \{H_1( \Xi_{\underline{u}}) \lorw  H_1(\curv_{\underline{u}}) \}.
\end{equation}
It follows from the sequence (\ref{lexrelcoh}) that $\dim H_{\mathfrak a,{\rm disc}}=2r-2$, 
and we have an exact sequence
\begin{equation}
\label{exseqCdisc}
0 \lorw H_{\mathfrak a, {\rm disc}} \lorw H_1(\curv_{\underline{u}}) \lorw H_1(\curv_{\underline{u}}, \Xi_{\underline{u}}) \lorw 0.
\end{equation}

\begin{remark}
\label{therareoddpoints}
{\rm In this case, in order to satisfy the assumption \ref{basicass},
we must have $\nop\geq 1$. 
In fact, if $\nop=0$, every singular point of $C_{\sigma}$ has two branches.
Let  $\nu: \widetilde{\curv_{\sigma}} \lorw \curv_{\sigma}$ be the 
normalization map. Since the inverse image by $\nu$ of every singular points consists of two points, the composition 
$\widetilde{\curv_{\sigma}} \stackrel{\nu}{\lorw} \curv_{\sigma} \lorw \curv$ is an \'etale covering, which must be trivial if $\widehat{\curv}$  is disconnected. 
This implies that $\widetilde{\curv_{\sigma}}$ is disconnected
 and $\curv_{\sigma}$ is reducible
against the assumption \ref{basicass}. See also \ci{ngo2} \S 11.
}
\end{remark}

\begin{remark}
\label{trivcycl}
{\rm Given a basis  
of $H_1(\curv)$, it is possible to represent its elements by cycles contained
in $\curv \setminus \Delta $. By taking pre-images  of these cycles via $\rho_{\underline{u}}$, we get $4g$
  linearly independent homology classes which split the exact sequence \ref{exseqCdisc}.
Since, as we have  already observed, the family $\Phi_{\Delta^{(2r)}}:  \curv_{\Delta^{(2r)}} \lorw \Delta^{(2r)} $, 
is obtained by glueing the constant family with the family 
$ \rho_{\Delta^{(2r)}}^{-1}(\overline{\Delta} \times \Delta^{(2r)})          \lorw \Delta^{(2r)}$, the direct sum decomposition 
\begin{equation}
\label{splitCdisc}
 H_1(\curv_{\underline{u}}) = H_{\mathfrak a,{\rm disc}} \oplus H_1(\curv )^{\oplus 2} 
 \end{equation} 
is invariant under the action of $\pi_1(\Delta^{(2r)}_{\rm reg}, \underline{u})$, which  is trivial on the second summand. 
}
\end{remark}

\bigskip
\n
{\em Case 2: $\widehat{\curv}$ is connected.} 

\n
In this case  $\dim H_2(\curv_{\underline{u}}, \Xi_{\underline{u}}) =1$ and $\dim H_1(\curv_{\underline{u}}, \Xi_{\underline{u}}) =4g-1$;
the sequence (\ref{lexrelcoh}) takes the form:
\begin{equation}
\label{exseqCconn}
0 \lorw  H_1(\Xi_{\underline{u}}) \lorw H_1(\curv_{\underline{u}}) \lorw H_1(\curv_{\underline{u}}, \Xi_{\underline{u}})= H_1(\widehat{\curv}, \partial \, \widehat{\curv})\lorw  0.
\end{equation}

As we already observed in Remark \ref{gluefamily},  $\widehat{\curv}= C' \setminus ( U_1 \coprod U_2)$ is obtained by removing two discs $U_1, U_2$
from a connected compact Riemann surface $C'$ of genus $2g-1$; 
 it is readily seen that the map
$H_1(C') \lorw  H_1(C', \partial \overline{U_1} \coprod \partial \overline{U_2})$ is injective.
We have  the excision isomorphism
$H_1(C', \partial \overline{U_1} \coprod \partial \overline{U_2}) =H_1(\widehat{\curv}, \partial \, \widehat{\curv})$,
by which we  identify $H_1(C')$ with a subspace of $H_1(\widehat{\curv}, \partial \, \widehat{\curv}) $.
Let $\widehat{\gamma} \in H_1(\widehat{\curv}, \partial \, \widehat{\curv})$ be the class
of a path in $\widehat{\curv}$ joining the two connected components of its boundary. It is then easy to see that  
$H_1(\widehat{\curv}, \partial \, \widehat{\curv})= H_1(C')  \bigoplus {\rm Span \, \widehat{\gamma}}$.
By using the excision isomorphism $ H_1(\curv_{\underline{u}}, \Xi_{\underline{u}}) = H_1(\widehat{\curv}, \partial \, \widehat{\curv})$,
we obtain an isomorphism $ H_1(\curv_{\underline{u}}, \Xi_{\underline{u}})= H_1(C')  \bigoplus {\rm Span \, \widehat{\gamma}}.$

The natural map $H_1(\widehat{\curv}) \lorw H_1({\curv}')$ is clearly surjective, as
every class in $H_1(C')$ can be represented by cycles contained in  $\widehat{\curv}$. 
By using this fact, we choose a (non-canonical) splitting $ H_1({\curv}') \lorw H_1(\widehat{\curv}) $. 
An easy argument, based on the Mayer-Vietoris exact sequence  associated with the 
decomposition $\curv_{\underline{u}}= \widehat{\curv} \cup \Xi_{\underline{u}}$, shows that the map
$H_1(\widehat{\curv}) \lorw H_1(\curv_{\underline{u}})$ is injective. Via the composition 
$H_1({\curv}') \lorw H_1(\widehat{\curv}) \lorw H_1(\curv_{\underline{u}})$,
we may then identify $H_1( \curv ') $  with a subspace of $H_1(C_{\underline{u}})$. 
The lack of canonicity of this identification will be harmless for what follows.

\subsubsection{Proof of Theorem  \ref{mainest}, step 2:  
construction of an adapted basis }
Let us choose a differentiable imbedding  $\beta:\I =[0, 2r+1]\lorw \overline{\Delta}$ with the following properties:
\begin{enumerate}
\item
$\beta(\I)\cap \partial \overline{\Delta}=\{\beta(0), \beta(2r+1)\}$. 
\item
$\underline{u}_i= \{  \beta(d_{i-1}+1), \cdots , \beta(d_i)  \}$ for $i=1, \cdots , 2\nop+\nep$, with $d_i$ defined in (\ref{mutypdai}).
\item
For every $i=1, \cdots , 2\nop+\nep $, the inverse images   $\beta^{-1}(\beta(\I)\cap \overline{\Delta_i})$ are closed sub-intervals of $\I$.
\end{enumerate}
As in  \S \ref{dcd} and \S\ref{ufcd}, $\beta$ defines the cycles $\cl{}{i} \in H_1(\Xi_{\underline{u}}), \, \mu \in H_1(\Xi_{\underline{u}}, \partial\, \Xi_{\underline{u}})$,
and the set $\{T_i\}$ of generators  of ${\mathscr B}^{2r}$.

The open imbedding 
$\dist_{\rm reg} \lorw \Delta^{(2r)}_{\rm reg}$ induces the group homomorphism
$ \pi_1(\dist_{\rm reg},\underline{u}) \lorw \pi_1({\Delta}^{(2r)}_{\rm reg},\underline{u}) = {\mathscr B}^{2r}$.
It is evident from the definition of $\Delta_i$ and $\beta$ that, if $d_i<j<d_{i+1}$, then $T_j$ can be represented by a pair of curves as in (\ref{genbraids}), 
whose image is entirely contained in $\Delta_i$, and is hence contained in the image of the homomorphism
above,  whereas this is not possible if  
$j \in \{ d_1, \ldots, d_{2\omega+ \epsilon-1} \}$. This observation readily implies the following lemma; 
the missing details of the proof are left  to the reader:

\begin{lemma}
\label{imgfg}
The map 
$$ \pi_1(\dist_{\rm reg},\underline{u}) \lorw \pi_1({\Delta}^{(2r)}_{\rm reg},\underline{u}) = {\mathscr B}^{2r}   $$
is injective. Its image is  the subgroup  ${\mathscr B}^{\mathfrak a}$ of 
${\mathscr B}^{2r}$ generated  by 
the elements $T_j$'s for  $j \in \{ 1, \ldots, 2r-1\}\setminus \{ d_1, \ldots,
d_{2\omega+ \epsilon-1} \}$.  
\end{lemma}

\begin{remark}
\label{constant}
{\rm It follows from Lemma \ref{imgfg} that if 
$\dist '\subseteq \dist$ is another distinguished neighborhood, 
and $\underline{u} \in \dist '$, then the natural map  
$
\pi_1(\dist_{\rm reg}',\underline{u})\lorw \pi_1(\dist_{\rm reg},\underline{u})
$ is an isomorphism, and 
$\Gamma\left(\dist_{\rm reg}, \bigwedge^l {R^1 \Phi_{\dist_{\rm reg}}}_* \rat\right) \lorw \Gamma\left( \dist ' _{\rm reg}, \bigwedge^l R^1 {\Phi_{\dist'_{\rm reg}}}_* \rat\right)
$ is an isomorphism. Hence the natural  maps:
$$
\Gamma\left(\dist_{\rm reg}, \bigwedge^l {R^1 \Phi_{\dist_{\rm reg}}}_* \rat\right) \lorw \Gamma\left( \dist ' _{\rm reg}, \bigwedge^l R^1 {\Phi_{\dist'_{\rm reg}}}_* \rat\right)
\lorw \left( R^0j_* \bigwedge^l {R^1 \Phi_{\dist_{\rm reg}}}_* \rat 
\right)_{\sigma}
$$
are isomorphisms. }
\end{remark}

\bigskip
In the case in which $\widehat{\curv}$ is disconnected, the kernel of the map $  H_1(\Xi_{\underline{u}}) \lorw H_1(\curv_{\underline{u}}) $
 is generated by the element $\sum_{l=1}^{d} \lambda_{2l-1} $, see (\ref{relbdry}).
Since, by Remark \ref{therareoddpoints}, $\nop \geq1$,    
we may use this relation  to eliminate $\lambda_{a_1}$, so that the set $\{ \cl{}{i}\}$, for $i=1, \cdots , 2r-1,\, i \neq a_1$, is a basis for the space $H_{\mathfrak a,{\rm disc}}=\im \, H_1(\Xi_{\underline{u}}) \lorw H_1(\curv_{\underline{u}}) .$   This choice is suggested by Lemma \ref{imgfg}
since $T_{a_1} \notin {\mathscr B}^{\mathfrak a}$. It will be evident in \S \ref{compmoninv} that this choice is computationally quite convenient.

In force of Remark \ref{trivcycl}, Remark \ref{moncomprmk} and Lemma \ref{imgfg}, we are reduced to compute the dimension of the subspace of invariants
of $\extal{H_{\mathfrak a,{\rm disc}}}$ for the action of the group ${\mathscr B}^{\mathfrak a} \subseteq {\mathscr B}^{2r}\simeq
\pi_1(\Delta^{(2r)}_{\rm reg}, \underline{u})$ defined in Lemma \ref{imgfg}.

\bigskip
 We now deal with the case in which $\widehat{\curv }$ is connected; we resume the notation introduced at the end of \S \ref{splittcp}. Recall in particular the relative cycle 
$\widehat{\gamma}$, 
the identification of $H_1(C')$ with a subspace of $ H_1(\curv_{\underline{u}}, \Xi_{\underline{u}})$ and the non-canonical one with a subspace of $H_1(C_{\underline{u}})$. 
We lift the relative class $\widehat{\gamma}$ to a homology class 
$\cl{}{0} \in H_1(\curv_{\underline{u}})$  by joining 
$\widehat{\gamma}$ with a representative of the relative cycle $\mu$ defined by (\ref{defimu}) in $\Xi_{\underline{u}}$. 
Setting $H_{\mathfrak a, {\rm conn}} :=H_1(\Xi_{\underline{u}}) \bigoplus {\rm Span \, \cl{}{0}}$, the decomposition 
\begin{equation}
\label{splitCconn}
H_1(\curv_{\underline{u}})= H_{\mathfrak a,{\rm conn}}  \bigoplus H_1(C')
\end{equation} 
is $\pi_1(\Delta^{(2r)}_{\rm reg}, \underline{u})\simeq{\mathscr B}^{2r}$-invariant. Note that, by (\ref{intnum2}), the action on $\cl{}{0}$ is given by:
\begin{equation}
\label{azionelambdazero}
T_i(\cl{}{0})=\cl{}{0} \hbox{ if } i\neq1 \hbox{ and }  T_1(\cl{}{0})=\cl{}{0}+\cl{}{1}.
\end{equation}
Since, by construction, the cycles in the subspace $H_1(C')$ may be chosen to be entirely contained in $\widehat{\curv}$,
 the action of ${\mathscr B}^{2r}$
on the summand $H_1(C')$  is trivial, and, by  Remark \ref{moncomprmk} and Lemma \ref{imgfg}, 
we are reduced 
to compute the dimension of the subspace of invariants
of $\extal{H_{\mathfrak a,{\rm conn}}}$ for the action of the group ${\mathscr B}^{\mathfrak a}\subseteq {\mathscr B}^{2r}\simeq
\pi_1(\Delta^{(2r)}_{\rm reg}, \underline{u})$.

\bigskip
\n
In either case, a local coordinate $\zeta: \overline{\Delta} \lorw \D$, defined  on an open set containing $\overline{\Delta}$,  identifies
the family $ \rho_{\Delta^{(2r)}}^{-1}(\overline{\Delta} \times \Delta^{(2r)})  \lorw \Delta^{(2r)}$ with the family 
$\Phi_{2r}: {\mathscr S}_{2r}  \lorw    \D^{(2r)}$ of \S \ref{ufcd}, and 
the restriction ${\rho_{\underline{u}}}_{|\Xi_{\underline{u}}} : \Xi_{\underline{u}} \lorw \overline{\Delta}$ with the double covering 
$\Su_{\underline{v}} \lorw \overline{\D}$, where 
$\underline{v}:=\zeta(\underline{u})$.
By (\ref{mndrmydisc}), the action  of ${\mathscr B}^{\mathfrak a}$ on $H_{\mathfrak a,{\rm conn}}$ and $H_{\mathfrak a,{\rm disc}}$ is then given by
\begin{equation}
\label{trccmnd}
T_i(\lambda_j)=\lambda_j \;\hbox{ if } |i-j| \neq1, \quad  T_i(\lambda_{i + 1})= \lambda_{i + 1} - \lambda_i, \quad   
 T_i(\lambda_{i - 1})= \lambda_{i - 1} + \lambda_i   .
\end{equation}

\subsubsection{Proof of Theorem  \ref{mainest}, step 3:  computation of monodromy invariants}
\label{compmoninv}

Lemma \ref{linealg}  and Proposition \ref{justlinearalg} below summarize the 
linear algebra facts which we  need to complete the proof of  Theorem \ref{mainest}
\begin{lemma}
\label{linealg}
Let $ U$ be a vector space of even dimension $2m$ with basis $c_1, \cdots, c_{2m}$, and denote by $\extal{ U}$ be its exterior algebra.
Let $T_1, \cdots ,T_{2m} \in \Aut ( U)$ be defined by
\begin{equation} 
\label{opt}
T_i(c_j)=c_j \; \hbox{ if } |i-j| \neq 1, \quad    T_i(c_{i+1})=c_{i+1}-c_i,
\quad T_i(c_{i-1})=c_{i-1}+c_i,
\end{equation}
and denote their natural extensions to $\extal{ U}$ again by $T_i$.
For $I \subseteq \{1, \cdots, 2m\}$,  let $T_I$ be  the subgroup  of  $\Aut (\extal {U})$  generated by the $T_i$'s with $i \in I$, and
denote by $\left(\extal { U}\right)^{T_I} \subseteq \extal { U}$ the subspace of $T_I$-invariants. 
\begin{enumerate} 
\item
For $I=\{1, 2, \cdots, 2m\}$, we have $\dim \left(\extal { U}\right)^{T_{I}}= m+1$.
\item
For $I'= \{2, 3, \cdots, 2m\} $, we have $ \dim    \left(\extal { U}\right)^{T_{I'}} = 2m+1$.
\item 
For $I''=\{2,  \cdots, t, t+2, \cdots , 2m\} $, with $t$ odd, we have $\dim \left(\extal { U}\right)^{T_{I''}}= (t+1)(2m-t+1)$.
\end{enumerate}
\end{lemma}
\n{\em Proof.}
For $a,b \in \{1, \cdots, 2m\}$, with $a \leq b$ and $a \equiv b (2)$, we set
$$
c_{[a,b]}:=c_a+c_{a+2}+\cdots + c_{b-2}+c_b.
$$
It immediately follows from (\ref{opt}) that
\begin{equation}
\label{invint}
T_i \left( c_{[a,b]}\right) =\left\{\begin{array}{ll}
c_{[a,b]}  & \hbox{ if } i\neq a-1, b+1, \\ 
c_{[a,b]}-c_{a-1} &  \hbox{ if } i=a-1, \\
c_{[a,b]}+c_{b+1} &  \hbox{ if } i=b+1.
\end{array}\right.
\end{equation}

\begin{itemize}
\item
Case $I=\{1, \cdots, 2m\}$: a direct computation using (\ref{invint}) shows that
$$
\Omega:= \sum_{s=1}^{m} c_{[1, 2s-1]}\wedge c_{2s}  \in \bigwedge^2  U
$$
is  $T^I$-invariant and $\Omega^{m} \neq 0$. Hence $1, \Omega, \Omega^2, \cdots ,\Omega^{m}$ are the desired $m+1$ $T^I$-invariants. 
\item
Case $I'= \{2, \cdots, 2m\} $. Since $T_{I'} < T_{I}$,
the $\Omega^t$'s introduced  above are  $T_{I'}$-invariant.
Furthermore, since $T_1 \notin T_{I'}$, it follows from (\ref{invint}) that $c_{[2,2m]} \in  U^{T_{I'}}$.
Then $1, c_{[2,2m]}, \Omega, c_{[2,2m]} \wedge \Omega, \cdots , c_{[2,2m]}\wedge \Omega^{m-1}, \Omega^{m},$ 
give the desired $2m+1$ $T^{I'}$-invariants, which, being non-zero and of different degrees, are linearly independent.
\item
Case $I''=\{2, \cdots, t, t+2, \cdots , 2m\} $ with $t$ odd. 
In addition to the $T^{I'}$-invariant $c_{[2,2m]} \in  U$ introduced above, 
we have $c_{[1,t]}  \in U^{T_{I''}} $, again by (\ref{invint}), since $T_{t+1} \notin T_{I''}$.

Let $U_0$ be the space spanned by $c_{[2,2m]}$ and $c_{[1,t]}$, and set  
$$ 
U_1:=  {\rm Span} \{ c_2, c_3, \cdots, c_{t} \} \hbox{ and }  U_2:={\rm Span} \{ c_{t+2}, c_{t+3}, \cdots, c_{2m} \} .
$$ 
It results from (\ref{opt}), and again from the fact that $T_{t+1} \notin T_{I''}$, that
the direct sum decomposition $ U=  U_0 \oplus  U_1 \oplus  U_2$ is $T_{I''}$-invariant. 

\medskip
Let $G_1$  be the group generated by  $ \{ T_2, T_3, \cdots, T_{t}\},$
and let
$G_2 $ be the group generated by $ \{ T_{t+2}, T_{t+3}, \cdots, T_{2m}\}.$

Applying  case 1 of this Lemma to the vector spaces $ U_1$ and $ U_2$ with the groups $G_1, G_2$ respectively, gives 

$$
\Omega_1^k \in \left( \bigwedge^{2k}  U_1\right)^{G_1}
\qquad \forall \;  0\leq k \leq \frac{1}{2}(t-1),
$$ 
and 
$$
\Omega_2^l \in \left(\bigwedge^{2l}  U_2\right)^{G_2} \qquad\forall \; 0\leq l \leq \frac{1}{2}(2m-t-1).
$$
Since $G_1$ acts trivially on $ U_2$, and 
$G_2$ acts trivially on $ U_1$, the $\frac{1}{4}(t+1)(2m-t+1)$ elements
$\Omega_1^k \otimes \Omega_2^l \in \bigwedge^{2k}  U_1\otimes \bigwedge^{2l}  U_2$ are $T_{I''}$ invariant. 
They  are, furthermore,  linearly independent, since they are non-zero and 
live in different summands of the direct sum decomposition of 
$\extal{\left( U_1\oplus U_2 \right)}$.

From the $T_{I''}$-isomorphism 
$$\extal{U} \simeq \left(\extal{ U_0} \right)\otimes \left(\extal{ U_1}\right)\otimes \left(\extal { U_2}\right),$$
we conclude that $\left(\extal { U}\right)^{T_{I''}}$ is a free $\frac{1}{4}(t+1)(2m-t+1)$-rank  module over the four-dimensional  $T_{I''}$-invariant 
algebra  $\extal { U_0}$, hence its dimension is $(t+1)(2m-t+1)$.
\end{itemize}
In all of the three cases considered, it is not hard to verify that there is no other invariant.
\blacksquare

\bigskip
Let ${\mathfrak a }$ be a partition of $d$, with associated integers $a_i, \nop, \nep , d_i$ as in (\ref{mutyp}), (\ref{mutypdai}), and let ${\mathscr B}^{\mathfrak a}$  
be the group of Lemma \ref{imgfg}. Let 
\beq
\label{vectsp1}
V_{\mathfrak a, {\rm disc}} \hbox{ be the }\rat \hbox{-vector space generated by the set }
I_{\rm disc}= \{ 1, \cdots, d-1\} \setminus \{a_1\},
\eeq
and let
\beq
\label{vectsp2}
V_{\mathfrak a, {\rm conn}} \hbox{ be the }\rat \hbox{-vector space generated by the set }
I_{\rm conn}= \{ 0, \cdots, d-1\}.
\eeq

In either case, denote by $\{  \lambda_i \}_{i \in I}$ the corresponding basis, with $I=I_{\rm disc}$ or $I=I_{\rm conn}$ and
endow the vector spaces and their exterior algebras $\extal{ V_{\mathfrak a, {\rm disc}}}$ and 
$\extal{ V_{\mathfrak a, {\rm conn}}}$  with the  ${\mathscr B}^{\mathfrak a}$-module structure defined by
(\ref{trccmnd}).

\begin{proposition}
\label{justlinearalg}
Let $(-)^{{\mathscr B}^{\mathfrak a}}$ denote the subspace of ${\mathscr B}^{\mathfrak a}$-invariants. We have   
\beq
\label{disccase}
\dim \left(\extal{ V_{\mathfrak a, {\rm disc }}}\right)^{{\mathscr B}^{\mathfrak a}} = \frac{1}{4}\prod (a_i+1).
\eeq 
\beq
\label{conncase}
\dim \left(\extal{ V_{\mathfrak a, {\rm conn }}}\right)^{{\mathscr B}^{\mathfrak a}} = \prod (a_i+1).
\eeq 

\end{proposition}

\n{\em Proof.} 
We first prove (\ref{disccase}), starting with the case $\nep=0$.
We proceed by induction on $\nop$.

\smallskip
\begin{itemize}
\item
{\em Assume $\nop$=1.} 
Then ${\mathfrak a}=(a_1, a_2)$ and 
$$
{\mathscr B}^{\mathfrak a}={\mathscr B}^{a_1} \times {\mathscr B}^{a_2} \hbox{ is the group generated by }
T_1, \cdots , T_{a_1-1} , T_{a_1+1}, \cdots , T_{a_1+a_2-1}.
$$ 
Since $T_{a_1} \notin {\mathscr B}^{\mathfrak a}$, it follows from (\ref{trccmnd}) that the direct sum decomposition
$$
V_{\mathfrak a, {\rm disc}}=W_1\bigoplus W_2, \;\;
\hbox{ with }
W_1={\rm Span}\{\lambda_1, \cdots, \lambda_{a_1-1}\}, \quad W_2={\rm Span}\{\lambda_{a_1+1}, \cdots, \lambda_{a_1+a_2-1}\} ,
$$
is  ${\mathscr B}^{\mathfrak a}$-invariant.
Since furthermore, ${\mathscr B}^{a_1}$ acts trivially on $ \extal{W_2}$ and
${\mathscr B}^{a_2}$ acts trivially on $ \extal{W_1}$, we have 
$$
\left( \extal {V_{\mathfrak a, {\rm disc}}} \right)^{{\mathscr B}^{\mathfrak a}}\simeq 
\left( \extal{W_1}\right)^{{\mathscr B}^{\mathfrak a}} \bigotimes  \left(\extal{W_2}\right)^{{\mathscr B}^{\mathfrak a}}
\simeq
\left( \extal{W_1}\right)^{{\mathscr B}^{a_1}} \bigotimes  \left(\extal{W_2}\right)^{{\mathscr B}^{a_2}}.
$$
We now apply twice Lemma \ref{linealg}, case 1, setting first $U=W_1$ and $2m=a_1-1$, and then $U=W_2$, and $2m=a_2-1$,
to find $\dim \left( \extal {V_{\mathfrak a, {\rm disc}}}\right)^{{\mathscr B}^{\mathfrak a}}= \frac{1}{4}(a_1+1)(a_2+1)$.

\item
Assume the statement is proved for every multiplicity type ${\mathfrak a}$ with $\nep=0$ and $\nop \leq k$, and let 
${\mathfrak a}'= (  {\mathfrak a}  , a_{2k+1}, a_{2k+2})$, with ${\mathfrak a}:=(a_1, \cdots, a_{2k})$.
Set $d':=  d+a_{2k+1}$ and $d'':=d+a_{2k+1} +a_{2k+2}$. 
We have 
${\mathscr B}^{{\mathfrak a}'} \simeq {\mathscr B}^{{\mathfrak a}} \times {\mathscr B}^{\mathfrak b}$, where 
$$
{\mathscr B}^{\mathfrak b} \simeq  {\mathscr B}^{a_{2k+1}}\times {\mathscr B}^{a_{2k+2}} \hbox{ is the subgroup generated by }
T_{d+1}, \cdots , T_{d'-1},T_{d'+1}, \cdots T_{d''-1}.
$$
The subspace $V_{{\mathfrak a, {\rm disc}}}= {\rm Span }\{ \cl{}{i}\}_{\stackrel{i \in \{1, \cdots , d-1\},}{ i\neq a_1}}  \subseteq V_{{\mathfrak a}'}$ is ${\mathscr B}^{{\mathfrak a}'}$-invariant,
and the subgroup ${\mathscr B}^{\mathfrak b}$ acts trivially on it by (\ref{trccmnd}). Hence 
$$
\left(\bigwedge V_{{\mathfrak a, {\rm disc}}}\right)^{{\mathscr B}^{{\mathfrak a}'}}=\left(\bigwedge V_{\mathfrak a, {\rm disc}}\right)^{{\mathscr B}^{\mathfrak a}} \hbox{ and }
\dim \left(\bigwedge V_{\mathfrak a, {\rm disc}}\right)^{{\mathscr B}^{\mathfrak a}}=\frac{1}{4}\prod_{i=1}^{2k} (a_i+1),
$$
by the inductive hypothesis.

The subspace ${\rm Span } \{\cl{}{d}, \cdots, \cl{}{d''-1}\}$, however, is not ${\mathscr B}^{{\mathfrak a}'}$ invariant as $T_{d-1}(\cl{}{d})=\cl{}{d}-\cl{}{d-1}$.
We correct this by introducing $\widehat{\lambda}_{d}:=  \cl{}{d}+ \cl{}{d-2}+ \cdots +  \cl{}{d_{2k-1}+1} $;
by (\ref{trccmnd}) we have $T_j(\widehat{\lambda}_{d})=\widehat{\lambda}_{d}$ if  $j \neq d_{2k-1}, d+1$.
Since $T_{d_{2k-1}} \notin {\mathscr B}^{{\mathfrak a}'}$, while $T_{d+1}(\widehat{\lambda}_{d})=\widehat{\lambda}_{d}+\lambda_{d+1}$,  the subspace 
$$W:= {\rm Span}\{\widehat{\lambda}_{d}, {\cl{}{d+1}}, \cdots,    \cl{}{d''-1}  \}$$ is ${\mathscr B}^{{\mathfrak a}'}$-invariant.  
The decomposition $V_{{\mathfrak a}',{\rm disc}}=V_{{\mathfrak a},{\rm disc}} \oplus W$ is hence ${\mathscr B}^{{\mathfrak a}'}$-invariant and 
${\mathscr B}^{{\mathfrak a}}$ acts trivially on $W$. Since $T_d \notin {\mathscr B}^{{\mathfrak a}'}$, we can apply  case 3 of lemma \ref{linealg} to $ U=W$ with $t=a_{2k+1}$.
The statement is now proved for every  ${\mathfrak a}$ such that $\nep=0$.

\end{itemize}
{\em Case $\nep>0$.}
Assume the statement is proved for every ${\mathfrak a}$ with $\nep \leq k$. 
Let ${\mathfrak a}':=({\mathfrak a}, a_{2\nop+\nep})$ and let
 $d':=  d+a_{2\nop+\nep}$.
Just as in the case above, we set 
$$W:= {\rm Span}\{\widehat{\lambda}_{d}, {\cl{}{d+1}}, \cdots,    \cl{}{d'-1}  \},$$ 
where $\widehat{\lambda}_{d}:= \cl{}{d}  + \cl{}{d-2}+ \cdots +  \cl{}{d_{2\nop -1}+1}  $, we have a 
$ {\mathscr B}^{{\mathfrak a}'}$-invariant decomposition
$V_{{\mathfrak a}',{\rm disc}}=V_{{\mathfrak a},{\rm disc}} \oplus W$ with the property that ${\mathscr B}^{{\mathfrak a}}$ acts trivially on  $W$
and ${\mathscr B}^{a_{2\nop+\nep}}$ acts trivially on $V_{{\mathfrak a},{\rm disc}}$; 
we may thus apply case 2 of lemma \ref{linealg} to $U=W$ with $2m=a_{2\nop+\nep}$.

\bigskip
\n
The proof of (\ref{conncase}) goes along the same lines as that of (\ref{disccase}), so we will skip some details. 
We proceed by induction on $\nop +\nep$.
\begin{itemize}
\item
Assume $\nep=1,$ $\nop=0$. Then ${\mathfrak a}=(a_1)$ with $a_1$ even,  and $V_{\mathfrak a, {\rm conn }}$ is generated by $\lambda_0, \cdots , \lambda_{a_1-1}$
with the action of the group generated by $T_1, \cdots, T_{a_1-1}$. This is, up to an obvious renumbering, precisely case 2 of Lemma \ref{linealg}, which gives 
$\dim (\extal{ V_{\mathfrak a, {\rm conn }}})^{{\mathscr B}^{\mathfrak a}}=a_1+1$.
\item 
Assume instead $\nop=1,$ $\nep=0$.
Then ${\mathfrak a}=(a_1, a_2)$ with $a_1, a_2$ odd, $d=a_1+a_2$, and $V_{\mathfrak a, {\rm conn }}$ is generated by $\lambda_0, \cdots , \lambda_{d-1}$
with the action of the group generated by $T_1, \cdots, T_{a_1-1},T_{a_1+1}, \cdots, T_{d-1}$. This is, up to an obvious renumbering, case 3 of Lemma \ref{linealg}
with $t=a_1$, $2m=a_1+a_2$, and we obtain $\dim (\extal{ V_{\mathfrak a, {\rm conn }}})^{{\mathscr B}^{\mathfrak a}}=(a_1+1)(a_2+1)$. 
\item
Assume the statement is proved for all $\mathfrak a$ with $\nop+\nep\leq k$. Given ${\mathfrak a}'$ with $\nop+\nep=k+1$, one needs to consider two cases:
\begin{itemize}
\item 
${\mathfrak a}':=({\mathfrak a}, a)$ with $a$ even. Defining $\widehat{\lambda}_{d}:= \cl{}{d}  + \cl{}{d-2}+ \cdots +  \cl{}{d_{0}}$, we have a   
${\mathscr B}^{{\mathfrak a}'}$-invariant decomposition $V_{{\mathfrak a}',{\rm conn}}=V_{{\mathfrak a},{\rm conn}} \oplus W$, with 
$ W:= {\rm Span}\{\widehat{\lambda}_{d}, {\cl{}{d+1}}, \cdots,    \cl{}{d+a-1}  \} $, and we proceed as above, applying case 2 of lemma \ref{linealg} to $U=W$ with 
$2m=a$.
\item
${\mathfrak a}':=({\mathfrak a}, a',a'')$ with $a',a''$ odd. Defining $\widehat{\lambda}_{d}:= \cl{}{d}  + \cl{}{d-2}+ \cdots +  \cl{}{d_{0}}$, we have a   
${\mathscr B}^{{\mathfrak a}'}$-invariant decomposition $V_{{\mathfrak a}',{\rm conn}}=V_{{\mathfrak a},{\rm conn}} \oplus W$, with 
$W:= {\rm Span}\{\widehat{\lambda}_{d}, {\cl{}{d+1}}, \cdots,    \cl{}{d+a'+a''-1}  \} $, and we proceed as above, applying case 3 of lemma \ref{linealg} to $U=W$ with 
$2m=a$.
\end{itemize}
\end{itemize}
\blacksquare

\n {\em Proof of  Theorem \ref{mainest}.}
In the case in which $\widehat{\curv }$ is disconnected, the decomposition (\ref{splitCdisc})
$H_1(\curv_{\underline{u}}) = H_{\mathfrak a,{\rm disc}} \oplus H_1(\curv )^{\oplus 2}$, 
 the fact that $\pi_1(\dist_{\rm reg}, \underline{u})$ acts trivially on the $4g$-dimensional space $H_1(\curv )^{\oplus 2}$,
 the identification of $\pi_1(\dist_{\rm reg},\underline{u})$ with ${\mathscr B}^{\mathfrak a}$ acting on $H_{\mathfrak a,{\rm disc}} $ as described in (\ref{trccmnd}) and case 1 of Proposition \ref{justlinearalg} applied to  $H_{\mathfrak a,{\rm disc}} $ imply that
 $$
\left( \extal{H_1(\curv_{\underline{u}})} \right)^{\pi_1(\dist_{\rm reg},\underline{u})}=
\left( \extal{H_{\mathfrak a,{\rm disc}}} \right)^{{\mathscr B}^{\mathfrak a}}\bigotimes \extal{ \left( H_1(\curv)^{\oplus 2}\right)},$$
and
$$
\dim \left( \extal{H_1(\curv_{\underline{u}})}\right)^{\pi_1(\dist_{\rm reg},\underline{u})}=2^{4g}\left(\frac{1}{4} \prod (a_i+1)\right)=2^{4g-2} \prod (a_i+1)
 $$
In a completely analogous way, in the case in which $\widehat{\curv }$ is connected, the decomposition (\ref{splitCconn}) $H_1(\curv_{\underline{u}})= H_{\mathfrak a,{\rm conn}}  \bigoplus H_1(C')$,  the fact that $\pi_1(\dist_{\rm reg}, \underline{u})$ acts trivially on the $4g-2$-dimensional space $H_1(C')$, 
the identification of $\pi_1(\dist_{\rm reg},\underline{u})$ with ${\mathscr B}^{\mathfrak a}$ acting on $H_{\mathfrak a,{\rm conn}} $ as described in (\ref{trccmnd}), and case 2 of Proposition \ref{justlinearalg} applied to  $H_{\mathfrak a,{\rm conn}}$, imply that
$$
\left( \extal{H_1(\curv_{\underline{u}})} \right)^{\pi_1(\dist_{\rm reg},\underline{u})}=  \left( \extal{H_{\mathfrak a,{\rm conn}}} \right)^{{\mathscr B}^{\mathfrak a}}\bigotimes \left( \extal{  H_1(C ')} \right), 
 $$
and
$$
\dim \left( \extal{H_1(\curv_{\underline{u}})} \right)^{\pi_1(\dist_{\rm reg},\underline{u})}= 2^{4g-2} \prod (a_i+1).
 $$
\blacksquare

\subsubsection{Proof of Theorem  \ref{mainmonthm}: Back to the spectral curve}
\label{bsc}

We resume the notations of the statement of Theorem \ref{mainmonthm}.
Let $\dist$ be a distinguished neighborhood of $\Theta(s) \in \curv^{(2d)}$.
For a small enough neighborhood $N$ of $s \in \basegood$, we have $\Theta(N)\subseteq  \dist$
and $\Theta(N\cap \basesm )\subseteq  \dist_{\rm reg}$. Let $\Phi_{\dist}$ be the family constructed in 
Proposition \ref{univfam} associated with the choice of the square root ${\mathcal O}_{\curv}(D)$ of 
${\mathcal O}_{\curv}(\Theta(s))\simeq {\mathcal O}_{\curv}(2D)$. By point 4. in Proposition \ref{univfam}, 
the restriction of the spectral curve family to $N\cap \basesm$ is the pullback via $\Theta$ of
$\Phi_{  \dist_{{\rm reg}} }$, hence, by Corollary \ref{locsys}
and by the base change theorem for proper maps, we have the isomorphisms of local systems on 
$N \cap \basesm$:

$$
R^l\hitmapsmooth_* \rat \simeq \Theta^*
\left(\bigwedge^l {R^1 \Phi_{\dist_{\rm reg}}}_* \rat \right).
$$
As the stalk $(R^0 j_*R^l\hitmapsmooth_* \rat)_s$ is the direct limit
over the set of neighborhoods $N$  of $s$ in $\base$ of the 
space of monodromy invariants of the local system
$R^l\hitmapsmooth_* \rat $ in  $N\cap \basesm$, the statement  follows from Theorem \ref{mainest} 
and Remark \ref{connect}.
\blacksquare

\medskip
As shown just after the statement of \ref{mainmonthm},
 we have also completed the proof of Theorem \ref{final}.
 This latter result  has the following consequence:

\begin{corollary}
\label{btnmbrjccom}
Let $s \in \basegood$,   and let $\curv_s, \widetilde{g}, O, E$ be as in the statement of 
Theorem {\rm \ref{mainmonthm}}.  Then
\begin{enumerate}
\item
the spectral sequence (\ref{clospseq}) degenerates at $E_1$.
\item
The Poincar\'e polynomial of the fiber $\hitmap^{-1}(s)\simeq \compjac{\curv_s}^0$  is
\begin{equation}
\label{genfnccmpjcb}
\sum_l t^lb_l(\hitmap^{-1}(s))=(1+t)^{2\widetilde{g}} \left( \prod_{c \in O} \left( 1+t+\ldots+t^{2\delta_c}\right) \right)\left(\prod_{c \in E}
\left(1+t^2+\ldots +t^{\delta_c}\right)\right). 
\end{equation}
\end{enumerate}
\end{corollary}
\n{\em Proof.}
It follows from Theorem \ref{estbncj} and Theorem \ref{mainmonthm} that  
$$
\sum_{p,q} \dim{ E_{\infty}^{p,q}}= \sum_l b_l(\hitmap^{-1}(s))=  
2^{2\widetilde{g}}  \left( \prod_{c \in O}(2\delta_c+1)\right)\left(\prod_{c \in E}(\delta_c+1)\right)=\sum_{p,q} \dim{ E_1^{p,q}},
$$
hence all the differentials in the spectral sequence (\ref{clospseq})  are forced to vanish, proving point 1. 
Point 2. follows immediately by the equality 
$(R^0 j_*R^l\hitmapsmooth_* \rat)_s=b_l(\compjac{\curv_s})$ of Theorem \ref{final}, keeping track of
the cohomological degrees of the monodromy invariants.
\blacksquare

\subsection{The cases $\SL_2$ and $\PGL_2$ }
\label{comp3gra}

We now extend the main result found for $\hitmap$, Corollary \ref{icareshfongl}, to $\hitmaph$.
In order to do this, we need to discuss how $\higgsbu,$ $\mdpgl,$ $\mdsl$ and the relative maps 
$\hitmap,$ $\hitmaph,$ $\hitmapc$ introduced in \S\ref{hitchmap} are related.
As in \S \ref{modhiggsbnd}, we denote by   
$\Gamma := \Jac^0_C[2] \simeq   \Z_2^{2g}$ the group of  points of order two in $\Jac^0_C$,
by  $\higgsbu^0 \subseteq \higgsbu$ be the subset of stable Higgs bundles with traceless
Higgs field:
$$
\higgsbu^0= \{(E,\phi ) \hbox{ with }{\rm tr}(\phi)=0 \}.
$$
We denote by $\hitmap^0:\higgsbu^0 \lorw  \base^0$ the restriction
of the Hitchin map, where, we recall, $\base^0:= H^0(\curv, 2D) \subseteq \base$.

\medskip
\n
There are natural maps, easily seen to be isomorphisms of algebraic varieties: 
\[
sq:H^0(\curv, D) \times  \higgsbu^0   \lorw \higgsbu, 
\qquad
sq': H^0(\curv, D) \times \base^0  \lorw 
 H^0(\curv, D) \times \base^0 =\base
\]
given by
\[
sq: \left( (E, \phi),s\right) \longmapsto \left( E, \phi + \frac{s}{2} \otimes 1_E \right),
\qquad
sq': 
(v, u) \longmapsto \left(v, u + \frac{v^{\otimes 2}}{4}\right),\]
making the following into a commutative diagram:

\begin{equation}
\label{prhitmp}
\xymatrix{
\higgsbu^0 \ar[d]^{\hitmap^0} &  H^0(\curv, D) \times \higgsbu^0 \ar[l]\ar[d]^{Id \times \hitmap^0} \ar[r]^-{sq}_-{\simeq} & \higgsbu \ar[d]^{\hitmap}\\
  \baseo      &     H^0(\curv, D) \times \baseo  \ar[l]_-{p} \ar[r]^-{sq'}_-{\simeq}& \base.
}
\end{equation}

\begin{remark}
\label{identiequal}
{\rm 
It follows from the commutative Diagram \ref{prhitmp} that we have an isomorphism
$$ 
\hitmap_* \rat_{\higgsbu} \simeq sq'_{*}p^* \hitmap^0_* \rat_{\higgsbu^0}.
$$
}
\end{remark}

\begin{remark}
\label{regandgoodlocus}
{\rm We clearly have (see Proposition \ref{singspcrv})
$$
\base^0_{\rm reg}=\basesm \cap \base^0= \{ s \in  H^0(\curv, 2D) \, |  \; s  \;
\hbox {has simple zeros}\}
$$
and 
$$
{sq '} ^{-1}(\basesm)= H^0(\curv, D) \times \base^0_{\rm reg}, \qquad    
{sq } ^{-1}(\higgsbusm)= H^0(\curv, D) \times \higgsbusm ^0 .
$$

Similarly,  
$$
\baseogood=\basegood \cap \baseo= \{ s \in  H^0(\curv, 2D) |\, \curv_s \hbox { is integral }\}
$$
and 
$$
{sq '} ^{-1}(\basegood)= H^0(\curv, D) \times \baseogood, \qquad    
{sq } ^{-1}(\higgsbugood)= H^0(\curv, D) \times {\hitmap^0} ^{-1}(\baseogood).
$$
}
\end{remark}

Diagram \ref{prhitmp}  and Remark  \ref{regandgoodlocus}  
reduce the study of $\hitmap: \higgsbu \lorw \base$ to that of $\hitmap^0: \higgsbu^0 \lorw \baseo$:
We can safely identify $ sq'^{*}\hitmap_* \rat_{\higgsbu}$ with $p^* \hitmap^0_* \rat_{\higgsbu^0}$, and
the main Theorem \ref{final} and its corollary \ref{icareshfongl}
hold for $\hitmap ^0$ on $\baseogood$, namely,
if $j: \baseosm \lorw \baseogood$ is the open imbedding, then
\begin{equation}
\label{mzerovabbene}
IC  \left(   R^l{\hitmap^0_{\rm reg}}_* \rat_{{\higgsbu}^0_{\rm reg}}      \right)_{|{\baseogood}}     
\simeq \left( R^0j_*  R^l {\hitmap^0_{\rm reg}}_* \rat_{{\higgsbu}^0_{\rm reg}}    \right) [\dim \baseo],
\hbox{ and }
{\hitmap ^0_*\rat_{\higgsbu^0}}_{| \baseogood} \simeq \bigoplus_l \left(R^0 j_*R^l {\hitmap^0_{\rm reg}}_* \rat_{{\higgsbu}^0_{\rm reg}}\right)[-l].
\end{equation}
\bigskip
For the other groups $\SL_2$ and $\PGL_2$, we have $$\mdsl=\lambda_{D}^{-1}((\Lambda,0)), \hbox{ and }  \mdpgl=\mdsl / \Gamma = \higgsbu^0 / \jac{\curv}^{\, 0},$$
and the corresponding Hitchin maps 
\[
\xymatrix{
 \mdsl \ar[rd]^{\check{q}}\ar[rdd]_{\hitmapc} \ar@{^{(}->}[rr] &  & \higgsbu^0 \ar[ld]_{q^0}\ar[ldd]^{\hitmap^0} \\
&  \mdpgl \ar[d]^{\hitmaph}&    \\
& \base^0 &
}
\]
where $\check{q}$ and $q^0$ are the two quotient maps.

\medskip
\n
The following is readily verified 
\begin{proposition}
\label{mapdifgrps}
The map 
\[
q:  \jac{\curv}^{\, 0} \times \mdsl \lorw \higgsbu ^0,
\qquad
\left(L,(E, \phi)\right) \longmapsto (E\otimes L, \phi \otimes 1_L ).
\]
is an unramified  Galois covering with group $\Gamma$, 
and there is a commutative diagram:
%\smallskip
\begin{equation}
\label{diagdifgrps}
\xymatrix{
    \jac{\curv}^{\, 0} \times \mdsl \ar[rd]_{Id \times \check{q}} \ar[rrrr]^{q} \ar[dd]^(.61){\check{p_2}} \ar[rrd]^{[2]\times \check{q}}  & & & &\higgsbu^0\ar[llddd]^{\hitmap^0}\ar[lldd]_(.5){q^0} \ar[lld]_r  \\
& \jac{\curv}^{\, 0} \times \mdsl \ar[r]_{[2]\times Id}   & \jac{\curv}^{\, 0} \times \mdpgl\ar[d]^{\hat{p_2}}&     \\
 \mdsl \ar[rr]^{\check{q}}\ar[rrd]_{\hitmapc} & &\mdpgl\ar[d]_{\hitmaph}&     \\
 &  &\baseo ,& 
}
\end{equation}
where $r: \higgsbu^0 \lorw \jac{\curv}^{\, 0} \times \mdpgl$ sends 
$(E, \phi)$ to $(\det E \otimes \Lambda^{-1}, q^0(E, \phi))$, and
$[2]\times \check{q}$ sends $(L, (E,\phi))$ to 
$(L^{\otimes 2}, \check{q}(E,\phi))$.
\end{proposition}

\begin{remark}
\label{quoprod}
{\rm The map $[2]\times \check{q}:    \jac{\curv}^{\, 0} \times \mdsl \lorw \jac{\curv}^{\, 0} \times \mdpgl $ is the quotient map relative to the
diagonal  action of $\Gamma \times \Gamma$ on  $\jac{\curv}^{\, 0} \times \mdsl$.
}
\end{remark}

Proposition \ref{mapdifgrps} implies that 
\[  H^*(\higgsbu) \simeq H^*(\higgsbu^0) \simeq  \left( H^*(\mdsl )\otimes H^*( \Jac^0_{\curv}  )    \right)^{ \Gamma }\simeq H^*(\mdpgl )\otimes H^*( \Jac^0_{\curv}  ) .\]
The last isomorphism follows from the fact that the action of $ \Gamma$ on  $H^*(\Jac^0_C)$ is trivial,  
as it is the restriction to a subgroup of the action of the connected group $\Jac^0_C$.

\medskip
Before stating Theorem \ref{identify}, 
which gives a refinement of the isomorphism above at the 
level of derived categories, we make some general remarks 
on actions of finite abelian groups and the splitting 
they induce on complexes. 

For ease of exposition, until further notice, we work with the constructible
derived category of sheaves of {\em complex}  vector spaces.
Let $K$ be an object of $D_{\ai}$ and suppose that a finite abelian group $\Gamma$ acts
on  the right on $K$,  i.e. that we are given a representation
$\Gamma \to ({\rm Aut}_{D_A} (K))^{op}$. 
It then follows from \ci{cortihana}, 2.24, (see also  \cite{laumonngo}  Lemma 3.2.5)  that there is a character decomposition 
 \begin{equation}
 \label{isort}
K \simeq \bigoplus_{\zeta \in \hat{\Gamma} } K_{\zeta} .
 \end{equation}
where $\hat{\Gamma}$ denotes the group of characters of $\Gamma$.

Suppose  $\Gamma$
acts on  the left on  an algebraic variety $\mi$,
and let  $\mi \stackrel{q}{\lorw} \mi / \Gamma $ be  the quotient map.
Since $q$ is finite, the derived direct image complex
$q_* \comp_{\mi}$ is a sheaf.
Clearly,   $\Gamma$  acts on 
$q_* \comp_{\mi}$ on the right via pull-backs,
and \ref{isort} above boils down to the canonical decomposition 
in the Abelian category of  sheaves
$$
\bigoplus_{\zeta \in  \hat{\Gamma} } L_{\zeta} \simeq q_* \comp_{\mi}.
$$

Let $\hi : \mi \lorw \ai $ be a proper map which is $\Gamma$-equivariant. 
We have the commutative diagram
$$
\xymatrix{
\mi \ar[d]^{\hi} \ar[r]^q & \mi/\Gamma  \ar[ld]^{\hi'} \\
\ai &
}
$$
and, by using ${h'}_* q_*= h_*$, we get
the  canonical identification.
\begin{equation}
\label{decoprima}
\bigoplus_{\zeta \in \hat{\Gamma}}  \hi'_*   L_{\zeta} \simeq \hi_*\comp_{\mi}.
\end{equation}
Clearly, $h_* \comp_{\mi}$ is endowed with the $\Gamma$-action, 
and (\ref{decoprima}) is just its character decomposition, namely
$\left(\hi_*\comp_{\mi}\right)_{\zeta}=\hi'_*   L_{\zeta}$.

In particular, taking the trivial representation  $\rho=1$,  we have  
$L_1=(q_*\comp_{\mi})^{\Gamma} \simeq \comp_{\mi/\Gamma}$ and 
we  thus identify  the direct image
$\hi'_*\comp_{\mi/\Gamma}=\hi'_*((q_*\comp_{\mi})^{\Gamma}) $  with 
the canonical direct summand (which we may call the $\Gamma$-invariant part)
$\left(\hi_*\comp_{\mi}\right)^{\Gamma}:= (\hi_*\comp_{\mi})_1$ 
of $\hi_*\comp_{\mi}$.

\medskip
Let us go  back to our situation where 
 $\Gamma \simeq \Z_2^{2g}$. In this case,
 the characters are all $\{\pm 1\}$-valued,
and  we can safely return to rational coefficients.

From what above and the diagram  (\ref{diagdifgrps}), it follows that 
$(\check{q}\circ \check{p_2})_* \rat_{\jac{\curv}^{\, 0} \times \mdsl}=
(q^0\circ q)_*\rat_{\jac{\curv}^{\, 0} \times \mdsl}$ contains $q^0_*\rat_{\higgsbu^0}$
and $ (\check{q}\circ \check{p_2})_* \rat_{\mdsl }$

\begin{theorem}
\label{identify}
There are canonical isomorphisms in $D_{\baseo}$:
\begin{equation}
\label{identifica}
\hitmap^0_*\rat_{\higgsbu^0}\simeq 
\bigoplus_{i \in \nat} \bigwedge^i  H^1(C)\otimes \hitmaph_* \rat_{\mdpgl} [-i]\simeq
\bigoplus_{i \in \nat} \bigwedge^i  H^1(C)\otimes (\hitmapc_* \rat_{\mdsl})^{\Gamma} [-i].
\end{equation}
\end{theorem}
\n{\em Proof. }
Consider the diagram (\ref{diagdifgrps}).
As noticed in Remark \ref{quoprod}, the map $[2]\times \check{q}$ is the quotient by the action of $\Gamma \times \Gamma$.
Consider the character decompositions 
$$
[2]_* \rat_{\jac{\curv}^{\, 0}} \simeq  \bigoplus_{\zeta \in \hat{\Gamma}}L_{\zeta},
\qquad
\check{q}_*\rat_{\mdsl}\simeq \bigoplus_{\zeta \in \hat{\Gamma}}M_{\zeta}.
$$
The K\"unneth formula gives the following canonical isomorphisms in 
$D_{\jac{\curv}^{\, 0} \times \mdpgl}$:

\begin{equation}
\label{urk1}
([2]\times \check{q})_* \rat_{\jac{\curv}^{\, 0} \times \mdsl} \simeq 
[2]_* \rat_{\jac{\curv}^{\, 0}} \boxtimes \check{q}_* \rat_{\mdsl}
\simeq \bigoplus_{(\zeta, \zeta') \in \hat{\Gamma} \times \hat{\Gamma}} L_{\zeta } \boxtimes M_{\zeta '},
\end{equation}

\begin{equation}
\label{urk2}
([2] \times Id)_*\rat_{\jac{\curv}^{\, 0} \times \mdpgl} \simeq 
\bigoplus_{\zeta \in  \hat{\Gamma}} L_{\zeta} \boxtimes \rat_{\mdpgl} \subseteq ([2]\times \check{q})_* \rat_{\jac{\curv}^{\, 0} \times \mdsl} ,
\end{equation}
and
\begin{equation}
\label{urk3}
r_*\rat_{\higgsbu^0}\simeq \bigoplus_{\zeta \in  \hat{\Gamma}} L_{\zeta} \boxtimes M_{\zeta} \subseteq ([2]\times \check{q})_* \rat_{\jac{\curv}^{\, 0} \times \mdsl} ,
\end{equation}
this latter since $q$ is the quotient by the diagonal action $\Gamma \lorw \Gamma \times \Gamma$.

\n
Noting that the map $[2]$ is a finite covering of a product of circles,
we have canonical isomorphisms
\begin{equation}
\label{urka}
(\hat{p_2})_*  L_{\zeta }=\left( \bigoplus_i H^i(\jac{\curv}^{\, 0}, L_{\zeta } )[-i]\right)\otimes \rat_{\mdpgl}= \left\{\begin{array}{ll}
0   & \hbox{ if } \zeta \neq 1 \\
\bigoplus_i \bigwedge^i H^1(C)\otimes \rat_{\mdpgl} [-i] & \hbox{ if } \zeta  = 1.
\end{array}\right. 
\end{equation}

Applying the functor $\hat{p_2}_*$ to $([2]\times \check{q}))_* \rat_{\jac{\curv}^{\, 0} \times \mdsl }$, to 
$([2]\times Id))_* \rat_{\jac{\curv}^{\, 0} \times \mdpgl }$ and to 
$r_* \rat_{\higgsbu^0 }$,
we obtain   canonical isomorphisms in $D_{\mdpgl}$:

\begin{equation}
\label{burk1}
(\hat{p_2}\circ ([2]\times \check{q}))_* \rat_{\jac{\curv}^{\, 0} \times \mdsl }  
\simeq 
\bigoplus_{i \in \nat, \zeta \in \hat{\Gamma}} \bigwedge^i H^1(C)\otimes M_{\zeta} [-i]
,\end{equation}

\begin{equation}
\label{burk2}
(\hat{p_2}\circ ([2]\times Id))_* \rat_{\jac{\curv}^{\, 0} \times \mdpgl }
\simeq
\bigoplus_{i \in \nat} \bigwedge^i  H^1(C)\otimes \rat_{\mdpgl} [-i],
\end{equation}

\begin{equation}
\label{burk3}
q^0_* \rat_{\higgsbu^0 }=(\hat{p_2}\circ r)_*\rat_{\higgsbu^0 }
\simeq
\bigoplus_{i \in \nat}
\bigwedge^i H^1(C)\otimes \rat_{\mdpgl} [-i].
\end{equation}
Taking the direct image $\hitmaph_*$ of the isomorphisms \ref{burk1}
\ref{burk2} and \ref{burk3} above, and using the fact that 
$\check{q}_*\rat_{\mdsl}\simeq \bigoplus_{\zeta \in \hat{\Gamma}} M_{\zeta}$, 
with $M_1 \simeq \rat_{\mdpgl}$, we find the 
canonical isomorphisms in $D_{\baseo}$ we are seeking for.
\blacksquare

\bigskip
The map $\hitmaph$ is projective, and it will be shown in Corollary \ref{alphample} that 
the class $\alpha$, defined by Equation \ref{defabps},
is the cohomology class of a $\hitmaph$-ample line bundle on $\mdpgl$.
We can thus apply the results of \S \ref{delspl}:
we have the Deligne  isomorphims, which depends on $\alpha$:
\[\phi_{\alpha}: \bigoplus_{p \geq 0} {\hat{\mathcal P}}^p[-p]  
 \stackrel{\simeq}\lorw \hat\hitmap_* \rat_{\mdpgl}
 [\dim \mdpgl], 
\]
underlying  an  even finer decomposition (see (\ref{inuno})). The cohomology groups 
$H^*(\mdpgl)$ are   endowed with the direct sum decomposition (\ref{qdec}):
$H^*(\mdpgl )= \sum Q^{i,j}$.

Similarly the class 
$$
\tilde{\alpha}:= \alpha \otimes 1 + 1\otimes \sum_i \epsilon_i \epsilon_{i+g} \in H^*(\mdpgl )\otimes H^*( \Jac_{\curv})= H^*(\mdgl),
$$ 
introduced in (\ref{defalpfatld}), is the class of a  
relatively ample line bundle on $\mdgl$.

We now determine the Deligne splitting $\phi_{\tilde\alpha}$
associated with $\hitmap$  and
$\tilde{\alpha}$.
Denote by      
 $S_{\Jac^0_{\curv}}: \Jac^0_{\curv} \lorw pt$  the ``structural map."
By using  the canonical splitting in $D_{pt}$
$$
{\phi_J}\, : \; S_{\Jac^0_{\curv} * }\rat_{\Jac^0_{\curv}}  \stackrel{\simeq}\lorw 
\left( \bigoplus_{i \geq 0} \bigwedge^i  H^1(C)[-i]\right),  
$$
the canonical isomorphism  of Theorem \ref{identify} takes the form
$$
\hitmap^0_*\rat_{\higgsbu^0} \simeq \hitmaph_* \rat_{\mdpgl} \boxtimes  
S_{\Jac^0_{\curv} * }\rat_{\Jac^0_{\curv}}   \simeq
\hitmaph_* \rat_{\mdpgl} \boxtimes  \left( \bigoplus_{i \in \nat} \bigwedge^i  H^1(C)[-i]\right).
$$
Note that this splitting does not depend on the choice of $ \sum_i \epsilon_i \epsilon_{i+g}$ and that the operation
 of cupping with this class is diagonal (with respect to the splitting).
Let 
\[
\phi_{\tilde{\alpha}}^{\,0}: \bigoplus_{p \geq 0} {\mathcal P}^p[-p] 
 \stackrel{\simeq}\lorw \hitmap^0_* \rat_{\higgsbu^0}[\dim \baseo], 
\]
be  the Deligne splitting  for $\hitmap^0$   relative to  (the restriction of) $\tilde{\alpha}$.
Since, as pointed out above, the action of $ \sum_i \epsilon_i \epsilon_{i+g}$ is diagonal,
\[
\phi_{\alpha} \otimes \phi_J: 
\left(\bigoplus_{p \geq 0} {\hat{\mathcal P}}^p[-p]\right) \otimes \left( \bigoplus_{i \geq 0} \bigwedge^i  H^1(C)[i]\right) 
\lorw \hitmaph_* \rat_{\mdpgl}
 [\dim \baseo] \otimes S_{\Jac^0_{\curv} * }\rat_{\Jac^0_{\curv}}  \simeq \hitmap^0_* \rat_{\higgsbu^0}[\dim \baseo]
\]
satisfies the properties stated in Fact \ref{chardeliso}, hence it is the isomorphism $\phi_{\tilde{\alpha}}^{\,0}$.
Since, by Remark \ref{identiequal}, we have the natural isomorphism
$ \hitmap_* \rat_{\higgsbu} \simeq sq'_{*}p^* \hitmap^0_* \rat_{\higgsbu^0}$, 
and $p$ is smooth with contractible fibres of dimension 
$\dim \base- \dim \baseo$,  the isomorphism 
$$ 
sq'_{*}p^*(\phi_{\alpha} \otimes \phi_J)=
sq'_{*}p^*(\phi_{\tilde{\alpha}}^0):\hitmap_* \rat_{\higgsbu}[\dim \baseo] \lorw 
\bigoplus_{r \geq 0}sq'_{*}p^* \left(
\bigoplus_{p+i=r}{\hat{\mathcal P}}^p \otimes \bigwedge^i  H^1(C)\right)[-r]
$$
is, up to the shift 
$\dim \baseo - \dim \base$, 
the Deligne isomorphism $\phi_{\tilde{\alpha}}$ associated with $\hitmap$ and $\tilde{\alpha}$.

\medskip
In particular, we have, for all $k$ and $p$, a canonical isomorphism:
\begin{equation}
\label{fe1}
H^k_{\leq p}(\higgsbu)=
\bigoplus_{j \geq 0}
\left(H^{k-j}_{\leq{p-j}} (\mdpgl)  \otimes  \bigwedge^j  H^1(C)\right).
\end{equation}
The isomorphism of Theorem \ref{identify} can be understood geometrically as  follows.

\medskip
Given the nonsingular double (branched) cover 
$\curv_s \lorw \curv$, the Prym variety $\prym(\curv_s) \subseteq  \jac{\curv_s}^{\, 0} $ is defined as
\begin{equation}
\label{defiprym}
\prym(\curv_s):=\{ {\mathcal F} \in  \jac{\curv_s}^{\, 0} \, |\; {\rm Nm}({\mathcal F})= {\mathcal O}_{\curv}\},
\end{equation}
where ${\rm Nm}: \jac{\curv_s}^{\, 0} \lorw
\jac{\curv}^{\, 0}$ is the norm map, see \ci{acgh},  Appendix B, \S 1.
Clearly, the image by pullback of the subgroup $\Gamma$ is contained in $\prym(\curv_s)$,
and we have the quotient isogeny 
$$
\prym(\curv_s) \lorw \prym(\curv_s) / \Gamma.
$$

The open subset $\mdslsm:= \mdsl \cap \higgsbusm$ is a torsor for 
the Abelian scheme  ${\mathscr P}^{-}_{\rm reg} $    over $\base^0_{\rm reg}$
whose fibre over $s \in \base^0_{\rm reg}$ is $\prym(\curv_s)$. Similarly,
the open subset $\mdpglsm:= \mdslsm / \Gamma$  is a torsor for 
the Abelian scheme 
${\mathscr P}^{-}_{\rm reg}  / \Gamma $ over      $\base^0_{\rm reg}$
whose fibre over $s \in \base^0_{\rm reg}$ is  $\prym(\curv_s)  / \Gamma \simeq \prym(\curv_s)^{\vee}$.

The involution $\iota$ (see  \S \ref{spctrcrv})
 on the family of  spectral curves   $u:{\mathscr C}_{\base}\to \base$  gives a $\zed /2\zed$-action on the local system
$R^1 \spectfamsmooth_* \rat_{{\mathscr C}_{\basesm}}$ and a corresponding decomposition  into $(\pm 1)$- eigenspaces ${\mathcal V}^{\pm}$. 
The pullback from $\curv$ gives a canonical isomorphism of local systems
$$
{\mathcal V}^+ = H^1(\curv) \otimes \rat_{\basesm},
$$
between the local system of invariants and
 the constant sheaf with stalk $H^1(\curv) $,
so that 
$$
 R^1 \hitmapsmooth_* \rat_{\higgsbusm} \simeq   R^1 \spectfamsmooth_* \rat_{{\mathscr C}_{\basesm}}= 
\left( H^1(\curv) \otimes \rat_{\basesm} \right) \bigoplus {\mathcal V}^-.
$$
It follows from Corollary \ref{locsys}  that, for every $l$,
\begin{equation} 
\label{altraroba}
 R^l \hitmapsmooth_* \rat_{\higgsbusm} \simeq   \bigwedge^l R^1 \spectfamsmooth_* \rat_{{\mathscr C}_{\basesm}}
= \bigoplus_{a+b=l} 
\bigwedge^aH^1(\curv) \otimes  \bigwedge^b {\mathcal V}^-   .
\end{equation}
Clearly, the analogous statement  for the restriction to $\baseosm$ holds true.
Comparing with Theorem \ref{identify} we see that 
\begin{equation}
\label{chessega}
{R^l \hitmaph_{\rm reg}}_* \rat_{\mdpgl}
\simeq \bigwedge^l {\mathcal V}^- \, \, \hbox{ and } \, \, 
\hitmaph_*{\rat_{\mdpgl}}_{|\baseosm}\simeq \bigoplus_l \bigwedge^l {\mathcal V}^-[-l]. 
\end{equation}
From     (\ref{altraroba}) we have
$$
IC \left( R^l {\hitmap ^0_{\rm reg}}_* \rat_{\higgsbusm} \right) \simeq  \bigoplus_{a+b=l} IC \left( \bigwedge^aH^1(\curv) \otimes  \bigwedge^b{\mathcal V}^-\right)
\simeq 
\bigoplus_{a+b=l}\bigwedge^aH^1(\curv) \otimes IC \left( \bigwedge^b{\mathcal V}^-\right),
$$
while, from the first equality  in (\ref{mzerovabbene}),
$$
IC  \left( R^l{\hitmap^0_{\rm reg}}_* \rat_{{\higgsbu}^0_{\rm reg}}      \right)_{|\baseogood}      
\simeq 
\left( R^0j_*  R^l {\hitmap^0_{\rm reg}}_* \rat_{{\higgsbu}^0_{\rm reg}}    \right) [\dim \baseo]
\simeq 
\bigoplus_{a+b=l}\bigwedge^aH^1(\curv) \otimes    R^0j_*  \bigwedge^b{\mathcal V}^- .
$$
Comparing with the second equality of  in (\ref{mzerovabbene}) and Theorem \ref{identify} we finally obtain  
\begin{corollary}
\label{pglvabene}
Let $j:\baseosm \mapsto \baseogood$ be the open imbedding.
Over the locus $\baseogood$ we have 
$$
IC\left(\bigwedge^l {\mathcal V}^-\right) _{|{\baseogood}}\simeq \left( R^0j_* \bigwedge^l {\mathcal V}^-\right)[\dim \baseo],
\hbox{ and }
\left(\hitmaph_* \rat_{\mdpgl}\right)_{|\baseogood}    \simeq  \bigoplus_l \left(R^0j_* \bigwedge^l {\mathcal V}^-\right)[-l].
$$
In particular, over the locus $\baseogood$, the map $\hitmaph$ satisfies Assumption \ref{ipot}.
\end{corollary}

\section{Preparatory results}
\label{someres}
\subsection{Placing the generators in the right perversity}
\label{plgenrpe}

Here we prove the following

\begin{theorem}
\label{eccolo}
We have 
\[
\epsilon_i \in H^1_{\leq 1}(\mdgl),  \;\; \forall \; 1 \leq i \leq 2g.\]
In each of the  three cases $\GL_2, \SL_2$ and $\PGL_2$, we have
\[
\alpha \in H^2_{\leq 2}(\higgsbu),\,\, \psi_i \in  H^3_{\leq 2}(\higgsbu),  
\;\;
\forall \; 1 \leq i \leq 2g.
\]
Furthermore, if $g>2$, or $g\geq 2$ and $\deg D>2g-2$,
\[
\beta \in H^4_{\leq 2}(\higgsbu).
\]
\end{theorem}
\n{\em Proof.}  By the isomorphism (\ref{fe1}), it is enough to
work in the case of $\GL_2$.
Recalling that $\deg {\epsilon}_i = 1 $, $\deg {\alpha}=2$,
$\deg {\psi_i} =3$ and $\deg {\beta}=4$,   Proposition \ref{ndomett} 
implies that
$$
\epsilon_i \in H^1_{\leq 1}(\higgsbu), \,   
\alpha \in H^2_{\leq  2}(\higgsbu), \psi_i \in H^3_{\leq  3}(\higgsbu), \hbox{ and }
\beta \in H^4_{\leq 4}(\higgsbu).
$$

\n
By Thaddeus' Proposition~\ref{restrunivcls} in the Appendix, we have that $\alpha$ does not vanish over the general fiber, while
$\psi_i$ and $\beta$ do.  Theorem \ref{gdpf} implies 
$\alpha \in H^2_{\leq 2}(\higgsbu), \psi_i \in H^3_{\leq 2}(\higgsbu)$
and $\beta \in  H^4_{\leq 3}(\higgsbu)$.  The same proposition
shows that in order to conclude that $\beta \in H^4_{\leq 2}(\higgsbu)$,  
we need to prove that $\beta$ vanishes over a generic line $\Lambda \subseteq
\base$.

Set, for simplicity of notation,
$M_{\Lambda}:=\hitmap^{-1}(\Lambda), M_{\Lambda_{\rm reg}}:=\hitmap^{-1}(\Lambda_{\rm reg})$
where $\Lambda_{\rm reg}:= \Lambda \cap \basesm$.
\n
Since, by Lemma \ref{smingo}, 
the generic line avoids $\base \setminus \basegood$ unless $g=2$ and $D=K_{\curv}$,
we have that Assumption \ref{ipot} holds for   
$\hitmap_{|M_{\Lambda}}: M_{\Lambda} \to \Lambda$ due to Corollary~\ref{icareshfongl}.
%%T22 I added a reference to Corollary~\ref{icareshfongl} because I believe that is what is used here%%
 We thus  have Fact \ref{fatto}.
Let $j: \Lambda_{\rm reg} \to \Lambda$ be the open immersion.

\n
Since $\Lambda$ and $\Lambda_{\rm reg}$
are affine and one dimensional,  their cohomology groups
in degree $\geq 2$ with coefficients in constructible sheaves are zero.
In particular, the Leray spectral sequences for $\hitmap_{|M_{\Lambda}}$ 
and $\hitmap_{|M_{\Lambda_{\rm reg}}}$ are necessarily $E_2$-degenerate. 
The restriction  map in cohomology yields a map of Leray spectral 
sequences and thus a  commutative diagram of
short exact ``edge" sequences (as in \S \ref{badb}, $R^l$ stands for the local system  $R^l{\hitmapsmooth }_* \rat$)
 \[
 \xymatrix{
 0 \ar[r] & \ar[r] H^1 (j_* R^3) \ar[r] \ar[d]^{r'}& H^4(M_{\Lambda}) \ar[r] \ar[d]^{r} &
 H^0( j_* R^4) \ar[r] \ar[d]^{=}& 0 \\
  0 \ar[r]  & \ar[r] H^1 ( R^3) \ar[r]  & H^4(M_{\Lambda_{\rm reg}}) \ar[r]  &
 H^0(  R^4) \ar[r]  & 0.
 }
 \]
 The arrow $r'$ is, in turn, arising from the edge sequence of the Leray spectral sequence
 for the map $j$ and is thus injective.

Below we prove that the restriction of the class $\beta \in H^4(\mdgl)$ to $\higgsbusm$ vanishes.  The sought-after conclusion $\beta_{|M_{\Lambda}}=0$ follows from this by a 
simple diagram chasing.

The class  $\beta$ is a multiple of the second Chern class $c_2(\mdgl)$. This can be seen by formally calculating
the total Chern class of $\mdgl$ using $\mathbb E$. The result $c(T\mdgl)=(1-\beta)^{2g-2}$  formally agrees with the formula for the total Chern class of
$T{\mathcal N}\oplus T^*{\mathcal N}$ which was calculated in \cite[Corollary 2]{newstead}.
 Every linear function  on  $\base$
 gives a Hamiltonian vector field on
$\higgsbu$, tangent to the fibres of $\hitmap$. These Hamiltonian vector fields trivialize the relative tangent bundle of $\higgsbusm$.
The tangent bundle $ T\higgsbusm$ is an extension of the 
trivial bundle $\hitmap_{\rm reg}^*T\basesm$ by the relative tangent bundle.
It follows that the  Chern classes of $ T\higgsbusm$  vanish.
\blacksquare

\begin{remark}
\label{genus2}
{\rm In fact, although the argument above cannot be applied, we have
that  $\beta \in H^4_{\leq 2}(\higgsbu)$, and that $\beta$ vanishes
over the generic line,  also in the case  $g=2$ and 
$D=K_{\curv}$. 
This fact is  proved in Proposition \ref{betaisfine}.}
\end{remark}

\subsection{Vanishing of the refined  intersection form}
The purpose of this section is to establish Corollary
\ref{nozeroperv}, a fact we need 
in \S \ref{dequalschi} 
as one of the pieces in the proof of  the equality of the  weight and perverse Leray filtrations
in the case $D=K_{\curv}$.
We need the following result proved in \ci{hauselint}, Theorem 1.1.

\begin{theorem} 
\label{tamasinterform}
The natural map 
$$
H^{6g-6}_c(\mdsld) \lorw H^{6g-6}(\mdsld)
$$
from compactly supported cohomology to cohomology is the zero map.
\end{theorem}

\begin{remark}
\label{topdimns}
Note that  $H^{d}(\mdsld)=H^{d}(\mdpgld)=0$ for every $d > 6g-6$, since, by  (\ref{naht}), 
$\mdsld$ and $\mdpgld$ are homeomorphic to $\slM_\B$ and $\pglM_\B$ respectively, which are
affine complex varieties of dim $6g-6$.
\end{remark}

Let us recall that  the refined intersection forms 
on the fibres of a map (see \ci{htam}, \S 3.4 for a general discussion)
is the composite of the two maps 
\begin{equation}
\label{factintfrm}
H_{12g-12-r}(\hitmapc^{-1}(0)) \lorw  H^r(\mdsld) \lorw  H^{r}(\hitmapc^{-1}(0)),
\end{equation}
and arises by taking  $r$-th cohomology of the adjunction maps
\begin{equation}
\label{compadj}
i_!i^!\hitmapc_* \rat_{\mdsld} \lorw  \hitmapc_* \rat_{\mdsld} 
\lorw   i_*i^*\hitmapc_* \rat_{\mdsld}
\end{equation}
and using the canonical isomorphisms 
${H}^r(\base,i_!i^!\hitmapc_* \rat_{\mdsld})
\simeq   H^r(\mdsld, \mdsld \setminus \hitmapc^{-1}(0) ) 
\simeq H_{12g-12-r}(\hitmapc^{-1}(0))$
and
${H}^r(\base, i_*i^*\hitmapc_* \rat_{\mdsld}) \simeq H^{r}(\hitmapc^{-1}(0)).$ 
The first map  in (\ref{factintfrm}) is the ordinary push-forward in homology (followed by Poincar\'e duality)
and the second map is the restriction map.

By the decomposition theorem (see Formula \ref{dt}),  the complex
$\hitmapc_* \rat_{\mdsld}$  splits into the direct sum of its perverse cohomology sheaves.
Combining this fact with (\ref{compadj}) we obtain 
one refined intersection form for each perversity $a$:
\[
\iota_a^r \, :\; H_{12g-12-r, a}(\hitmapc^{-1}(0))\lorw H^{r}_a(\hitmapc^{-1}(0)).
\]

\medskip
The following fact will be used in the proof of the next corollary:
\begin{theorem}
\label{dcminterform}
The map $\iota^{3g-3}_{3g-3}: H_{6g-6,3g-3}(\hitmapc^{-1}(0) )\lorw H^{6g-6}_{3g-3}(\hitmapc^{-1}(0)) $
is an isomorphism.  
\end{theorem}
\n{\em Proof.} See  \ci{htam}, Theorem 2.1.10 (where a  different numbering convention is adopted).
\blacksquare

\begin{corollary}
\label{nozeroperv}
The perverse Leray filtration on the middle-dimensional groups satisfies
$$H^{6g-6}_{\leq 3g-3}(\mdsld)=0, \qquad H^{6g-6}_{\leq 3g-3}(\mdpgld)=0.$$
\end{corollary}
\n{\em Proof.} Since   $\hitmaph_* \rat_{\mdpgld} = (\hitmapc_* \rat_{\mdsld})^{\Gamma} $, 
it suffices to prove the statement for $\mdsld$.

\medskip
\n For every $r$, the first map in (\ref{factintfrm}) factors as follows
\begin{equation}
\label{factntrmp}
H^r(\mdsld, \mdsld \setminus \hitmapc^{-1}(0) ) \lorw H^r_c(\mdsld) 
\lorw H^r(\mdsld) 
\end{equation}
so that, by Theorem \ref{tamasinterform}, it is the zero map.
It follows that the refined intersection form 
(\ref{factintfrm})
$$
H_{6g-6}(\hitmapc^{-1}(0)) \lorw  H^{6g-6}(\hitmapc^{-1}(0))
$$
vanishes. This, in turn, implies that all the graded refined intersection forms
$\iota^{3g-3}_a$ are zero. 

If we combine the vanishing of $\iota^{3g-3}_{3g-3}$ with Theorem
\ref{dcminterform}, then we deduce that
\begin{equation}
\label{bsque}H^{6g-6}_{3g-3}(\hitmapc^{-1}(0))=0.\end{equation}

 Proposition \ref{ndomett} implies that 
$H^{6g-6}_{\leq 3g-4}(\mdsld)=0$. In order to conclude, we need to
prove that $H^{6g-6}_{3g-3}(\mdsld)=0$.

Since the restriction map to the fiber is compatible with any splitting
coming from the decomposition theorem, in view of the vanishing
(\ref{bsque}), it is enough to show that the restriction map
$H^{6g-6}(\mdsld) \lorw H^{6g-6}(\hitmapc^{-1}(0))$ is an isomorphism. 
In fact it follows from \cite[\S 3]{simpson-ubiquity} that $\hitmapc^{-1}(0)$ - being
the downward flow \cite[Theorems 3.1 and 5.2]{hausel-compact} of a $\C^\times$-action on $\mdsld$ - is a deformation retract of $\mdsld$.
\blacksquare

\subsection{Bi-graded $\slt$-modules}
\label{bigsm}

We collect here some linear algebra considerations which will be used in the sequel of the paper.
Let 
 ${\mathbb H} = \bigoplus_{d,w \geq 0} H^d_w$ be a   finite dimensional
 bi-graded
 vector space. We say that $d$ is the degree and $w$ is the weight.
 We employ the following notation
 \[
 {\mathbb H}_w := \bigoplus_{d \geq 0} {\mathbb H}^d_w, \qquad
 {\mathbb H}^d:= \bigoplus_{w \geq 0} {\mathbb H}^d_w.\]
 Let $Y $ be a nilpotent endomorphism of ${\mathbb H}$ which is
bi-homogeneous   of type $(2,2)$, i.e. $Y: {\mathbb H}^*_\star \to {\mathbb H}^{*+2}_{\star+2}$.
Let 
$w_o \in \zed^{\geq 0}$ be such that for every $l \geq 0$  we have 
hard-Lefschetz-type isomorphisms
\[
Y^l:  {\mathbb H}_{w_o - l}  \stackrel{\simeq}\lorw  {\mathbb H}_{w_o+l}.\]
Note that we must then have  that ${\mathbb H}_w= \{0\}$ for every $w > 2w_o$.

It is well-known that we can turn ${\mathbb H}$ into 
an $\slt$-module in a natural way by means of 
a unique pair $(X, H)$ of homogeneous  endomorphism of ${\mathbb H}$
or respective types $(-2,-2)$ and $(0,0)$ subject to $[X,Y] =H$.
In this case $H$ is just the ``$w$-grading'' operator: $Hu=(w-w_0)u$ if $u \in {\mathbb H}_w$.

Given a bi-homogeneous element $u \in {\mathbb H}^d_w$, we define
\begin{equation}
\label{dd11} \Delta (u)  :  = d-w.
\end{equation}
Note that the action of $\slt$ leaves
$\Delta$ invariant.

We define the primitive space $\P := \ke \, X \subseteq {\mathbb H}$ and we obtain the primitive
decomposition

\beq
\label{primideco}
{\mathbb H} = \bigoplus_{j \geq 0} \;Y^j \cdot \P.
\eeq

Note that we have 
\[
Y^j \cdot \P = \ke{\, X^{j+1}} \cap \im {\,Y^j}. \]

Since $X$ is homogeneous, the space $\P$ is also bi-graded. 
We set $\P_w:= \P \cap {\mathbb H}_w$, $\P^d:= \P \cap {\mathbb H}^d$ and
$\P^d_w:= \P\cap {\mathbb H}^d_w$. 
We have $\P_w =\{ 0 \} $ for every $w > w_o$ and 
\[
\P= \bigoplus_{d,w} \P^d_w= \bigoplus_{w \geq 0} \P_w = 
\bigoplus_{d \geq 0} \P^d. 
\]
For every fixed weight $w$, the primitive decomposition can be re-written as follows
\[
{\mathbb H}_w = \bigoplus_{j \geq 0}  Y^j \cdot \P_{w-2j},
\qquad
{\mathbb H}_w^d = \bigoplus_{j \geq 0}  Y^j \cdot \P^{d-2j}_{w-2j}
.\]
We denote by $\Pi$ the operator of projection onto $\P$; 
clearly, if  $u \in \Ha^d_w$  then $\Pi(u) \in \P^d_w$, and $
\Pi(u)=u + \sum_{j>0} Y^j u_j,$
with $u_j \in \Ha^{d-2j}_{w-2j}$.

\medskip
Given a subset 
 $S \subseteq {\mathbb H}$, we define  the associated   $Y$-string
\[
\langle S \rangle_Y = \bigoplus_{j\geq 0} \; Y^j \cdot \langle S \rangle_\rat \; \subseteq {\mathbb H}.\]
In particular, $\langle \P \rangle_Y ={\mathbb H}$.
If $u$ is a bi-homogeneous element, then $\Delta$  is constant on the $Y$-string $\langle u\rangle_Y$ generated by  $u$. 

Let $0 \neq p_w \in \P_w$. Then $\langle p_w \rangle_Y= \langle p_w, Y\cdot p_w, \ldots, Y^{w_o-w} \cdot p_w \rangle_\rat
\subseteq  {\mathbb H}$
is isomorphic to the irreducible $\slt$-submodule of dimension $(w_o-w)+1$. Hence,
the isotypical decomposition of the $\slt$-module ${\mathbb H}$ is:
\[
{\mathbb H} =  \bigoplus_{0 \leq w \leq w_o} \langle \P_w \rangle_Y.\]
We define the {\em isobaric} decomposition of $\mathbb H$ as the  direct sum decomposition  obtained by 
  grouping  terms according to the powers of $Y$ in the isotypical decomposition given above, namely
\begin{equation}
\label{isobbaric}
{\mathbb H} = \bigoplus_{0 \leq w \leq w_o} \;\bigoplus_{0 \leq j \leq w_o -w}
Y^j \cdot \P_w.
\end{equation}

The proof of the following proposition is completely elementary, and safely left to the reader
(for point 4 just remark that if  $u \in \Ha^d_w\cap {\rm Im} Y$, then $u=Yv$, with $v \in \Ha^{d-2}_{w-2}$):
\begin{proposition}
\label{projprim}
Let $M \subseteq {\mathbb H}$ be a subset of  bi-homogeneous elements such that
 $\langle M\rangle_Y={\mathbb H}$.  
Then:
\begin{enumerate} 
\item 
The set  $\Pi(M)\subseteq \P$  obtained by projecting $M$ to the primitive space
is also a $Y$-generating subset of  bi-homogeneous elements with the same bi-degrees, and its linear span is $\P$:
\[
\langle \Pi(M) \rangle_Y ={\mathbb H}, \qquad \langle \Pi(M) \rangle_\rat =\P.\]

\item
 If $T' \subseteq \Pi(M)$ is a linearly independent set, then  it can be completed 
 to a basis $T \subseteq \Pi(M)$ for $\P$ which is also $Y$-generating:
 $\langle T \rangle_\rat =\P$, and 
 $\langle T \rangle_Y = {\mathbb H}$. 
 
Let $T \subseteq \Pi (M)$ be a basis for $\P$.
\item
Let $T^d_w:=T\cap \Ha^d_w,\, T^d:=T\cap \Ha^d, \, T_w:=T\cap \Ha_w$; then
\[
\P^d_w=  \langle T^d_w \rangle_\rat, \qquad 
\langle \P_w \rangle_{\rat} = \langle T_w \rangle_{\rat}
, \qquad \langle \P^d \rangle_{\rat} = \langle T^d \rangle_{\rat}.
\]

\item
If $m \in M$ has bidegree $(d,w)$, then there are 
$C_{j,t}\in \rat$, for $j>0, t \in T^{d-2j}_{w-2j}$,such that 
\[
\Pi(m)=m+ \sum_{j,t}C_{j,t}Y^jt.
\]
In particular, if $\Pi(m)\neq 0$ then $\Delta(m)=\Delta(\Pi(m))$.  
\end{enumerate} 
\end{proposition}

\section{W=P}
\label{WequalsP}
\stepcounter{subsection}

This section contains the proof of the main result of this paper:
the identification of the perverse Leray filtration associated to the Hitchin map
with the weight filtration of the cohomology of the character variety 
in the case $D=K_{\curv}$ ( \S \ref{dequalschi}), or, if $\deg D>2g-2$, with 
the abstract weight filtration, defined in  Definition \ref{abstrwfiltr}
(\S  \ref{modhiggsbnd}). The latter case is easier and we deal with it first.

\bigskip
Let
$\Ha:= \bigoplus_{d \geq 0} H^d(\mdpgl),$ and let
 $Y:= \alpha \cup : \Ha^d_w \lorw \Ha^{d+2}_{w+2}$ 
 be the operation of cupping with $\alpha$. 

In virtue of the curious hard Lefschetz Theorem \ref{chlwithpoles},
we are in the situation described in \S \ref{bigsm}, with $w_0=g-1+\deg D$. 
The monomials $\psi^{\underline{t}}\beta^s$ (as defined in Proposition~\ref{betapsi}) 
give a set of 
bi-homogeneous elements which is $Y$-generating, for $\Ha$, and, 
by Proposition \ref{projprim}, the elements $\Pi(\psi^{\underline{t}}\beta^s) $
span $\P$. The perverse filtration  is denoted by $\Ha_{\leq}$. As noticed in Corollary \ref{alphample}, $\alpha$ is relatively ample, and it defines a
Deligne decomposition  of $\Ha$, whose summands are  denoted, as in \S
\ref{delspl}, by $Q^{i,j}$ and,  when we want to emphasize the cohomological 
degree, by  $Q^{i,j;\, d}:=Q^{i,j}\cap \Ha^d$.
We set 
$Q:=\bigoplus_{i,d}Q^{i,0;\,d}.$
Clearly,
\begin{equation}
\label{lafa}
\langle Q \rangle_{Y} ={\mathbb H}.
\end{equation}

\begin{proposition}
\label{luck}
Let $u \in \Ha^d_w$, i.e. the weight $w(u)=w$.
\begin{enumerate}
\item
If $\Delta(u)\leq {\rm codim}\,\base^0 \setminus \basegood^0$, then
$u \in \Ha^d_{\leq w}$,  i.e. the perversity $p(u) \leq w$.
\item
If $u \in \P^d_w$, and $\Delta(u)\leq {\rm codim} \,\base^0 \setminus \basegood^0$, then, more precisely,
 $u \in Q^{w,0;\, d} \subseteq  \Ha^d_{\leq w}$.
\end{enumerate}
\end{proposition}
\n{\em Proof.}
Since the monomials $u=\alpha^r \psi^{\underline{t}}\beta^s$ are additive generators,
it is enough to prove 1. for these monomials.
By virtue of  Lemma \ref{fax}, we are  further reduced to  the case  $u=\psi^{\underline{t}}\beta^s$. 
Keeping in mind
the upper bound $p \leq 2$ on the perversity of $\beta$ and $\psi$
given by Theorem \ref{eccolo}, we can apply
the perversity test given by  Proposition \ref{testpr} (where the set $Y$ in loc.cit.
is the present set  $\base^0 \setminus \basegood^0$) and obtain
that 
$u \in \Ha^{4s+3t}_{\leq 2(s+t)}$, as soon as
$
(4s+3t)-(2s+2t) -1= \Delta (\psi^{\underline{t}}\beta^s )-1 < {\rm codim}\base^0 \setminus \basegood^0 .$ This proves 1.

\medskip
In view of Proposition \ref{projprim},
in order to prove the second statement, it is enough to prove that $\Pi(\psi^{\underline{t}}\beta^s) \in Q^{2(s+t),0;\, 3t+4s}$. 

\n
By using induction on $r:=s+t$, we first show that $ \Pi(\psi^{\underline{t}}\beta^s) \in \Ha_{\leq 2(s+t)}$:

\begin{itemize}
\item
For $r=0$ there is nothing to prove.
\item
Suppose we know that $\Pi(\psi^{\underline{t'}}\beta^{s'}) \in \Ha_{\leq 2(s'+t')}$ for all $(s', \underline{t'})$ 
with $s'+t'\leq r-1.$ If $(s, \underline{t})$ is such that $s+t=r$, we may write, see again Proposition \ref{projprim} part 4.,
\begin{equation}
\Pi(\psi^{\underline{t}}\beta^s) = 
  \psi^{\underline{t}} \beta^s + \sum_{(s_i,\underline{t}_i)} C_{s_i,{\underline{t}}_i}
\alpha^{j_i} \Pi(\psi^{{\underline{t}}_i}\beta^{s_i}), \hbox{ with }  j_i=r-s_i-t_i>0.
\end{equation}
By 1., we have that
$ \psi^{ \underline{t}} \beta^s \in \Ha_{\leq 2(s+t)}$. By the inductive hypothesis,
we have that $\Pi(\psi^{{\underline{t}}_i}\beta^{s_i})  \in \Ha_{\leq 2(s_i+t_i)}$ so that,
 by Lemma \ref{fax}, we have that $\alpha^{j_i}\Pi(\psi^{{\underline{t}}_i}\beta^{s_i}) \in \Ha_{\leq 2r}$.
The conclusion  on the sum $\Pi(\psi^{\underline{t}}\beta^s) \in \Ha_{\leq 2r}$ follows.
\end{itemize}
\n
By definition of primitivity,
$\alpha^{w_0-2(s+ t)+1}\Pi(\psi^{\underline{t}}\beta^s)=0.$
Since $ \Pi(\psi^{\underline{t}}\beta^s) \in \Ha_{\leq 2(s+t)}$,
the non-mixing lemma \ref{nomix}, coupled with the equality above, implies that $\Pi(\psi^{\underline{t}}\beta^s) \in Q^{2(s+t),0}$. 
\blacksquare

\subsection{The case  $\deg D>2g-2$, $\G=\PGL_2, \GL_2$ }
\label{dbiggerchi}

\bigskip
In this section we assume $n:=\deg D+2-2g>0$.

Recall the relations (\ref{relaht}) between the generators of the cohomology ring.  
Theorem \ref{hauthaddmntm} readily implies the following:

\begin{proposition}
\label{betapsi}
For nonnegative integers $t_1, \ldots, t_{2g}$,
let us write\footnote{As the $\psi$ classes have degree $3$ 
they anticommute, so we could assume $t_i=0,1$.}  
$\psi^{\underline{t}}$ 
for $\psi_1^{t_1} \ldots \psi_{2g}^{t_{2g}}$ and  let us  set $t:= \sum_{i=1}^{2g} t_i$ .
Then: 
\begin{enumerate}
\item
If $n:=\deg D+2-2g>0,$  then $ \beta^s \, \psi^{\underline{t}}=0$  for $ 2s+t\geq \deg D.$
\item
If $D=K_{\curv},$  then $ \beta^s \, \psi^{\underline{t}}=0$ for  $2s+t\geq 2g-2$, 
unless $\psi^{\underline{t}}=\gamma^r$,  with $r+s=g-1.$ 
\end{enumerate}
\end{proposition}
\n{\em Proof.}
Let us assume $n>0$. The monomial $\beta^s \, \psi^{\underline{t}}$ is a sum of terms of the form
$A_i \gamma^i\beta^s$ with $A_i \in \Lambda_0^{t-2i}.$ From the relations 
in Theorem \ref{hauthaddmntm}, it follows  that 
$$
\hbox{ if } \gamma^i\beta^s \in I^{g+2i-t}_{n+t-2i}, \hbox{ then }A_i \gamma^i\beta^s=0.
$$
By the inequalities (\ref{indexrange}) and Remark \ref{verytrivial}, 
$\gamma^i\beta^s \in I^{g+2i-t}_{n+t-2i}$  if
$ 2i+2s\geq 2(g-t+2i)-2+n+t-2i$, that is, if $ 2s+t\geq 2g-2+n=d$.
If $n=0$, we proceed as above, noting that Remark \ref{verytrivial} fails exactly in 
the case $i=0$, in which case we find that $\beta^{g-1}, \gamma \beta^{g-2},\ldots, \gamma^{g-1} \neq 0$. 
\blacksquare

\begin{theorem}
\label{w=pwithpoles}
The  abstract weight filtration $W'_{\bullet}$ on $H^*(\mdpgl)$ (Definition \ref{abstrwfiltr}) coincides with the
perverse Leray filtration associated with the Hitchin map $\hitmaph$: 

\n
For every integer $i$, we have
\begin{equation}
\label{w=pwithpoleseq}
H^*_{\leq i}(\mdpgl)=W'_iH^*(\mdpgl) =\langle  \alpha^r \psi^{\underline{t}}\beta^s   \rangle_{2(r+s+t)\leq i}.
\end{equation}
More precisely, the isobaric decomposition (\ref{isobbaric}) 
coincides with the Deligne decomposition (\ref{qdec}) in \S \ref{delspl}
associated to $\alpha$: for every $w$, we have $\P_w = Q^{w,0}$. 
\end{theorem} 
\n{\em Proof.}
Since the statement about the equality of the filtrations follows at once from the second on
the equality of the internal direct sum decompositions, we prove the latter one. 
By  Proposition \ref{betapsi}, we have that $\psi^{\underline{t}} \beta^s =0$ as soon as
$\Delta(\psi^{\underline{t}}\beta^s )=t+2s \geq \deg D$.  It follows that
we only need to consider the monomials 
$ \psi^{\underline{t}}\beta^s$ with $t+s<\deg D$. By  Lemma \ref{smingo}, 
${\rm codim}\, \base^0 \setminus \basegood^0=\deg D$. 
We can then  apply Proposition \ref{luck} and deduce that 
\begin{equation}
\label{mesa}
\Pi ( \psi^{\underline{t}}\beta^s )\in Q^{2(s+t),0;\, 4s+3t}.
\end{equation}
Since  $\{\Pi ( \psi^{\underline{t}}\beta^s)\}$ is a set of generators for the 
primitive space   $\P$,
and since $\Pi$ strictly preserves the weights by construction,
we deduce that  $\P_w \subseteq Q^{w,0}$ for all $w$. 
Since, by (\ref{primideco}),  $\langle \P \rangle_Y ={\mathbb H}$ and, by (\ref{lafa}), $\langle Q \rangle_Y ={\mathbb H}$, 
it follows that $\P_w$ and $Q^{w,0}$ coincide for all $w$.
\blacksquare

\begin{theorem}
The abstract weight filtration $W'_{\bullet}$ on $H^*(\mdgl)$ coincides with the 
perverse filtration associated with the Hitchin map. 
More precisely, the isobaric decomposition coincides with the Deligne decomposition
associated to 
$\tilde{\alpha}=\alpha \otimes 1 + 1\otimes (\sum \epsilon_i \epsilon_{i+g})$.
\end{theorem} 
\n{\em Proof.}
The statement follows from the isomorphism (\ref{fe1}) 
and Lemma \ref{fax}, as cupping with $\epsilon_i$ increases the perversity exactly by one.
\blacksquare

\subsection{The case  $D=K_{\curv}$, $\G=\PGL_2, \GL_2$}
\label{dequalschi}

In this section, we set
$$
\Ha=\bigoplus_{d\geq 0}H^d(\mdpgld).
$$
\begin{lemma}
\label{betapsieasy} 
We have the following
$$
\alpha^r \psi^{\underline{t}}\beta^s  \in \Ha^{2r+3t+4s}_{\leq 2(r+ t+s)}, 
\,\, \Pi(\psi^{\underline{t}}\beta^s)\in Q^{2(t+s),0;\, 3t+4s},
$$ 
unless $\psi^{\underline{t}}\beta^s= \gamma^v\ \beta^s$ with $v+s=g-1$ (cfr. 2. in Proposition \ref{betapsi}).
\end{lemma}
\n{\em Proof.} 
Since, by Lemma \ref{fax}, cupping with $\alpha$ increases the perversity by at most $2$,
we may suppose $r=0$. Note that we are excluding precisely the classes
in the statement of Proposition \ref{betapsi} (case $D=K_{\curv}$).
By this same proposition, we may thus assume that $2s+t<2g-2= {\rm codim }\,\base^0 \setminus \basegood^0 +1$, where 
the last equality results from Lemma \ref{smingo}.
The result  follows from Proposition \ref{luck}.
\blacksquare

\medskip
\begin{remark}
\label{mancasolouno}
{\rm The argument above breaks when dealing with the classes $\gamma^r \beta^{g-1-r}$
that we have excluded from the statement.
To check that $\gamma^r \beta^{g-1-r} \in \Ha^{2r+4g-4}_{\leq 2g-2+2r}$ we should 
consider a linear subspace of dimension $(2r+4g-4)-(2g-2+2r)-1=2g-3$, which is exactly 
the codimension of the ``bad locus'' $\base^0 \setminus \basegood^0$. 
On the other hand, a  general linear subspace of dimension one less, i.e. $2g-4$, 
misses the bad locus,   and thus yields, by Theorem \ref{gdpf}, 
the
following  upper bound on the perversity
\begin{equation}
\label{mancauno}
\gamma^r \beta^{g-1-r} \in \Ha^{2r+4g-4}_{\leq 2g-1+2r}.
\end{equation}
While this  upper bound is not sufficient  for our purposes,  it is used in what follows.
}
\end{remark}
\medskip
\begin{remark}
\label{relabeta}
{\rm
The relations 
$\rho_{1,s,g-1-s}^c \in I^g_0$  in 
(\ref{relaht}) show that 
for all $r$, the class $\gamma^r \beta^{g-1-r}$ is a multiple of $\alpha^r \beta^{g-1}$.
}
\end{remark}

\bigskip
As pointed out several times, in view of Lemma 
\ref{fax}, 
cupping with $\alpha$ is harmless for us and
we are reduced to prove that $\beta^{g-1} \in \Ha^{4g-4}_{\leq 2g-2}.$
The remainder of the analysis is devoted to improve
the upper bound  (\ref{mancauno}),  by one unit,
i.e. to proving that  $\beta^{g-1} \in \Ha^{4g-4}_{\leq 2g-1}$.

\begin{lemma}
\label{getsmart}
For every $s$ in the range $0 \leq  s \leq g-1$ we have
$$
\beta^s \in \P^{4s}_{2s}.
$$
In particular the classes $\beta^s$ are not divisible by $\alpha$.
\end{lemma}
\n{\em Proof.} 
Recall that  $\Ha^{d} = 0,$ for every $d > 6g-6$ (see Remark \ref{topdimns}).
Clearly, since weights $w$ are strictly multiplicative, $\beta^s \in \Ha^{4s}_{2s}$. 
Since, in the terminology of \S\ref{bigsm}, $\Ha$ is a bi-graded $\slt$-module
with $w_0=3g-3$, we have that $\alpha^{3g-3 -2s} \beta^s \neq 0$.
On the other hand $\alpha^{3g-3 -2s +1} \beta^s \in \Ha^{6g-4} = \{0\}$. 
These are precisely the conditions defining  primitivity.
\blacksquare

\bigskip
\bigskip

\n Set 
\begin{equation}
\label{defijj}
\JJ:= \{ \beta^s \psi^{\underline{t}}\in \Ha \, \hbox{ such that } 
\, \Delta(\beta^s \psi^{\underline{t}}) \leq 2g-3 \,\hbox{ if } t \hbox{ is odd, and }\, 
\Delta(\beta^s \psi^{\underline{t}}) \leq 2g-4 \,\hbox{ if } t \hbox{ is even} \}, 
\end{equation}
and
\[ 
\widetilde{\Ha }: =  \langle  \, \JJ \, \rangle _{\alpha}=  \langle  \, \Pi(\JJ)  \, \rangle _{\alpha}.
\]

By Remark \ref{relabeta}, if $r\geq 1$, then  $\gamma^r \beta^{g-1-r}$ is divisible by $\alpha$. 
In this case, the projection to the primitive space $\Pi(\gamma^r \beta^{g-1-r})=0$. By Lemma
\ref{getsmart} we have  $\Pi(\beta^{g-1})=\beta^{g-1}$, 
and $\gamma^r \beta^{g-1-r} \in  \langle  \beta^{g-1} \rangle_{\alpha}$.
Since $\Delta (\beta^s \psi^{\underline{t}})=2s+t$, 
point 2 of the statement of Proposition \ref{betapsi} can be rephrased by saying that,
unless $\beta^s \psi^{\underline{t}}\in \langle  \beta^{g-1} \rangle_{\alpha}$,
we have that $\beta^s \psi^{\underline{t}}\in \widetilde{\Ha}$, so that
we have an {\em $\slt$-invariant decomposition}
\begin{equation}
\label{isobetastrng}
\Ha\,=\,\widetilde{\Ha} \, \bigoplus\, \langle  \beta^{g-1} \rangle_{\alpha}.
\end{equation}

\begin{lemma}
\label{ndosta}
The following facts hold:
\begin{enumerate}
\item
$\Ha ^d = \widetilde{\Ha}^d $ unless  $d+4-4g$ is  even non-negative. 

\item
$\dim \Ha ^{4g-4+2k} = \dim \widetilde{\Ha}^{4g-4+2k} +1 $ for  $0 \leq  k \leq g-1,$

\item
In the range  of point 1., $\P^d_w=Q^{w,0;\,d}$. 
In particular, all the non-zero summands $Q^{i,j,d}$ in the Deligne decomposition  satisfy
$d-i-2j \leq 2g-3$ if $d$ is odd, and $d-i-2j \leq 2g-4$ if $d$ is even.

\item
In the range of point 2., there is at most one 
non-zero, necessarily one-dimensional, 
summand $Q^{i,j,d}$ satisfying $d-i-2j>2g-4$. 
\end{enumerate}
\end{lemma}
\n{\em Proof.}
Point 1. and 2. follow immediately from the fact that
the $\alpha$-string  $\langle  \beta^{g-1} \rangle_{\alpha}$ 
contains the classes $\beta^{g-1}, \alpha\beta^{g-1}, 
\ldots, \alpha^{g-1}\beta^{g-1}$ whose cohomological degrees are 
$4g-4, 4g-2\ldots,  6g-6$.

Notice that the monomials $\beta^s \psi^{\underline{t}}$ in $\JJ$ are precisely those to which
Lemma \ref{betapsieasy} applies, hence
$$
\P^d_w \cap \,\widetilde{\Ha}^d_w \subseteq Q^{w,0;\, d},
$$
and, for $j\geq 0$,
$$
(\alpha^j \,\P^{d-2j}_{w-2j}) \cap \widetilde{\Ha}^{d}_w \subseteq \alpha^jQ^{w-2j,0;\,d-2j} =Q^{w-2j,j;\, d}.
$$
Combining this fact with the $\alpha$-decompositions 
(\ref{isobbaric}) and (\ref{qdec}) in \S\ref{delspl}, 
and with the points 1. and 2. which we just proved,
we immediately obtain  points 3. and 4.
\blacksquare

\begin{lemma}
\label{dfer}
Either $\beta^{g-1} \in  Q^{2g-2,0;\, 4g-4}$ or 
$\beta^{g-1} \in  Q^{2g-1,0;\, 4g-4}$.
\end{lemma}
\n{\em Proof.}
By  (\ref{mancauno}), with $r=0$,   we have that  
$\beta^{g-1} \in \Ha^{4g-4}_{\leq 2g-1}$. 
There is hence at least one non-zero summand  
$Q^{i_0,j_0;\, 4g-4}$ in the Deligne decomposition, 
satisfying 
\begin{equation}
\label{inequa}
i_0+2j_0 \leq 2g-1. 
\end{equation}
Suppose $j_0\neq 0$; by  (\ref{12w}), we have 
$$
Q^{i_0,j_0;\, 4g-4}=\alpha^{j_0}Q^{i_0,0;\, 4g-4-2j_0}.    
$$   
Since $Q^{i_0,0;\, 4g-4-2j_0}\neq \{0\}$,
we have, by point 3. of Lemma \ref{ndosta}, 
$(4g-4-2j_0)-i_0-2j_0  \leq 2g-4$, which contradicts the inequality 
(\ref{inequa}) above, showing that $j_0=0$.  By Corollary \ref{ndomett2}, we then have  
%$\Ha^{4g-4}_{\leq 2g-3}=0$ and that 
\[
\beta \in \Ha^{4g-4}_{\leq 2g-1} = Q^{2g-2, 0;4g-4} \oplus Q^{2g-1,0; 4g-4} 
\]
\blacksquare

\begin{proposition}
\label{betaisfine} We have that
\[
\beta^{g-1} \in Q^{2g-2,0; \, 4g-4}.
\]
\end{proposition}
\n{\em Proof.}
Suppose the statement is false. By Lemma \ref{dfer},  the space $ Q^{2g-2,0; \, 4g-4}=0$
and  the class $\beta^{g-1} \in Q^{2g-1,0; \, 4g-4}$.
From the property  of the Deligne decomposition expressed by  (\ref{12w}), 
it follows that, for $j\leq g-2$, we have 
$0\neq \alpha^{j}\beta^{g-1} \in   Q^{2g-1,j; \, 4g-4+2j}.$  
By using the decomposition (\ref{isobetastrng}), and Lemma \ref{ndosta},
it follows that,  for every even non-negative integer $d<6g-6$,
\begin{equation}
\label{prcs1}
\Ha^{d}=\widetilde{\Ha}^{d} =\bigoplus_{d-i-2j\leq 2g-4} Q^{i,j; \, d}\, \hbox{ if }\, d<4g-4,
\end{equation}
and
\begin{equation}
\label{prcs2}
\Ha^{d}=\left( \bigoplus_{d-i-2j\leq 2g-4} Q^{i,j; \, d}\right)
\bigoplus Q^{2g-1, j_0;\, d}, \hbox{ with }j_0= d/2-2g+2, \hbox{ if }\, d\geq 4g-4.
\end{equation}
In this latter case
$Q^{2g-1,j_0 ; \, d}=\langle \, \alpha^{j_0}\beta^{g-1} \,\rangle_{\rat}$.

Applying one of the defining properties of the Deligne decomposition, 
i.e. the second equation in (\ref{novel}), (with $f=3g-3$ and $i=2g-1$), 
to $\alpha^{g-2}\beta^{g-1}\in   Q^{2g-1, g-2; \, 6g-8}$,
we have the following upper bound for the perversity
\[
\alpha^{g-1}\beta^{g-1} = \alpha (\alpha^{g-2} \beta^{g-1}) \in \Ha^{6g-6}_{\leq 4g-5}.
\]
From Corollary \ref{nozeroperv} it follows that
\[
\Ha^{6g-6}_{\leq 3g-3}=\{0\}.
\]
It follows that there exists  $1 \leq r \leq  g-2$ such that 
%{\bf --- improve display---}
\begin{equation}
\label{prvst}
\alpha^{g-1} \beta^{g-1} \in \Ha^{6g-6}_{\leq 3g-3+r} \hbox{ and }  
\alpha^{g-1} \beta^{g-1} \notin \Ha^{6g-6}_{\leq 3g-4+r}. 
\end{equation}

From this and from point 4. of Lemma \ref{ndosta} it follows
that $\alpha^{g-1} \beta^{g-1}$ must belong
to the unique summand $Q^{3g-3-r,r;\,6g-6}$ with  $1 \leq r \leq  g-2$. 

\medskip
Since $r\geq 1$, the relation (\ref{12w}) gives 
$$
Q^{3g-3-r,r;\,6g-6}=\alpha^r Q^{3g-3-r,0;\,6g-6-2r}.
$$
On the other hand, (\ref{prcs2}), 
with $d=6g-6-2r$ shows  that $ Q^{3g-3-r,0;\, 6g-6-2r}=0$.
\blacksquare

\begin{remark}
\label{gequals2}
{\rm For $g=2$ the previous argument shows that $\beta \in \Ha_{\leq 2}^4$,
as anticipated in Remark \ref{genus2}.
}
\end{remark}

\bigskip 
Proposition \ref{betaisfine} allows us to complete point 3. in Lemma \ref{ndosta}:
\[
\P^d_w=Q^{w,0;\,d} \hbox{ for all } d,w. 
\]
\n
We finally summarize what we proved in the following theorem, which is the main result of this paper: 
\begin{theorem}
\label{maintm}
The non-Abelian Hodge theorem for $\PGL_2$, (resp. $\GL_2$)  identifies

\begin{itemize}
\item
the perverse Leray  filtration with the weight filtration: for every integer $i$, we have 
$$
H^*_{\leq i}(\mdpgl_{\rm Dol})\simeq W_{2i}H^*(\pglM_\B)=W_{2i+1}H^*(\pglM_\B), \qquad H^*_{\leq i}(\higgsbu_{\rm Dol})\simeq W_{2i}H^*(\M_\B)=W_{2i+1}H^*(\M_\B).
$$

\item the relative hard Lefschetz theorem (\ref{rhlg}) relative to the Hitchin map $\hitmaph$ (resp. $\hitmap$) and to the relatively ample class $\alpha$
(resp. $\tilde{\alpha}=\alpha \otimes 1 + 1\otimes (\sum \epsilon_i \epsilon_{i+g})$ ) 
with the curious hard Lefschetz theorem~\ref{curioushrdlf}: 
%setting $\hat{f}:=\frac{1}{2}\dim \mdpgl_{\rm Dol}=3g-3$, and $f=\hat{f}+g=4g-3=\frac{1}{2}\dim \higgsbu_{\rm Dol}=\dim \M_\B$, we have 
$$
\xymatrix{
H^*_{3g-3 -i}(\mdpgl_{\rm Dol}) \ar[r]^(.4){\simeq}\ar[d]^{\alpha^i}_{\simeq}   & {\rm Gr}^W_{6g-6-2i}H^*(\pglM_\B) \ar[d]^{\alpha^i}_{\simeq}  & {\rm and } &  H^*_{4g-3 -i}(\higgsbu_{\rm Dol}) \ar[r]^(.4){\simeq}\ar[d]^{\tilde{\alpha}^i}_{\simeq}   & {\rm Gr}^W_{8g-6-2i}H^*(\M_\B) \ar[d]^{\tilde{\alpha}^i}_{\simeq}      \\
H^{*+2i}_{3g-3 + i}(\mdpgl_{\rm Dol}) \ar[r]^(.4){\simeq}                                  & {\rm Gr}^W_{6g-6+2i}H^{*+2i}(\pglM_\B)  &  & H^{*+2i}_{4g-3+i}(\higgsbu_{\rm Dol}) \ar[r]^(.4){\simeq}                                  & {\rm Gr}^W_{8g-6+2i}H^{*+2i}(\M_\B) 
}
$$
\item
the  Deligne $Q$-splitting (\ref{qdec}) associated with the relatively ample class $\alpha$
(resp. $\tilde{\alpha}$) with the isobaric  splitting (\ref{isobbaric}).
\end{itemize}
\end{theorem} 
\n{\em Proof.}
It follows from Lemma \ref{betapsieasy} that every cohomology class of the additive basis, 
with the possible exception of those in the $\alpha$-chain of $\beta^{g-1}$ (see Remark \ref{relabeta}), satisfies ``W=P". Lemma \ref{betaisfine}
show that also $\beta^{g-1}$ satisfies the condition "W=P", and so do the classes in its $\alpha$-chain, thus proving the first statement for $\PGL_2$. 
The two other statements follow similarly. 
The extension to $\GL_2$ 
follows immediately from the isomorphism \ref{fe1}
and Lemma \ref{fax}, as cupping with $\epsilon_i$ increases the perversity exactly by one.
\blacksquare

\begin{remark}
\label{othersplit}
{\rm
We have made a heavy use and made explicit
Deligne's  splitting of the direct image complex via the use of a
relatively-ample line bundle. This general splitting mechanism
 is described  in \ci{deligneseattle}.
The same paper details the construction of two additional splittings.
As the simple example of  the ruled surface $\pn{1}\times \pn{1} \to \pn{1}$
already shows, in general, the three splittings differ. It is possible to show
that, in the case of all three  Hitchin maps 
considered in this paper, all three splittings coincide when viewed in cohomology.
}
\end{remark}

\begin{remark}
\label{ptow}
{\rm
Let $u \in Q^{i,j}$. Such a class has perversity $p:=i +2j$,
when viewed as a cohomology class for the Higgs  moduli space.
The main result of this paper, i.e. $P=W$, shows that the non Abelian Hodge theorem
turns this class into  a $(p,p)$-class for the split
Hodge-Tate mixed Hodge structure on the associated character variety.
}
\end{remark}

\subsection{$\SL_2$}
\label{SL2}

\rm 
In this section, for the sake of notational simplicity, we will denote simply by  $\slM$ the moduli space $\slM_{\rm Dol }$ of stable Higgs  bundles on $C$ of rank $2$ and fixed determinant of degree $1$. Let $\check{\chi}:\slM\to {{\base}^0}$ the Hitchin map.  
The action of  $\Gamma=\Pic^0_C[2] \simeq \Z_2^{2g}$  on $\slM$ by tensorization preserves the map $\hitmapc$, and, as discussed in \S \ref{comp3gra}, we have a direct sum decomposition  according to the characters of $\Gamma$
\begin{equation}
\label{decdirim}
\hitmapc_*\rat_{\slM}\simeq \bigoplus_{\kappa \in \hat{\Gamma}}(\hitmapc_*\rat_{\slM} )_{\kappa}= \left(\hitmapc_*\rat_{\slM}\right)^{\Gamma} \bigoplus  \left(\hitmapc_*\rat_{\slM}\right)_{{\rm var}} ,
\end{equation}
where we set $\left(\hitmapc_*\rat_{\slM}\right)_{{\rm var}} =  \bigoplus_{0\neq \kappa \in \hat{\Gamma}}(\hitmapc_*\rat_{\slM} )_{\kappa} .$    Taking cohomology,  (\ref{decdirim}) gives  
\beq\label{decdirim2}
H^*(\slM)=\bigoplus_{\kappa\in \hat{\Gamma}} H^*(\slM)_\kappa = H^*(\slM)^{\Gamma} \bigoplus H_{{\rm var}}^*(\slM),
\eeq
where $H^*(\slM)_\kappa=H^*({{\base}^0},(\hitmap_*\rat_{\slM} )_{\kappa})$, is the subspace of $H^*(\slM)$ where $\Gamma$ acts via the character $\kappa$,
and $H_{{\rm var}}^*(\slM):=\bigoplus_{0 \neq \kappa \in \hat{\Gamma}} H^*(\slM)_{\kappa}= H^*(\base^0,(\hitmap_*\rat_{\slM} )_{{\rm var}})$ is the {\em variant} part of $H^*(\slM)$.

Recall from \cite{hit}, formula after (7.13), that 
\beq\label{dimension}
\dim H_{{\rm var}}^{4g+2d-5}(\slM)=\left\{\begin{array}{ll}
 (2^{2g}-1)\bino{2g-2}{2g-2d-1}    & \hbox{ if }d=1,\dots,g-1 , \\
0   &  \hbox{ otherwise }.
\end{array}\right. 
\eeq

\medskip
\n
For $\gamma\in \Gamma\subseteq \Pic^0_{\curv}$, let $L_\gamma $ be the corresponding order $2$ line bundle, and let $i_\gamma$ be 
the ``squaring'' map 

\begin{equation}
 i_\gamma: H^0(\curv, K_{C}\otimes L_\gamma) \lorw H^0(\curv, 2K_C)={\base}^0, \qquad  i_\gamma(a)=a\otimes a,
\end{equation}
%\bes i_\gamma: H^0(\curv, K_{C}\otimes L_\gamma)&\to& H^0(\curv, 2K_C)={\base}^0\\  a&\to& a\otimes a,\ees 
with image ${\base}^0_\gamma:={\rm Im}(i_\gamma)\subset {\base}^0.$ 
By the Riemann-Roch theorem,
$\dim(\base^0_0)=g$ and,  if 
$\gamma\in \Gamma^*=\Gamma\setminus\{0\},$ then
$\dim(\base^0_\gamma)=g-1$. 
Points in $\cup_{\gamma\in \Gamma^*}\base^0_\gamma$ are called {\em endoscopic points}. Set
$$\base^0_{ne}:=\base^0\setminus \cup_{\gamma\in \Gamma\setminus 0} \base^0_\gamma , \hbox{ and } \slM_{ne}:=\check{\chi}^{-1}(\base^0_{ne}).$$ 

Our goal is to prove the following: 
\begin{proposition} 
\label{trivial}
Let $s\in\base_{ne}^0$, and $\slM_s:=\hitmapc^{-1}(s)$ the fiber of the Hitchin fibration over $s$.
The group $\Gamma$ acts trivially on $H^*(\slM_s)$.
\end{proposition}
\medskip
\n
Proposition~\ref{trivial} immediately implies:
\begin{corollary}
The variant complex $\left(\hitmapc_*\rat_{\slM}\right)_{{\rm var}} $
is supported on $\cup_{\gamma\in \Gamma^*}\base^0_\gamma=\base^0 \setminus \base^0_{ne}$.
\end{corollary}
\n{\em Proof.}
Taking the cohomology sheaves of the decomposition \ref{decdirim}
we have that $\Gamma$ acts as multiplication by the character $\kappa \in \hat{\Gamma}$
on ${\mathcal H}^i(\left(\hitmap_*\rat_{\slM} \right)_{\kappa})$.
By Proposition \ref{trivial} 
$$
\hbox{ if } s \in \base^0_{ne} \hbox{ and }   \kappa \neq 0, \hbox{ then }   {\mathcal H}^i(\left(\hitmap_*\rat_{\slM} \right)_{\kappa})_s=0 \hbox{ for all }i, 
$$
therefore the restriction of 
$\bigoplus_{\kappa \in \hat{\Gamma}\setminus \{0\}}\left(\hitmap_*\rat_{\slM} \right)_{\kappa}$
to $\base^0_{ne}$ vanishes. 
\blacksquare

\medskip
\n
In order to  prove  Proposition~\ref{trivial}   
we show that the action of $\Gamma$ on $\slM_s$, for $s\in\base_{ne}^0$, 
is the restriction to $\Gamma$ of an action of a connected group  $\prym_{C_s/C}$;
as such $\Gamma$ acts trivially on $H^*(\slM_s)$. We
begin with some preliminary considerations on the norm map.

Fix $s\in\base_{ne}^0$, and let $\pi:C_s\to C$ be the corresponding spectral
cover. When $s\in \base^0_0$ we have that $C_s=C_1\cup C_2$ is reducible, 
otherwise $C_s$ is an integral curve. We denote by 
$\Pic^0_{C_s}$ the connected component of the identity of $\Pic_{C_s}$. 
Denote by $\nu:\widetilde{C}_s\to C_s$  the normalization. 
Define the norm map $ {\rm Nm}_{C_s/C}:\Pic^0_{C_s}\to {\rm \Pic}^0_C $
by \beq\label{normnorm}  {\rm Nm}_{C_s/C}:={\rm Nm}_{\widetilde{C}_s/C}\circ \nu^*\eeq
where, for a divisor $D$ on the non-singular curve $\widetilde{C}_s$, the norm map ${\rm Nm}_{\widetilde{C}_s/C}
({\mathcal O}(D))={\mathcal O}((\pi\circ \nu)_* D )$ is the classical one. Consequently ${\rm Nm}_{C_s/C}:\Pic^0(C_s)\to \Pic^0(C)$ is a group homomorphism.

By Proposition 3.8 in \cite{hausel-pauly} we have the following alternative formula for the norm map:
\beq  \label{normap}{\rm Nm}_{C_s/C}({\mathcal L})=\det(\pi_*({\mathcal L}))\otimes \det(\pi_*({\mathcal O}))^{-1}.\eeq

\medskip
As ${\rm Nm}_{C_s/C}:\Pic^0_{C_s}\to\Pic^0_C$ is a group homomorphism, the kernel $\Prym_{C_s/C}:={\rm Nm}_{C_s/C}^{-1}(\mathcal O_C)$ is a subgroup of $\Pic^0_{C_s}$. 

\begin{lemma}
\label{prymcnnct}
If $s\in\base_{ne}^0$, then the group $\Prym_{C_s/C}={\rm Nm}_{C_s/C}^{-1}(\mathcal O_C)$ is connected.
\end{lemma}
\n
{\em Proof.}
The case  of an integral spectral curve $C_s$ is treated in \cite{ngo2} \S11.
 The argument is easily adapted to the case of a reducible and reduced curve. For a proof more in the spirit of the present paper 
see \cite{hausel-pauly}, Theorem 1.2.
\blacksquare

\begin{lemma} 
\label{dettensor}
Let $\pi:X\to Y$ be a degree two map from a reduced projective curve $X$ to a non-singular projective curve $Y$. Let ${\mathcal E}$ be a rank $1$ torsion free sheaf on $X$ and $\mathcal L$ an invertible one. Then \beq \label{determinant}\det(\pi_*({\mathcal L}\otimes{\mathcal E}))=\det(\pi_*({\mathcal E}))\otimes {\rm Nm}_{X/ Y}({\mathcal L}).\eeq
\end{lemma}
\begin{proof} First we note that by Theorem~\ref{genfctscj} we have that there is a unique partial normalization $\pi^\prime:X^\prime\to X$ and an invertible sheaf ${\mathcal L}^\prime$ on $X^\prime$ such that ${\mathcal E}=\pi_*({\mathcal L}^\prime)$. The multiplicativity of
the norm map and \eqref{normap} imply $$\det((\pi\circ \pi^\prime)_*((\pi^\prime)^*({\mathcal L})\otimes{\mathcal L}^\prime))=\det((\pi\circ \pi^\prime)_*({\mathcal L}^\prime))\otimes {\rm Nm}_{X^\prime/ Y}((\pi^\prime)^*({\mathcal L})).$$ This together with ${\rm Nm}_{X^{\prime}/Y}\circ ({\pi^\prime})^*={\rm Nm}_{X/Y}$ yield the result.
\end{proof}

\n{\em Proof of Proposition~\ref{trivial}}. We first prove that $\Prym_{C_s/C}$ acts on the fiber $\slM_{s}$. 
Recall that  $\slM_{s}$ can be identified with pure rank $1$  torsion-free sheaves ${\mathcal E}$ on $C_s$ for which the corresponding Higgs bundle $(\pi_*({\mathcal E}),\phi_{\mathcal E})$ is stable, and $\det(\pi_*({\mathcal E}))\cong\Lambda$. (Note that $\tr(\phi_{\mathcal E})=0$ is automatic as $s\in H^0(C;2K_{\curv})$.) Now, ${\mathcal L}\in\Prym_{C_s/C}$ acts on ${\mathcal E}$ as ${\mathcal E} \mapsto {\mathcal L}\otimes{\mathcal E}$, where the result  is again a pure
  rank $1$ torsion-free sheaf on $C_s$. It follows that
$(\pi_*({\mathcal L}\otimes{\mathcal E}),\phi_{{\mathcal L}\otimes{\mathcal E}})$ is a rank 
$2$ Higgs bundle. By Lemma~\ref{dettensor}, we have that \beq \label{fixedet}\det(\pi_*({\mathcal L}\otimes{\mathcal E}))=\Lambda.\eeq 

\smallskip
\n
Finally, we prove that tensoring with an element of  $\Prym_{C_s/C}$ preserves stability.
Assume there is a rank $1$ Higgs subbundle
$(L_{\mathcal F},\phi_{\mathcal F})$ of $(\pi_*({\mathcal L}\otimes{\mathcal E}),\phi_{{\mathcal L}\otimes{\mathcal E}})$  corresponding to the torsion-free sheaf ${\mathcal F}\subset {\mathcal L}\otimes{\mathcal E}$ on $C_s$.
Then the spectral curve $C_{\mathcal F}={\rm supp}({\mathcal F})$ of $(L_{\mathcal F},\phi_{\mathcal F})$ must be a subscheme  of $C_s$, so that $C_{\mathcal F}\to C$ is degree one. Thus $C_s=C_1\cup C_2$ must be reducible and $\mathcal F$ be supported on one of the components, say  $C_{\mathcal F} =C_1$; with $\pi_1:=\pi|_{C_1}:C_1\stackrel{\cong}{\to} C$ an isomorphism. As ${\mathcal L}_{\mathcal F}=\pi_*({\mathcal F})$ we can identify it with ${\mathcal F}|_{C_1}$. Finally since ${\mathcal L}\in \Prym_{C_s/C}\subset {\rm Pic}^0_{C_s}$ we have that $\deg({\mathcal L}|_{C_1})=0$ and so \bes \deg(L_{{\mathcal L}^{-1}\otimes{\mathcal F}})=\deg({\mathcal L}^{-1}|_{C_1}\otimes{\mathcal F}|_{C_1})=\deg({\mathcal F}|_{C_1})=\deg(L_{\mathcal F}).\ees
To summarize $(\pi_*({\mathcal E}),\phi_{\mathcal E})$ and $(\pi_*({\mathcal L}\otimes{\mathcal E}),\phi_{{\mathcal L}\otimes{\mathcal E}})$ have the same degree sub Higgs-bundles. It follows that if $(\pi_*({\mathcal E}),\phi_{\mathcal E})$ is stable so is
$(\pi_*({\mathcal L}\otimes{\mathcal E}),\phi_{{\mathcal L}\otimes{\mathcal E}})$.

\smallskip
\n
We thus proved that the connected group scheme $\Prym_{C_s/C}$ acts on $\slM_{s}$. For any 
order $2$ line bundle $L\in \Gamma={\rm Pic}^0_{C}[2]$, we have ${\rm Nm}_{C_s/C}(\pi_s^*(L))=L^2={\mathcal O}_{C}$, therefore
$\pi^*(L) \in \Prym_{C_s/C}$ and consequently $\Gamma\subset \Prym_{C_s/C}$ acts trivially on $H^*(\slM_{s})$. 

\blacksquare

\begin{remark} For a  more detailed calculation of the group of components of Prym varieties of spectral covers see \cite{hausel-pauly}.
\end{remark}

\medskip
\n
We introduce the notation $H^k_{\leq p, {\rm var}}:=P_p\,\cap H^k_{{\rm var}}(\slM)$ and $H^k_{p,{\rm var}}:={\rm Gr}^P_{p}(H^k_{{\rm var}}(\slM)).$
\begin{theorem}\label{detpfvar} 
	The perverse Leray filtration on $H_{{\rm var}}^*(\slM)$ satisfies
	\[
	0  = H_{\leq k -2g+1, {\rm var}}^k(\slM) \subseteq H_{ \leq k -2g+2, {\rm var}}^k(\slM)
	=H_{{\rm var}}^k(\slM).
	\]
\end{theorem}
\begin{proof}
Since $\dim{ \base^0 \setminus  {\base}^0_{ne}} =g-1$, a general $(2g-3)$-dimensional
linear subspace $\Lambda^{2g-3}$ of the $(3g-3)$-dimensional affine base $\base^0$
lies entirely inside $\base^0_{ne}$. 
By Proposition~\ref{trivial}, the restriction of  a  class in $H_{{\rm var}}^*(\slM)$   to $H^*(\slM_{ne})$,
and thus to $H^*(\slM|_{\Lambda^{2g-3} })$,
is trivial. This fact, coupled with the test for perversity given by Theorem \ref{gdpf}, implies the inclusion $H_{{\rm var}}^k(\slM) \subseteq 
H_{\leq k -2g+2,{\rm var}}^k(\slM)$.

We are left with proving that
\beq \label{perv2} H_{\leq k-2g+1,{\rm var}}^k(\slM)=0.
	\eeq
%Therefore it will vanish over any generic  $2g-3=3g-3-(g-1)-1$ dimensional subvariety
%$\Lambda^{2g-3}$ of $\base^0$ (namely a generic such subvariety will avoid the endoscopic points %in $\base^0$, 
%which form a $g-1$ dimensional subvariety). Consequently \eqref{perv1} follows.  
Let $k$ be the smallest integer such that $H_{\leq k-2g+
1,{\rm var}}^{k}(\slM)\neq 0.$ We thus have that
\[
H_{\leq k^\prime-2g+
1, {\rm var}}^{k^\prime}(\slM)= 0, \qquad \forall \; k^\prime < k.\]
%This implies that for $k^\prime < k$, $$P_{k^\prime-2g+
%1}\cap H_{var}^{k^\prime}(\slM)= 0,$$ 
By combining this vanishing with the equality   established above, we deduce that
%and this combined with \eqref{perv1} we have 
\beq\label{nograde}
%Gr^P_{k^\prime-2g+2}
H_{k^\prime-2g+2, {\rm var}}^{k^\prime}(\slM)\cong H_{{\rm var}}^{k^\prime}(\slM).\eeq 
The class $\alpha$ is $\Gamma$-invariant, so that cupping with the powers of $\alpha$
respects the $\Gamma$-decomposition. In particular, 
by relative hard Lefschetz, we see that
cupping with the appropriate power of $\alpha$ yields an isomorphism
of graded groups \beq\label{hardl}
%Gr^P_{k^\prime-2g+2}
H_{k^\prime-2g+2, {\rm var}}^{k^\prime}(\slM)\cong 
%Gr^P_{8g-8-k^\prime}
H_{8g-8-k^\prime, {\rm var}}^{10g-10-k^\prime}(\slM).\eeq 

\smallskip
\n
In view of  \eqref{dimension}, we have that $\dim H_{{\rm var}}^{k^\prime}(\slM)=\dim H_{{\rm var}}^{10g-10-k^\prime}(\slM)$ so that,  by \eqref{hardl} 
and by \eqref{nograde},  we have that
$ H_{8g-8-k^\prime, {\rm var}}^{10g-10-k^\prime} (\slM) \cong  H^{10g-10-k^\prime}(\slM)$, and
consequently 
\beq \label{contr} 
%P_{8g-9-k^\prime}\cap
 H_{\leq 8g-9-k^\prime,
{\rm var}}^{10g-10-k^\prime}(\slM)= 0.\eeq 
By our choice of $k$, we have that 
$
%P_{k-2g+1}\cap 
H_{\leq k-2g+
1, {\rm var}}^k(\slM)\neq 0,$
so that there is $l <k$ such that 
\[
H_{l-2g+2, {\rm var}}^{k}(\slM)\neq 0.\]
 
 As above, the 
 relative hard Lefschetz yields 
%Gr^P_{8g-8-k^\prime}
\[ H_{8g-8-k^\prime, {\rm var}}^{10g-10+k-2l}(\slM)\neq 0.\]
 In view of the fact that $k>l>2l-k$,
 this contradicts (\ref{contr}) with the choice of $k^\prime=2l-k$, and  (\ref{perv2}) follows.\end{proof} 
 
 Theorem \ref{detpfvar} determines the perverse Leray filtration
 on the $\Gamma$-variant  part $H_{{\rm var}}^*(\slM)$.
 In the course of the proof we have proved that  for $2d<g$ the Lefschetz map   
\beq \label{lefschetz}   
\cup \alpha^{g-2d} :  H_{{\rm var}}^{4g+2d-5}(\slM) \lorw H_{{\rm var}}^{6g-5-2d}(\slM)
\eeq 
is  an isomorphism.
  
Now we determine the mixed Hodge structure 
on the cohomology of the character variety $\slM_\B$. Notice that the direct sum decomposition 
$H^*(\slM_\B) \simeq H^*(\slM_\B)^{\Gamma}\oplus H_{{\rm var}}^*(\slM_\B)$, being 
associated with the algebraic action of a group, 
is a decomposition into a direct sum of Mixed Hodge structures.

\begin{theorem} 
	\beq\label{weight}0=W_{2k-4g+3}H_{{\rm var}}^k(\slM_\B)\subset W_{2k-4g+4}H_{{\rm var}}^k(\slM_\B)=H_{{\rm var}}^k(\slM_\B)\eeq
and \beq\label{hodge}0=F^{k-2g+2}H_{{\rm var}}^k(\slM_\B)\subset F^{k-2g+1}H_{{\rm var}}^k(\slM_\B)=H_{{\rm var}}^k(\slM_\B)\eeq
\end{theorem}
\begin{proof}

 As the invariant part $H^*(\slM_\B)^\Gamma\cong H^*(\pglM_\B)$ we have
\begin{equation}\label{evar}\begin{split} 
E_{{\rm var}}(\slM_\B;x,y):=\sum_{d,i,j} x^iy^j (-1)^k \dim \left({\rm Gr}^W_{i+j}H_{c,{\rm var}}^k(\slM_\B)\right)_{\comp}^{ij}= E(\slM_\B;x,y)-E(\pglM_\B;x,y)=\\
=(2^{2g}-1)(xy)^{2g-2}\left( \frac{(xy-1)^{2g-2}-(xy+1)^{2g-2}}{2}\right)=\sum_{i=1}^{g-1}(2^{2g}-1)
\bino{2g-2}{2i-1}(xy)^{2g-3+2i}, \end{split}\end{equation}
where $E(\pglM_\B;x,y)$  is given by the right hand side of \cite[(1.1.3)]{hausel-villegas}  and $E(\slM_\B;x,y)$ is given by
\cite[(4.6)]{mereb} with $q=xy$. 

 We first observe that 
$$E_{{\rm var}}(\slM_\B;1/x,1/y)=(xy)^{6g-6}E_{{\rm var}}(\slM_\B;x,y)$$ is palindromic. Consequently by Poincar\'e duality
the corresponding expression on ordinary cohomology $$\sum_{d,i,j} x^iy^j (-1)^k \dim({\rm Gr}^W_{i+j}H_{{\rm var}}^k(\slM_\B)_{\comp})^{ij}=E_{{\rm var}}(\slM_\B;x,y)$$ is thus also given
by \eqref{evar}. 
Now we note that \eqref{evar} only depends on $xy$, i.e. every term is of the form $x^py^p$. Additionally, by \eqref{dimension}, if $H_{{\rm var}}^k(\slM_\B)\neq 0$, then $k$ is odd, and every non-trivial $\dim\left({\rm Gr}_{i+j}^WH^{d}(\slM_\B)_{\comp}\right)^{i,j}$ will contribute with a negative coefficient, thus there is no cancellation and the only non-trivial terms are of the form $({\rm Gr}_{2p}^WH_{{\rm var}}^{d}(\slM_\B)_{\comp})^{p,p}$. It follows 
that the mixed Hodge structure on $H^*_{{\rm var}}(\slM_\B;x,y)$ is of Hodge-Tate type, thus \eqref{hodge} follows from \eqref{weight}.

We now  determine the weights on $H^*_{{\rm var}}(\slM_\B)$. Again as $H^*_{{\rm var}}(\slM_\B)$ is only non-trivial in odd cohomology weights cannot cancel each other as they all contribute with a negative coefficient. The possible weights therefore are ${4g-2}, {4g+2},\dots,{8g-10}$ (twice the degrees in $xy$  of the monomials  appearing in \eqref{evar}), with multiplicities which turn out to be equal to  
$$\dim \,H_{{\rm var}}^{4g-3}(\slM_\B) ,\; \dim \,H_{{\rm var}}^{4g-1}(\slM_\B))
\;,\ldots, \;  \dim \,H_{{\rm var}}^{6g-7}(\slM_\B),$$ respectively (these appear as the coefficients in \eqref{evar}). To conclude we need to show that they will be the weights on the cohomologies $H_{{\rm var}}^{4g-3}(\slM_\B),H_{{\rm var}}^{4g-1}(\slM_\B),\dots, H_{{\rm var}}^{6g-7}(\slM_\B)$ respectively. This follows from
\eqref{lefschetz} by an argument  similar to the proof of 
\eqref{detpfvar}.
\end{proof}
\begin{corollary}
We have that $P=W$ on $H^*_{{\rm var}}(\slM)$ and consequently on $H^*(\slM)$. 
\end{corollary}

\begin{remark} \label{description} Note also that this implies a complete description of the ring $H^*(\slM)$. We already know the ring structure on the invariant part. Now,
$\alpha$ acts on $H^*_{{\rm var}}(\slM)$ as described in (\ref{lefschetz}), $\beta$ and $\psi_i$ act trivially as their weights are such that when multiplied with any class in $H^*_{{\rm var}}(\slM_\B)$ it would provide a class with a degree and weight which does not exist in $H^*_{{\rm var}}(\slM_\B)$. Finally by degree reasons
variant classes multiply to $0$. This implies a complete description of the ring structure on $H^*(\slM)$. 
\end{remark}

\begin{remark} The determination of the perverse filtration
	on $H^*_{{\rm var}}(\slM)$ using the symmetry provided by the group scheme $Prym$ is inspired by Ng\^o's approach in \cite{ngo2, ngo}. More connections to his work
	is discussed in \cite[\S 5.2]{hausel-survey}.\end{remark}

\section{Appendix}
\stepcounter{subsection}

In this appendix we recall a result of M. Thaddeus',  Proposition~\ref{restrunivcls} below, concerning the restriction 
of the generators $\alpha, \psi_i, \beta$ defined in \S \ref{modhiggsbnd} 
to a general fibre of the Hitchin fibration for $\mdsl$. 
In view of Theorem \ref{gdpf}, this  yields an upper bound for their perversity, which is used in Theorem~\ref{eccolo}. 
Since  these results, contained in Thaddeus' Master thesis (\ci{thaddeusmt}) have not been published, we report here 
the original proof.

Let $s \in \basesm^0 $, let  $\pro:=\pro_s: \curv_s \to \curv$
be the corresponding spectral curve covering and let
$\iota:= \iota_s: {\curv}_{s} \lorw {\curv}$  be the involution
exchanging the two sheets of the covering (see  \S \ref{spctrcrv}).
The following easy to prove facts are used in the course of the proof of Proposition~\ref{restrunivcls}:
\begin{itemize}
\item 
The involution $\iota$ induces  $\iota_*$ on $H_1(\curv_s)$, which splits 
into the $\pm 1$-eigenspaces 
$$
H_1(\curv_s)=H_1(\curv_s)^+\oplus H_1(\curv_s)^-.
$$
We denote by $\Pi^{\pm}$ the corresponding projections.

Analogously, 
$$
H^1(\curv_s)=H^{1}(\curv_s)^+\oplus H^{1}(\curv_s)^-,
$$
and, via Poincar\'e duality, 
$H^{1}(\curv_s)^+\simeq (H_1(\curv_s)^+)^{\vee}, \, \, 
 H^{1}(\curv_s)^- \simeq (H_1(\curv_s)^-)^{\vee}.$
The projections are still denoted by $\Pi^{\pm}$.

\item
We let $[\curv] \in H_2(\curv)\simeq H^0(\curv), \, 
\,[\curv_s] \in H_2(\curv_s)\simeq H^0(\curv_s)$ be the fundamental classes,
and 
$[c] \in H_0(\curv)\simeq H^2(\curv), \, \,[c_s] \in H_0(\curv_s)\simeq H^2(\curv_s)$ 
the classes of a point in $\curv$, resp. $\curv_s$.
 
We have 
\begin{equation}
\label{pst}
\pi_*([c_s])=[c], \, \,  
\pi_*([\curv_s])=2[\curv],\, \,  \ke \, \pi_*=H^1(\curv_s)^-,
\end{equation}
the last equality  due to the fact that $\pi \circ \iota=\pi$ and $\pi_*$ is surjective. 

\item
In terms of the identification $H^1(\curv_s)^+ \simeq H^1(\curv)$
given by the pull-back map  $\pi^*$, the map 
\begin{equation}
\label{pstone1} 
\pi_*: H^1(\curv_s) \lorw H^1(\curv) \simeq H^1(\curv_s)^+ 
\end{equation}
is identified with the projection $\Pi^+$.

\end{itemize}

\begin{remark}
\label{restrpicpry}
{\rm 
We have isomorphisms
$\hitmap^{-1}(s) \simeq \jac{\curv_s}^0$, 
$\hitmapc^{-1}(s) \simeq \prym_{\curv_s}$ and 
$\hitmaph^{-1}(s) \simeq \prym_{\curv_s}/\Gamma$, 
(see Theorem \ref{fibres} for $\hitmap$ and 
\S\ref{comp3gra} for $\hitmapc$ and $\hitmaph$). 
We have the canonical isomorphisms 
$$ 
H^1(\jac{\curv_s}^0)\simeq H_1(\curv_s), \qquad
H^1(\prym_{\curv_s})\simeq H_1(\curv_s)^{-},
$$
in terms of which the restriction map 
$H^1(\jac{\curv_s}^0) \lorw  H^1(\prym_{\curv_s})$
is identified with the projection
$\Pi^-:H_1(\curv_s) \lorw  H_1(\curv_s)^{-}$.
}
\end{remark}

\begin{proposition}
\label{restrunivcls}
$\,$

\begin{enumerate}
\item
The restrictions of the classes $\psi_i, \beta$ to a general 
fibre of the Hitchin maps  $\hitmap, \hitmaph,$ and $\hitmapc$ vanish. 
\item
The restriction of the class $\alpha$ to a general
fibre of the Hitchin maps  $\hitmap$ is non-zero. 
When restricted to a general fibre
of $\hitmaph$ and $\hitmapc$, the class $\alpha$ is an ample class. 
\end{enumerate}
\end{proposition}
\n{\em Proof.} 
For notational simplicity,  
we denote the fundamental classes 
$[\curv] \in H^0(\curv), \, [\curv_s] \in H^0(\curv_s)$,  
$[\jac{\curv_s}^0] \in H^{0}(\jac{\curv_s}^0)$ and 
$[\prym_{\curv_s}] \in H^{0}(\prym_{\curv_s})$
by 1, so that, for example, 
the second equality in (\ref{pst}) above reads $\pi_*(1)=2$.

Let $s \in \baseosm$. As discussed in \S \ref{hitchmap}, 
the map $\hitmapc$ is the restriction of $\hitmap$ to $\mdsl$, 
and $\hitmaph$ is derived from $\hitmapc$ by passing to the quotient $\mdpgl=\mdsl/\Gamma$. 
It is therefore enough to prove the first statement for the map $\hitmap$
We  denote the product map 
$\pro \times \id: \curv_s \times \jac{\curv_s}^0 \lorw \curv \times \jac{\curv_s}^0$
simply by  $\pro$.

Let  $\mathcal L$ be the Poincar\'e line bundle on $\curv_s \times \jac{\curv_s}^0$. 
The restriction of  ${\mathbb E}$ to $\curv \times \hitmap^{-1}(s)\simeq \curv \times \jac{\curv_s}^0$
is isomorphic to $\pro_*{\mathcal L}$ (see \S \ref{spctrcrv}). The Grothendieck-Riemann-Roch theorem 
(see \ci{fultonit}, Theorem 15.2) gives the following equality in $H^*(\curv \times \jac{\curv_s}^0)$: 
\begin{equation}
\label{grrtm}
{\rm ch}({\mathbb E}_{|\curv \times \hitmap^{-1}(s)}){\rm td}(\curv \times \jac{\curv_s}^0)=
{\rm ch}(\pro_*{\mathcal L }){\rm td}(\curv \times \jac{\curv_s}^0)=
\pro_*( {\rm ch}({\mathcal L}){\rm td}(\curv_s \times \jac{\curv_s}^0)).
\end{equation}

\medskip
\n By the K\"unneth formula   and Poincar\'e duality
$$
H^2(\curv_s  \times \jac{\curv_s}^0)
\simeq \left( H^0(\curv_s)\otimes H^2(\jac{\curv_s}^0)\right) 
\oplus \left(H^1(\curv_s)\otimes H_1(\curv_s) \right)\oplus
\left(H^2(\curv_s)\otimes H^0(\jac{\curv_s}^0)\right), 
$$
and we have the natural isomorphism
\begin{equation}
\label{scndchm} 
{\rm End} \, H^1(\curv_s) \simeq H^1(\curv_s)\otimes H_1(\curv_s) \simeq
H^1(\curv_s)\otimes H^1(\jac{\curv_s}^0 )  
\subseteq H^2(\curv_s  \times \jac{\curv_s}^0).
\end{equation}
We say that a class in  $H^*(\curv \times \jac{\curv_s}^0)$  
has type $(a,b)$ if it is in the
K\"unneth summand $H^a(\curv)\otimes H^b(\jac{\curv_s}^0)$, and similarly for 
classes in $H^*(\curv_s \times \jac{\curv_s}^0)$. 
Clearly the cup product of a class of type $(a,b)$ with one of type $(c,d)$
is of type $(a+c,b+d)$, in particular it vanishes if $a+c>2$. 

We set, for simplicity, $g':=g(\curv_s)$.
We fix a  symplectic basis $\delta_1, \ldots, \delta_{2g'}$ for $H^1(\curv_s)$, and 
we identify its dual basis, $\delta^{\vee}_1, \ldots, \delta^{\vee}_{2g'}$ with a basis for 
$H^1(\jac{\curv_s}^0)$.
In terms of the isomorphism (\ref{scndchm}) 
above, the first Chern class of ${\mathcal L}$  is represented by 
the identity in 
${\rm End} \,H^1(\curv_s)$ (see \ci{acgh}, VIII.2), namely
$c_1({\mathcal L})=\sum_i \delta_i \otimes \delta^{\vee}_i,$
hence  it is of type (1,1). 
Let $[c_s] \in H^2(\curv_s), \, [c] \in H^2(\curv)$ be as defined above.

\n A direct computation gives
$
c_1^2({\mathcal L})=2[c_s]  \otimes (\sum_i \delta^{\vee}_i \wedge \delta^{\vee}_{i+g'}) =2[c_s] \otimes \theta ,
$
where $\theta:= \sum_i \delta^{\vee}_i \wedge \delta^{\vee}_{i+g'} \in H^2(\jac{\curv_s}^0)$ denotes the cohomology class of the 
theta divisor on $\jac{\curv_s}^0$. 

\medskip
\n Since $c_1({\mathcal L})$ is of type $(1,1)$, we have
$c_1^r({\mathcal L})=0 \hbox{ if }r\geq 3$, hence
$$
{\rm ch}({\mathcal L})=1+c_1({\mathcal L})+\frac{c_1^2({\mathcal L})}{2}=1+
\sum_i \delta_i \otimes \delta^{\vee}_i + [c_s]  \otimes \theta . 
$$
\medskip
\n
Since the tangent bundle of a torus is trivial,
we have  that ${\rm td}(\jac{\curv_s}^0)= 1$, while  
${\rm td}(\curv_s)=1+(1-g')[c_s]$ and 
${\rm td}(\curv)=1+(1-g)[c]$. 
By  the multiplicativity properties
of the Todd class, we have that 
$$
{\rm td}(\curv_s \times \jac{\curv_s}^0)=
1 \otimes 1 +(1-g')[c_s]\otimes 1, \qquad
{\rm td}(\curv \times \jac{\curv_s}^0)=
1\otimes 1+(1-g)[c] \otimes 1,
$$
so that 
$$
{\rm ch}({\mathcal L}){\rm td}(\curv_s \times \jac{\curv_s}^0)=
1\otimes 1+ \sum_i \delta_i \otimes \delta^{\vee}_i + [c_s] \otimes \theta +   (1-g')[c_s] \otimes 1.
$$
Plugging this 
into the Grothendieck-Riemann-Roch theorem Formula (\ref{grrtm}), we get:
\begin{equation}
\label{eq1}
{\rm ch}({\mathbb E}_{|\curv \times \hitmap^{-1}(s) }) (1+(1-g)[c])=
\pro_*\left( 1\otimes 1 +\sum_i \delta_i \otimes \delta^{\vee}_i + [c_s] \otimes \theta +   (1-g' )[c_s] \otimes 1 \right).
\end{equation}
Applying  (\ref{pst}) gives $\pro_*(1 \otimes 1)=2\otimes 1,$ 
$\pro_*( [c_s] \otimes \theta )=[c] \otimes \theta$ and   $\pro_*([c_s] \otimes 1)=[c]\otimes 1$. 
Combining   the third equality of (\ref{pst}) with the isomorphism 
$H^1(\curv) \simeq H^1(\curv_s)^+$,
we get
$\pro_*(\sum_i \delta_i \otimes \delta^{\vee}_i)=\sum_i \Pi^+(\delta_i) \otimes 
\delta^{\vee}_i.$

\n
By plugging the above equalities in Equation (\ref{eq1}), and equating the components of degree 2, we deduce that
\begin{equation}
\label{grr1res}
c_1({\mathbb E}_{|\curv \times \hitmap^{-1}(s) })+2(1-g)[c]\otimes 1= 
\sum_i \Pi^+(\delta_i) \otimes \delta^{\vee}_i + (1-g')[c] \otimes 1. 
\end{equation}
Since the product of $[c]\otimes 1$ with a class not of type $(0,2)$ vanish, we obtain
\begin{equation}
\label{grr2res}
c^2_1({\mathbb E}_{|\curv \times \hitmap^{-1}(s) })= 
\left(\sum_i \Pi^+(\delta_i) \otimes \delta^{\vee}_i\right)^2
\end{equation}
which has type $(2,2)$. 
Equating the components of degree 4 in Equation (\ref{eq1}), 
and using  the fact that, by  type consideration, the product
$c_1({\mathbb E}_{|\curv \times \hitmap^{-1}(s) })([c]\otimes 1 )=0$, 
we have 
\begin{equation}
\label{grr3res}
\frac{1}{2}\left(c_1^2({\mathbb E}_{|\curv \times \hitmap^{-1}(s) })-2c_2({\mathbb E}_{|\curv \times \hitmap^{-1}(s) })\right)= [c] \otimes \theta,
\end{equation}
from which we deduce that also $c_2({\mathbb E}_{|\curv \times \hitmap^{-1}(s) })$
has type $(2,2)$. From the equality, true in general for rank two vector bundles,
$$
c_2({\rm End}\, {\mathbb E}_{|\curv \times \hitmap^{-1}(s) } )=
4c_2({\mathbb E}_{|\curv \times \hitmap^{-1}(s) }) -
c_1^2({\mathbb E}_{|\curv \times \hitmap^{-1}(s)}),
$$ it follows that  
$c_2({\rm End}\, {\mathbb E}_{|\curv \times \hitmap^{-1}(s) } )$
has type $(2,2)$, which, by the very definition (see the defining Equation \ref{defabps} in \S\ref{modhiggsbnd}) of $\beta$ and $\psi_i$, 
means that these classes vanish on $\hitmap^{-1}(s)$.

\bigskip
\n
The proof of the second statement is immediately reduced to the case of $\hitmapc$.  By Remark \ref{restrpicpry}, the  restriction of  
the first Chern class of ${\mathcal L}$ 
to $\curv_s \times \prym_{\curv_s}$ 
is $\sum_i \delta_i \otimes \Pi^-(\delta^{\vee}_i)$,
and the Grothendieck-Riemann-Roch  theorem now reads:
\begin{equation}
{\rm ch}({\mathbb E}_{|\curv \times \hitmapc^{-1}(s) }) (1+(1-g)[c])=
\pro_*\left( 1\otimes 1 +\sum_i \delta_i \otimes \Pi^-(\delta^{\vee}_i) + [c_s] \otimes \theta +   (1-g')[c_s] \otimes 1 \right),
\end{equation}
where $\theta$ denotes now the cohomology class of the restriction of the theta divisor to $\prym_{\curv_s}$.
As in the first part of the proof, we apply Formul\ae\  \ref{pst}, giving
$\pro_*(1 \otimes 1)=2\otimes 1,$
$\pro_*( [c_s] \otimes \theta )=[c] \otimes \theta$, $\pro_*([c_s] \otimes 1)=[c]\otimes 1$, and 
$\pro_*(\sum_i \delta_i \otimes \delta^{\vee}_i)=\sum_i \Pi^+(\delta_i) \otimes 
\Pi^-(\delta^{\vee}_i)=0.$

\n
By plugging the above equalities in equation (\ref{eq1}), we deduce that
\begin{equation}
\label{grr1res1}
c_1({\mathbb E}_{|\curv \times \hitmapc^{-1}(s) })+2(1-g)[c]=  
(1-g')[c], 
\end{equation}
so that
$c^2_1({\mathbb E}_{|\curv \times \hitmapc^{-1}(s) })= 0$ and
$\left(c_1^2({\mathbb E}_{|\curv \times \hitmapc^{-1}(s) })-2c_2({\mathbb E}_{|\curv \times \hitmapc^{-1}(s) })\right)= 2[c] \otimes \theta.$

\smallskip
\n
Hence, $c_2({\mathbb E}_{|\curv \times \hitmapc^{-1}(s) }))= -[c] \otimes \theta,$ and finally,
\begin{equation}
\label{grr4res1}
c_2({\rm End} \,{\mathbb E}_{|\curv \times \hitmapc^{-1}(s) } )=4c_2({\mathbb E}_{|\curv \times \hitmapc^{-1}(s) })
- 4c_1^2({\mathbb E}_{|\curv \times \hitmapc^{-1}(s) })=
-4[c] \otimes \theta,
\end{equation}
which shows (cfr. (\ref{defabps}) in \S \ref{modhiggsbnd})
that the restriction of $\alpha$ to $\hitmapc^{-1}(s)$
equals $4 \theta$.
\blacksquare

\begin{corollary}
\label{alphample}
The class $\alpha$ is  ample on $\mdpgl$.
\end{corollary}
\n{\em Proof.} The variety $\mdpgl=\higgsbu^0/\Jac^0_C$ is quasiprojective as $\higgsbu$ is by \cite[Proposition 7.4]{nitsure}. Additionally it follows from Theorem~\ref{hauthaddmntm} that $\dim H^2(\mdpgl)=1$. As we proved above that $\alpha$ is ample on the generic fiber of $\hitmaph$ it must be ample on $\mdpgl$ as well. \blacksquare

\end{document}